\documentclass[aop]{imsart}

\RequirePackage{amsthm,amsmath,amsfonts,amssymb, verbatim, subcaption}
\RequirePackage[numbers]{natbib}
\RequirePackage[colorlinks,citecolor=blue,urlcolor=blue]{hyperref}
\RequirePackage{graphicx}

\startlocaldefs
\newtheorem{theorem}{Theorem}
\newtheorem{lemma}[theorem]{Lemma}

\newcommand{\E}{\ensuremath{\mathbb E}}
\newcommand{\R}{\ensuremath{\mathbb R}}

\newcommand{\F}{\ensuremath{\psi}}
\newcommand{\tr}{\ensuremath{{\scriptscriptstyle\mathsf{T}}}}

\newcommand{\lab}{\label}  \newcommand{\ra}{\ensuremath{\rightarrow}}  \def\a{{\mathbf{\alpha}}} \def\de{{\mathbf{\delta}}} \def\De{{{\Delta}}}  
  \def\beq{\begin{eqnarray}} \def\eeq{\end{eqnarray}} \def\ben{\begin{enumerate}}
\def\een{\end{enumerate}} \def\bit{\begin{itemize}}
\def\bel{\begin{lemma}}
\def\eel{\end{lemma}}
\def\eit{\end{itemize}} \def\beqs{\begin{eqnarray*}} \def\eeqs{\end{eqnarray*}} \def\bel{\begin{lemma}} \def\eel{\end{lemma}}
\newcommand{\N}{\mathbb{N}} \newcommand{\Z}{\mathbb{Z}}  \newcommand{\C}{\mathcal{C}} 
\newcommand{\T}{\mathbb{T}}      \renewcommand{\b}{\mathbf{b}} 
\newcommand{\tb}{\tilde{\mathbf{b}}}
 \newcommand{\I}{I}   \newcommand{\p}{\mathbb{P}}
\newcommand{\PP}{\mathcal P}    \newcommand{\one}{\mathrm{1}}
 \newcommand{\MM}{\mathcal M} \newcommand{\la}{\lambda}  
   \def\eps{{\epsilon}}  \def\ie{i.\,e.\,} \def\g{G}
\def\vol{\mathrm{vol}}
\newcommand{\tP}{\tilde{P}}

\newcommand{\tg}{\hat{G}}

\newcommand{\La}{\Lambda}

\newcommand{\RR}{\mathbb{R}}

\newcommand{\J}{\mathcal{J}}

\renewcommand{\g}{\mathcal{G}}

\renewcommand{\C}{\mathcal C}
\newcommand{\f}{\mathbf{f}}
\newcommand{\ta}{\tilde{a}}

\newcommand{\hess}{\nabla^2}
\theoremstyle{plain}
\newtheorem{thm}{Theorem}
\newtheorem{corollary}{Corollary}
\newtheorem{observation}{Observation}
\newtheorem{defn}{Definition}
\newtheorem{claim}{Claim}

\theoremstyle{remark}

\newtheorem{notation}{Notation}

\numberwithin{equation}{section}
\numberwithin{figure}{section}
\renewcommand{\a}{\alpha}
\renewcommand{\b}{\beta}
\renewcommand{\g}{\gamma}

\newcommand{\spec}{\mathrm{spec}}
\newcommand{\Spec}{\mathrm{Spec}}
\newcommand{\Leb}{\mathrm{Leb}}
\renewcommand{\ta}{\tilde{\alpha}}
\renewcommand{\tb}{\tilde{\beta}}
\renewcommand{\tg}{\tilde{\gamma}}
\newcommand{\tla}{\tilde{\lambda}}
\newcommand{\tmu}{\tilde{\mu}}
\newcommand{\tnu}{\tilde{\nu}}

\renewcommand{\F}{\mathcal F}
\renewcommand{\RR}{\mathcal R}
\newcommand{\tX}{\tilde{X}}
\newcommand{\tY}{\tilde{Y}}
\newcommand{\D}{\mathcal{D}}
\newcommand{\bb}{\mathbf b}

\newcommand{\tilh}{\tilde{h}}
\renewcommand{\one}{{\bf{1}}}
\date{\today}
\endlocaldefs
\begin{document}

\begin{frontmatter}
\title{Large deviations for random hives and the spectrum of the sum of two random matrices}

%\title{A sample article title with some additional note\thanksref{t1}}
\runtitle{Large deviations for random hives}
%\thankstext{T1}{A sample additional note to the title.}

\begin{aug}
%%%%%%%%%%%%%%%%%%%%%%%%%%%%%%%%%%%%%%%%%%%%%%%
%% Only one address is permitted per author. %%
%% Only division, organization and e-mail is %%
%% included in the address.                  %%
%% Additional information can be included in %%
%% the Acknowledgments section if necessary. %%
%%%%%%%%%%%%%%%%%%%%%%%%%%%%%%%%%%%%%%%%%%%%%%%
\author[A]{\fnms{Hariharan} \snm{Narayanan}\ead[label=e1]{hariharan.narayanan@tifr.res.in}},
\and \author[B]{\fnms{Scott} \snm{Sheffield}\ead[label=e2]{sheffield@math.mit.edu}}

%%%%%%%%%%%%%%%%%%%%%%%%%%%%%%%%%%%%%%%%%%%%%%
%% Addresses                                %%
%%%%%%%%%%%%%%%%%%%%%%%%%%%%%%%%%%%%%%%%%%%%%%
\address[A]{School of Technology and Computer Science, Tata Institute of Fundamental Research, Mumbai,
\printead{e1}}

\address[B]{Department of Mathematics, 
Massachusetts Institute of Technology,\\
\printead{e2}}
\end{aug}

\begin{abstract}

Suppose $\alpha, \beta$ are Lipschitz, strongly concave functions from $[0, 1]$ to $\R$ and $\g$ is a concave function from $[0, 1]$ to $\R$, such that $\a(0) = \g(0) = 0$, and  $\a(1) = \b(0) = 0$ and $\b(1) = \g(1) = 0.$  
For an $n \times n$ Hermitian matrix $W$, let $\spec(W)$ denote the vector in $\R^n$ whose coordinates are the eigenvalues of $W$ listed in non-increasing order. Let $\la = \partial^- \a$, $\mu = \partial^- \b$ on $(0, 1]$ and $\nu = \partial^- \g,$ at all points of $(0, 1]$, where $\partial^-$ is the left derivative.
Let $\la_n(i) := n^2(\a(\frac{i}{n})-\a(\frac{i-1}{n}))$, for $i \in [n]$, and similarly,
$\mu_n(i) := n^2(\b(\frac{i}{n})-\b(\frac{i-1}{n}))$, and 
$\nu_n(i) := n^2(\g(\frac{i}{n})-\g(\frac{i-1}{n}))$. 

Let $X_n, Y_n$ be independent random Hermitian matrices from unitarily invariant distributions with spectra $\la_n$,
$\mu_n$  respectively. We define norm $\|\cdot\|_\I$ to correspond in a certain way to the sup norm of an antiderivative. We prove that the following limit exists.
\beqs \lim\limits_{n \ra \infty}\frac{\log \p\left[\|\spec(X_n + Y_n) - \nu_n\|_\I < n^2 \eps\right]}{n^2}.\eeqs
We interpret this limit in terms of the surface tension $\sigma$ of continuum limits of the discrete hives defined by Knutson and Tao. 

We provide matching large deviation upper and lower  bounds for the spectrum of the sum of two random matrices $X_n$ and $Y_n$, in terms of the surface tension $\sigma$ mentioned above.

We also prove large deviation principles for random hives with $\a$ and $\b$ that are $C^2$, where the rate function can be interpreted in terms of the maximizer  of a functional that is the sum of a term related to the free energy of hives associated with $\a, \b$ and $\g$ and a quantity related to logarithms of Vandermonde determinants associated with $\g$.
\end{abstract}



\end{frontmatter}
\newpage
\tableofcontents

\section{Introduction}

\subsection{Motivation from Representation Theory and Random Matrices}
Littlewood-Richardson coefficients play an important role in the representation theory of  the general linear groups. Among other interpretations, they count the number of tilings of certain domains using squares and equilateral triangles  \cite{squaretri}. 
Let $\la, \mu, \nu$ be vectors in $\Z^n$ whose entries are non-increasing non-negative integers. Let the $\ell_1$ norm of a vector $\a \in \R^n$ be denoted $|\a|$ and let $$|\la| + |\mu| = |\nu|.$$ Take an equilateral triangle $\Delta$ of side $1$. Tessellate it with equilateral triangles of side $1/n$. Assign boundary values to $\Delta$ as in Figure~\ref{fig:tri3}; Clockwise, assign the values $0, \la_1, \la_1 + \la_2, \dots, |\la|, |\la| + \mu_1, \dots, |\la| + |\mu|.$ Then anticlockwise, on the horizontal side, assign  $$0, \nu_1, \nu_1 + \nu_2, \dots, |\nu|.$$

Knutson and Tao defined this hive model for Littlewood-Richardson coefficients in \cite{KT1}. They showed that the Littlewood-Richardson coefficient
$c_{\la\mu}^\nu$ is given by the number of ways of assigning integer values to the interior nodes of the triangle, such that the piecewise linear extension to the interior of $\Delta$ is a concave function $f$ from $\Delta$ to $\R$.  
Such  an integral ``hive" $f$ can be described as an integer point in a certain polytope known as a hive polytope. The volumes of these polytopes shed light on the asymptotics of Littlewood-Richardson coefficients \cite{Ikenmeyer, Nar, Okounkov, Greta}. Additionally, 
they appear in certain calculations  in free probability \cite{KT2, Zuber}. Indeed,
the volume of the polytope of all real hives with fixed boundaries $\la, \mu, \nu$ is equal, up to known multiplicative factors involving Vandermonde determinants, to the probability density of obtaining a Hermitian matrix with spectrum $\nu$ when two independent Haar random Hermitian matrices with spectra  $\la$ and $\mu$   are added \cite{KT2}. The  relation was stated in an explicit form by Coquereaux and Zuber in \cite{Zuber} and has been reproduced here in (\ref{eq:2.4new}). 

We refer the reader to \cite{Zuberhorn} for background on hive polytopes and the Horn problem.
\begin{comment}
\begin{figure}\label{fig:tri}
\begin{center}
\includegraphics[scale=0.25]{triangle.pdf}
\caption{Hive model for Littlewood-Richardson coefficients}
\end{center}
\end{figure}
\end{comment}

\begin{figure}
\begin{center}
\includegraphics[scale=0.40]{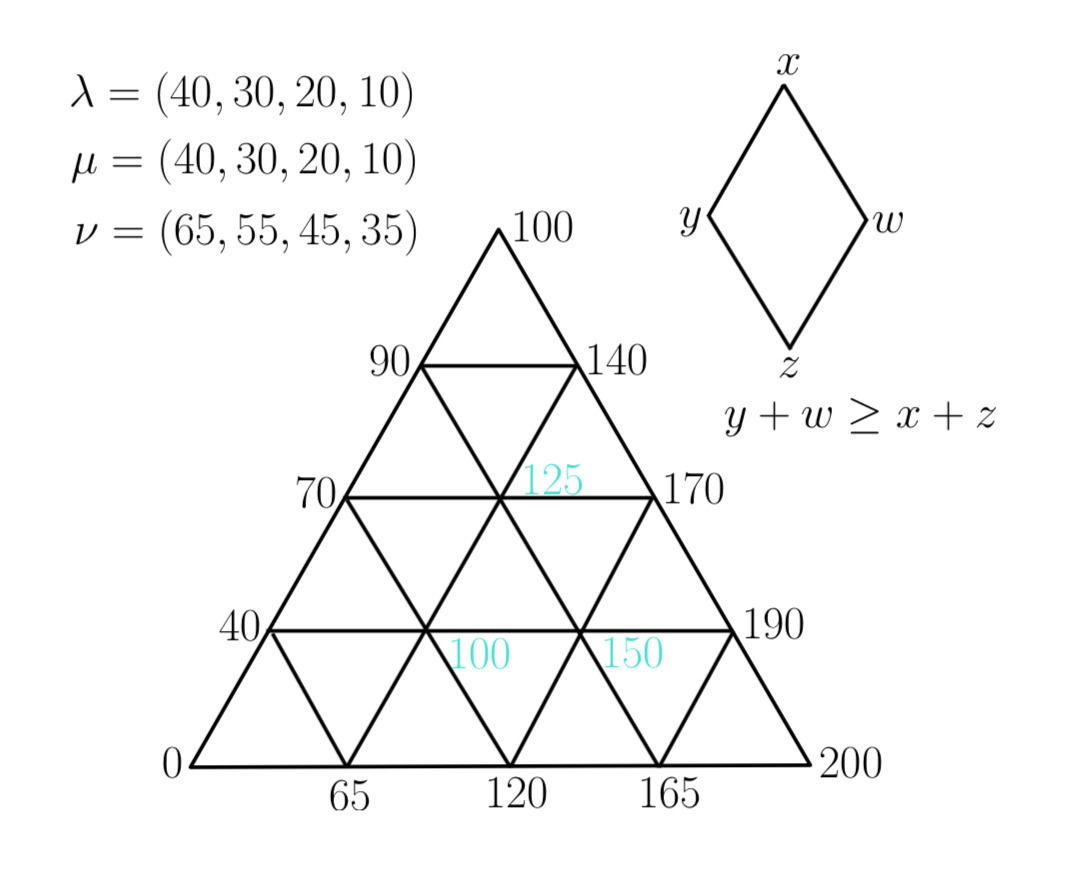}
\caption{Values taken at interior vertices in the hive model}\label{fig:tri3}
\end{center}
\end{figure}

\begin{comment}
Corresponding to every real hive, is a gadget known as a honeycomb, which is a hexagonal tiling. The positions of the lines corresponding to the semi-infinite rays are fixed by the boundary data $\la, \mu$ and $\nu$, with each segment being parallel to one of the sides of a regular hexagon. One obtains a random honeycomb from a random hive by mapping the gradient of the hive on each of the unit equilateral triangles to a point in $\R^2$. This point becomes a vertex of the honeycomb.
\begin{figure}
\begin{center}
\includegraphics[scale=0.80]{honeycomb3.png}
\caption{A honeycomb, from Knutson and Tao \cite{KT2}}\label{fig:honey}
\end{center}
\end{figure}
\end{comment}

In this paper, we will primarily be studying limits of hives of increasing size whose boundary conditions remain roughly the same, beyond the necessary modifications entailed by the increasing size.  For this reason, we will now introduce new notation.

Throughout this paper, we suppose $\alpha, \beta$ are Lipschitz strongly concave functions\footnote{We assume strong concavity, \ie that there exists a positive $\de$ such that $\a + \de x^2 $ and $\beta + \de x^2$ are concave. Without some such assumption, the normalized asymptotic volumes $V(\la)V(\mu)$ of certain polytopes (corresponding to the set of ``augmented hives" that will be defined in Definition~\ref{def:defn6}) can go to $0$ leading to degeneracies.} from $[0, 1]$ to $\R$ and $\g$ is a concave function from $[0, 1]$ to $\R$, such that $\a(0) = \g(0) = 0$, and  $\a(1) = \b(0) = 0$ and $\b(1) = \g(1) = 0.$ 
 \begin{defn}
 Let $U$ be a subset of $\R$.  We say that a function $f:U \ra \R$ is strongly decreasing if there exists $\de > 0$ such that $f(x) + \de x$ is monotonically decreasing.
 \end{defn}
For a $n \times n$ Hermitian matrix $W$, let $\spec(W)$ denote the vector in $\R^n$ whose coordinates are the eigenvalues of $W$ listed in non-increasing order. Let $\la = \partial^- \a$, $\mu = \partial^- \b$ on $(0, 1]$ and $\nu = \partial^- \g,$ at all points of $(0, 1]$, where $\partial^-$ is the left derivative. We note that $\la$ and $\mu$ are strongly decreasing and $\nu$ is monotonically decreasing.
Let $\la_n(i) := n^2(\a(\frac{i}{n})-\a(\frac{i-1}{n}))$, for $i \in [n]$, and similarly,
$\mu_n(i) := n^2(\b(\frac{i}{n})-\b(\frac{i-1}{n}))$, and 
$\nu_n(i) := n^2(\g(\frac{i}{n})-\g(\frac{i-1}{n}))$. 

We define the seminorm $\|\cdot\|_\I$ on $\R^n$ as follows.
Given  $\la_n(1) \geq \dots \geq \la_n(n)$, such that $\sum_i \la_n(i) = 0$, let $\|\la_n\|_\I$ be defined to be the supremum over all $1 \leq i \leq n$ of $\la_n(1) + \dots + \la_n(i)$. Let $\one$ denote the vector in $\R^n$ whose coordinates are all $1$. Given any vector $v \in \R^n$ such that $v - (n^{-1}) \one \langle v, \one\rangle$ has coordinates that are a permutation of $\la_n(1), \dots, \la_n(n)$, we define $\|v\|_\I$ to be the same as $\|\la_n\|_\I,$ and $\|\cdot\|_\I$ is readily seen to be a seminorm.
We use $\|\cdot\|_\I$ to also denote a seminorm on the space of functions of bounded variation (\ie expressible as the difference of two monotonically decreasing functions) 
where the value on a given function $f:[0,1]\ra\R$  equals $\sup_{n \in \N} n^{-2}\|f_n\|_\I$. Note that 
if $f$ is monotonically decreasing, then $\|f\|_\I = \lim_{n \ra \infty}n^{-2}\|f_n\|_\I$. Let $B_I(\nu, \eps):= \{\nu'\in \R^n| \|\nu' -\nu\|_I < \eps\}.$ 

Let $X_n, Y_n$ be independent random Hermitian matrices from unitarily invariant distributions with spectra $\la_n$,
$\mu_n$  respectively. 
We wish to prove that the following limit exists.
\beq \lim\limits_{n \ra \infty}\frac{\ln \p\left[n^{-2}\|\spec(X_n + Y_n) - \nu_n\|_\I < \eps\right]}{n^2}.\lab{eq:master}\eeq
We will interpret this limit in terms of the entropy of certain scaling limits of hives introduced by Knutson and Tao \cite{KT1, KT2}. In particular, 
if there exists concave surface satisfying certain conditions, on an equilateral triangle in $\R^2$, with boundary conditions derived from $\a, \b, \g$, we will show that the expression in (\ref{eq:master}) can be interpreted in terms of the maximizer  of a functional that is the sum of a term related to the free energy of hives associated with $\a, \b$ and $\g$ and a quantity related to logarithms of Vandermonde determinants associated with $\g$.

More precisely,  for a concave function $\g':[0, 1] \ra \R$ such that $\nu' = \partial^-\gamma'$,  in Section~\ref{sec:8},  we define the rate function \beqs I(\g') :=  \log\left(\frac{V(\la)V(\mu)}{V(\nu')}\right) + \inf\limits_{h' \in H(\la, \mu; \nu')}  \int_T 2\sigma((-1)(\hess h')_{ac})\Leb_2(dx) .\eeqs
In the above expression,  $V$ is defined in Definition~\ref{def:32} and $\sigma$ in Definition~\ref{def:sigma}.
The following is proved in Theorem~\ref{thm:4}:\\
{\bf Theorem}: Let $X_n, Y_n$ be independent random Hermitian matrices from unitarily invariant distributions with spectra $\la_n$,
$\mu_n$  respectively. 
Let $Z_n = X_n + Y_n$.
For any $\eps > 0$,
$$\lim\limits_{n \ra \infty}\left(\frac{-2}{n^2}\right)\log \p_n\left[\spec(Z_n) \in {B}_\I^n(\nu_n, \eps)\right] = \inf\limits_{\partial^-\gamma'\in {B}_\I(\nu, \eps)}I(\g').$$

The question of understanding the spectrum of $X_n + Y_n,$ in the setting of large deviations has been studied before in \cite{Belinschi}, where upper and lower bounds were given which agreed for measures that corresponded to free products with  amalgamation (see Theorem 1.3, \cite{Belinschi}), but not in general.  
In this paper, in the case where $\la$ and $\mu$ have integrals that are strongly concave, we provide matching upper and lower bounds in terms of a certain surface tension $\sigma$ in Theorem~\ref{thm:5}. This theorem states the following.\\
Let $I(\g)$ be defined by (\ref{eq:I}) (also stated above).

{\bf Theorem}: Let $X_n, Y_n$ be independent random Hermitian matrices from unitarily invariant distributions with spectra $\la_n$,
$\mu_n$  respectively. 
Let $Z_n = X_n + Y_n$.
Given a monotonically decreasing sequence of $n$ numbers $v_n$ whose sum is $0$, we define $\iota_n(v_n)$ to be the piecewise linear concave function $f$ from $[0, 1]$ to $\R$ such
$f(i/n) - f((i-1)/n) = v_n(i)/n^2$ for each $i \in [n]$ , and $f(0) = 0.$
For any Borel set $E\subset L^\infty[0, 1]$, let $\PP_n(E) := \p_n[\iota_n(\spec( Z_n)) \in E].$
Let $a_n := \frac{n^2}{2}.$
For each Borel measurable set 
${\displaystyle E\subset L^\infty[0, 1],}$ 
$${\displaystyle -\inf_{\g\in E^{\circ }}I(\g)\leq \liminf_{n \ra \infty}a_{n}^{-1}\log(\mathcal {P}_{n}(E))\leq \limsup_{n \ra \infty}a_{n}^{-1}\log(\mathcal {P}_{n}(E))\leq -\inf_{\g\in {\overline {E}}}I(\g).}$$

We also study the related geometric question of the shape of the restriction of a random augmented hive to the upper triangle (see Figure~\ref{fig:aug-hive}) as $n \ra \infty$. 
If $\la$ and $\mu$ are $C^1$ and strongly decreasing, Theorem~\ref{thm:5new} provides a large deviation principle for the shape of 
the restriction of a random augmented hive to the upper triangle $T$. This theorem states the following.\\
For $h \in H(\la, \mu; \nu)$,  let $I_1(h)$ be defined (as in (\ref{eq:I1}))  as follows:
\beqs I_1(h) :=  \log\left(\frac{V(\la)V(\mu)}{V(\nu)}\right) +   \int_T 2\sigma((-1)(\hess h)_{ac})\Leb_2(dx) .\eeqs
{\bf Theorem}: Suppose  $\la, \mu$ are strongly decreasing $C^1$ functions.
Given a discrete hive $h_n \in H_n(\la_n, \mu_n; \nu_n)$, we define $h' = \iota_n(h_n)$ to be the  piecewise linear extension of $h_n$. This is a map $h'$ from $T$ to $\R$ such that for any $(x, y) \in T \cap \left(\Z/n \times \Z/n\right)$,  
$h'(x, y) = h_n(nx, ny)/n^2$.
For any Borel set $E\subset L^\infty(T)$, let $\PP_n(E) := \p_n[\iota_n(h_n) \in E].$
Let $a_n := \frac{n^2}{2}.$
For each Borel measurable set 
${\displaystyle E\subset L^\infty(T),}$ 
$${\displaystyle -\inf_{h\in E^{\circ }}I_1(h)\leq \liminf_{n \ra \infty}a_{n}^{-1}\log(\mathcal {P}_{n}(E))\leq \limsup_{n \ra \infty}a_{n}^{-1}\log(\mathcal {P}_{n}(E))\leq -\inf_{h\in {\overline {E}}}I_1(h).}$$

In order to study the large deviations of random surfaces \cite{Scott}, it has proven beneficial to first examine the situation with periodic boundary conditions as was done by Cohn, Kenyon and Propp in \cite{CohnKenyonPropp} for the case of domino tilings. In our setting, these structures correspond to random hives with periodic Hessians, where the periodicity is at a scale that tends to infinity. The results of \cite{random_concave} provide some insight into concentration phenomena for these objects.

\begin{comment}
In the discrete setting, this hive shape, has a number of remarkable properties and is related to both random matrices and representation theory (see \cite{KT2}).
In \cite{Conant-Tao}, Tao stated in 2005 that ``Another completely open question is how these honeycombs ``converge" in the large dimensional limit (as $n$ goes to infinity), as there should definitely be some connection with free probability and free convolution." Our paper makes progress on this issue by establishing a large deviation principle for continuum hives  (which are in a sense, Legendre duals of large $n$ limits of honeycombs) in Theorem~\ref{thm:5new}.
\end{comment}

\section{Preliminaries}\lab{sec:prelim}
We will denote universal constants by $c$ and $C$,  and their precise value may depend upon the occurrence.
Let $\Spec_n$ denote the cone of all possible vectors $x = (x_1, \dots, x_n)$ in $\R^n$ such that $$x_1 \geq x_2 \geq \dots \geq x_n,$$ and $\sum_i x_i = 0.$ Let $\Leb_{n-1, 0}$ denote the $n-1$ dimensional Lebesgue measure on the linear span of $\Spec_n$.
%See \cite{Belinschi}.

%We consider the triangular lattice $\mathbb L \subseteq \R^2$, generated by $(0, 1)$ and $(1, 0)$ and $(1, 1)$ by integer linear combinations. 
%We define the edges  $E(\mathbb L )$ to be  the lattice parallelograms of side $1$ in $\mathbb L$. 
Let $\Box$ be the square with vertices $(0, 0), (0, 1), (1, 1), (1, 0)$. We consider a $n \times n$ lattice square  $\Box_n$ with vertices $(0, 0)$, $(0, n)$, $(n, 0)$ and $(n, n)$. 
We use $[n]$ to denote $\Z \cap [1, n].$

\begin{defn}\lab{def:Tn}  Let $T$ be the isosceles right triangle in $\R^2$ whose vertices are $(0, 0)$, $(0, 1)$ and $(1, 1)$.  Let $T_n$ be the set of lattice points $nT \cap (\Z^2)$, where $n$ is a positive integer.
\end{defn}

\begin{figure}\label{fig:factorize}
\begin{center}
\includegraphics[scale=0.75]{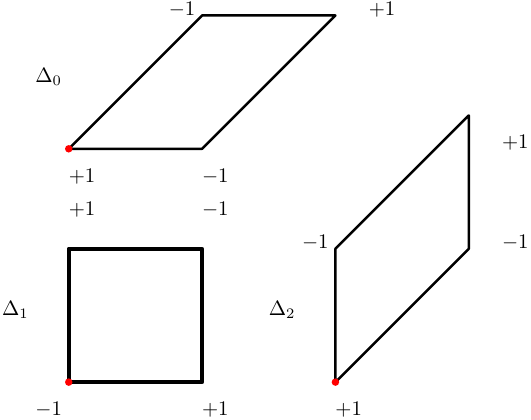}
\caption{A visual depiction of the the second order discrete operators $\De_i$.  A red dot indicates the vertex $(v_1, v_2)$.}
\end{center}
\end{figure}
\begin{comment}
\begin{defn}[Discrete Hessian and the $\De_i$ on $\T_n$]
Let $f: \T_n \ra \R$ be a function.
We define the discrete Hessian $\nabla^2(f)$ to be a  function from the set $E(\T_n)$ of parallelograms of the form $\{a, b, c, d\}$ of side $1$ on the discrete torus to the reals as follows. 

We define the second order difference operators $\De_0$,  $\De_1$ and $\De_2$ from $\R^{V(\T_n)}$ to $\R^{V(\T_n)}$ given by 
\beq\lab{eq:A}
\De_0 f(v_1,  v_2) =  f(v_1, v_2) - f(v_1 + 1, v_2) - f(v_1 + 1, v_2 + 1) + f(v_1 +2, v_2 + 1).\nonumber\\
\De_1 f(v_1,  v_2) =   - f(v_1, v_2) + f(v_1 + 1, v_2) + f(v_1, v_2 + 1) - f(v_1 + 1,  v_2 + 1).\nonumber\\
\De_2 f(v_1,  v_2) =  f(v_1, v_2) - f(v_1+1, v_2+1) - f(v_1, v_2 + 1) + f(v_1 + 1, v_2 + 2).\nonumber\\
\eeq

\end{defn}

For $i = 0, 1$ and $2$,  we define $E_i(\T_n)$ to be the sets of parallelograms corresponding to $\De_0, \De_1$ and $\De_2$ on $\T_n$, respectively.  
\end{comment}
\begin{defn}[Discrete Hessian and the $\De_i$ on $T_n$]
Let $f: T_n \ra \R$ be a function.
\bit
\item Let $E_0(T_n)$ be the set of all parallelograms $e_0\subseteq T_n$ whose vertices are $\{(v_1, v_2),  (v_1 + 1, v_2),  (v_1 + 1, v_2 + 1),  (v_1 +2, v_2 + 1)\}.$

\item Let $E_1(T_n)$ be the set of all parallelograms $e_1\subseteq T_n$ whose vertices are $\{(v_1, v_2),  (v_1 + 1, v_2),  (v_1, v_2 + 1),  (v_1 +1, v_2 + 1)\}.$ 
\item Let $E_2(T_n)$ be the set of all parallelograms $e_2\subseteq T_n$ whose vertices are$\{(v_1, v_2),  (v_1 + 1, v_2+1),  (v_1, v_2 + 1),  (v_1 +1, v_2 + 2)\}.$ 
\eit

We define the discrete Hessian $\nabla^2(f):E(T_n) \ra \R$ to be a  real-valued function on the set $E(T_n) = E_0(T_n) \cup E_1(T_n) \cup E_2(T_n)$ and the $\De_i$ from $\R^{V(T_n)}$ to $\R^{E_i(T_n)}$ by 
\beq\lab{eq:A}
\nabla^2 f(e_0) := \De_0 f(e_0) :=  f(v_1, v_2) - f(v_1 + 1, v_2) - f(v_1 + 1, v_2 + 1) + f(v_1 +2, v_2 + 1).\nonumber\\
\nabla^2f(e_1) := \De_1 f(e_1) :=  - f(v_1, v_2) + f(v_1 + 1, v_2) + f(v_1, v_2 + 1) - f(v_1 + 1,  v_2 + 1).\nonumber\\
\nabla^2f(e_2) := \De_2 f(e_2) :=  f(v_1, v_2) - f(v_1+1, v_2+1) - f(v_1, v_2 + 1) + f(v_1 + 1, v_2 + 2).\nonumber\\
\eeq
\end{defn}

\begin{defn}[Rhombus concavity]\lab{def:rhomb}
Given a function $h:T \ra \R$, and a positive integer $n$,  let $h_n$ denote the function from $T_n$ to $\R$ such that for $(nx, ny) \in T_n$, $h_n(nx, ny) = n^2 h(x, y).$
A function $h:T \ra \R$ is called {\bf rhombus concave} if for any positive integer $n$, 
and any $i$, $\De_i h_n$ is nonpositive on $E_i(T_n),$ and $h$ is continuous on $T$.
The corresponding function $h_n$ is called {\bf discrete (rhombus) concave}.  
Note that a necessary and sufficient condition for a function $h_n$ from $T_n$ to $\R$ to be discrete concave,  is that the piecewise linear extension (which we denote $\tilh_n$) of $h_n$ to $nT$ is rhombus concave.
Here each piece is an isosceles right triangle with a $\sqrt{2}-$length  hypotenuse parallel to the vector $(1, 1).$
\end{defn}
\begin{lemma}\lab{lem:3-Nov20}
Let $h:T \ra \R$ be a rhombus concave function.  Let the restrictions of $h$ to the line segment joining $(0, 0)$ and $(0, 1)$ and to the line segment joining $(0, 1)$ and $(1, 1)$ be $L$-Lipschitz functions.  Then $h$ is $CL$-Lipschitz on $T$.
\end{lemma}
\begin{proof}
Let $z_1$ and $z_2$ be two points in $T$ with  coordinates $(x_1, y_1)$ and $(x_2, y_2)$ respectively, where $n x_1,  n y_1, n x_2,  n y_2$ and $n$ are integers.  We wish to prove that \beq |h(z_1) - h(z_2)| < CL(|x_1 - x_2| + |y_1 - y_2|),\lab{eq:h}\eeq for some universal constant $C$ that is independent of $z_1$ and $z_2.$ There is a piecewise linear path involving at most one vertical and at most one  horizontal segment with endpoints in $\Z^2/n$,   of total length at most $2$ (the $\ell_1-$diameter of $T$),  joining $z_1$ and $z_2.$ Therefore it suffices to prove (\ref{eq:h}) in the two 
special cases when $x_1 = x_2$ and when $y_1 = y_2.$ Suppose first that $x_1 = x_2$ and w.l.o.g $y_1 < y_2.$ By (\ref{eq:A}) applied to parallelograms in $E_1(T_n),$
we see that \beqs h(z_1) - h(z_2) \leq h(0,  y_1)  - h(0,  y_2) < L |y_1 - y_2|.\eeqs Also,  by (\ref{eq:A}) applied to parallelograms in $E_2(T_n),$
we see that \beqs h(z_2) - h(z_1) \leq h(0,  y_2 - x_2)  - h(0,  y_1 - x_1) < L |y_1 - y_2|\lab{eq:h2}.\eeqs Next suppose that $y_1 = y_2.$ W.l.o.g assume that $x_1 < x_2.$
 By (\ref{eq:A}) applied to parallelograms in $E_1(T_n),$
we see that \beqs h(z_2) - h(z_1) \leq h(x_2,  1)  - h(x_1,  1) < L |x_2 - x_1|.\eeqs Also,  by (\ref{eq:A}) applied to parallelograms in $E_0(T_n),$
we see that \beqs h(z_1) - h(z_2) \leq h(x_1 - y_1 + 1,  1)  - h(x_2 - y_2 + 1,  1) < L |x_2 - x_1|.\eeqs Since $h$ is continuous,  the Lipschitz condition for rational points implies the Lipschitz condition for all  $z_1, z_2 \in T.$
\end{proof}
\begin{lemma}\lab{lem:lrc}
Any Lipschitz rhombus concave function from $T$ to $\R$ is concave.
\end{lemma}
\begin{proof}
Let $h$ be a Lipschitz rhombus concave function from $T$ to $\R.$ Let $\tilh_n$ be the piecewise linear extension of $h_n$ to $nT,$ as in Definition~\ref{def:rhomb}.  We observe that the graph of $\tilh_n$ is polyhedral,  and at a point where two or more facets meet,  it is locally concave.  Hence,   $\tilh_n$ is a concave function.  Since $h$ is Lipschitz,  for any $(x, y) \in T,$
$h(x, y) = n^{-2} \lim\limits_{n \ra \infty} \tilh_n(nx, ny).$ Since the limit of Lipschitz concave functions that converge pointwise,  is concave,  we conclude that $h$ is a concave function from $T$ to $\R.$
\end{proof}

Given $e \in E_i(T_n),$ let $conv(e)$ denote the (closed) convex hull of $e.$
\begin{defn}[\cite{Folland}, page 212]
A Radon measure, is a measure on the $\sigma$-algebra of Borel sets of a Hausdorff topological space X that is finite on all compact sets, outer regular on all Borel sets, and inner regular on open sets.
\end{defn}

\begin{lemma}\lab{lem:3.1}
Given a Lipschitz rhombus concave function $h:T \ra \R$,  for $i =0, 1, 2,$ there exists a unique nonpositive Radon  measure $D_i h$ on $T$ such that for any $n$, and any parallelogram $e \in E_i(T_n)$, and the corresponding parallelogram $conv(e)/n \subseteq T,$  $$ n^2 \int\limits_{conv(e)/n} D_i h(dx) =   n^2 \int\limits_{(conv(e)/n)^o} D_i h(dx) = \hess h_n(e).$$ 
\end{lemma}
\begin{proof} By Lemma~\ref{lem:lrc}, $h$ is concave.
The existence and uniqueness of a nonpositive definite matrix valued Radon measure that is the Hessian of a concave function are known by Theorem 3.3 \cite{dudley}. The fact that  $$ n^2 \int\limits_{conv(e)/n} D_i h(dx) =   n^2 \int\limits_{(conv(e)/n)^o} D_i h(dx)$$ follows from the fact that $h$ is Lipschitz and the inner regularity of $D_i h$ on open sets. 
The fact that  $$ n^2 \int\limits_{conv(e)/n} D_i h(dx) \leq \hess h_n(e)   \leq n^2 \int\limits_{(conv(e)/n)^o} D_i h(dx),$$ follows by considering the convolution of $h$ with a $C^\infty$ approximate identity,  letting the support of this approximate identity tend to $\{0\}$ in Hausdorff distance (note that $\hess h_n(e)$ is nonpositive) and using inner and outer regularity.
\end{proof}

\begin{defn}[Continuous Hessian and the $D_i,$ for $i \in \{0, 1, 2\}$]
Given a Lipschitz rhombus concave function $h:T \ra \R$,  for $i =0, 1, 2,$ we define $D_i h$ to be the nonpositive Radon  measure  on $T$ such that for any $n$, and any parallelogram $e \in E_i(T_n)$, and the corresponding parallelogram $conv(e)/n \subseteq T$,  $$ n^2 \int\limits_{conv(e)/n} D_i h(dx)  = \hess h_n(e).$$ 
We shall define the continuous Hessian $\hess h$ to be the vector valued Radon measure $(D_0h,  D_1h,  D_2 h).$ When $h$ is $C^2$, we have 
\beq D_0 h & = & \partial_x\partial_y h+ \partial^2_x h \lab{eq:D0}\\
         D_1 h & = & - \partial_x \partial_y h\lab{eq:D1}\\
         D_2 h & = & \partial_x \partial_y h + \partial_y^2 h.\lab{eq:D2}\eeq
\end{defn}
%Let $f$ be a function defined on $\mathbb{L}$ such that $\nabla^2(f)$ is periodic modulo $n\mathbb{L}$ and the piecewise linear extension of $f$ to $\R^2$ is concave. Such a function $f$ will be termed discrete concave. Then $\nabla^2(f)$ may be viewed as a function $g$ from $E(\T_n)$ to $\R$.

\begin{defn}[Discrete hive]Let $H_n(\la_n, \mu_n; \nu_n)$ denote the set of all discrete concave functions $h_n:T_n \ra \R,$ (which,  following Knutson and Tao \cite{KT1},  we call discrete hives)
such that \ben \item $\forall i \in [n]\cup\{0\},\quad h_n(0,  i) = \sum_{j = 1}^i \la_n(j).$
\item $\forall i \in [n]\cup\{0\},\quad h_n(i,  n) = \sum_{j = 1}^n \la_n(j) + \sum_{j = 1}^i \mu_n(j).$
\item $\forall i \in [n]\cup\{0\},\quad h_n(i,  i) = \sum_{j = 1}^i \nu_n(j).$
\een
Let $|H_n(\la_n, \mu_n; \nu_n)|$ denote the ${n-1 \choose 2}-$dimensional Lebesgue measure of this hive polytope. 
\end{defn}
%we call discrete hives) $h_n$ from $T_n$ to $\R$, such that the rhombus concave function $h:T\ra \R$ that corresponds via scaling to the piecewise linear extension of $h_n$ to $nT$ with boundary values $\la_n, \mu_n$ and $\nu_n$ on the respective sides joining $(0, 0)$ and $(0, n)$, joining $(0, n)$ and $(n, n)$ and joining $(0, 0)$ and $(n, n)$ each of which is a monotonically decreasing function from $[n]$ to $\R$.

\begin{defn}[Discrete augmented hive] \lab{def:defn6} A function $a_n$ from $\Box_n := ([n]\cup 0)\times ([n]\cup 0) \ra \R$ is called a (discrete) augmented hive and is said to belong to $A_n(\la_n, \mu_n; \nu_n)$ if the following are true.
\ben
\item Its restriction to $T_n$ belongs to $H_n(\la_n, \mu_n; \nu_n).$
\item For every unit lattice parallelogram that is not a square with vertices contained in $\Box_n \cap \{(a, b)|a \geq b\}$,  the sum of the values across the longer diagonal is at most the sum of the values across the shorter diagonal.
\item All vertices with coordinates of the form $(a, 0)$ are mapped to $0.$
\een
\end{defn}

In Figure~\ref{fig:aug-hive}, a single point from the augmented hive polytope $A_4(\la_4, \mu_4; \nu_4)$ with boundary values $\la_4 = (15, 5, -5, -15)$, $\mu_4 = (15, 5, -5, -15)$, $\nu_4 = (15, 5, -5, -15)$ is depicted. Note that the rhombus inequalities corresponding to squares need not be satisfied at the pink diagonal or below it. Indeed they are violated at the square in the southeast corner.
\begin{defn}
 Let $V_n(\la_n(1), \dots, \la_n(n))$ denote the Vandermonde determinant  corresponding to the set of $n$ real numbers $\{\la_n(1), \dots, \la_n(n)\}$, whose value is $\prod\limits_{1 \leq i < j \leq n}(\la_n(i) - \la_n(j)).$ 
\end{defn}
\begin{defn}
Let $\tau_n = \left(\frac{n-1}{2}, \frac{n-3}{2}, \dots, -\frac{n-3}{2}, -\frac{n-1}{2}\right).$
\end{defn}
\begin{lemma}The $m = (n-1)^2$ dimensional Lebesgue measure of the set of all $n \times n$ augmented hives with boundary conditions $\la_n, \mu_n, \nu_n$ is exactly $|H_n(\la_n, \mu_n; \nu_n)|V_n(\nu_n)/V_n(\tau_n)$.\end{lemma}
\begin{proof}

The Lebesgue measure of a full dimensional polytope $P \subset \R^m$ can  be expressed as $\lim_{t \ra \infty}t^{-m}  |tP \cap \Z^n|$ since the boundary effects become insignificant as $t \ra \infty.$

 The restriction of $A_n(\la_n, \mu_n; \nu_n)$ to the South-East lower triangle below the main diagonal is a linear image of a Gelfand-Tsetlin polytope,  and the integer points of this polytope are in one-to-one correspondence with the integer points of a Gelfand-Tsetlin polytope (see Figure~\ref{fig:aug-hive}). The value of the Schur polynomial at the all ones vector of length $n$ is given by $$s_{\la_n}(1, \dots, 1) = \frac{\prod\limits_{1 \leq i < j \leq n}(\la_n(i) - \la_n(j) + j - i)}{\prod\limits_{1 \leq i < j \leq n}(j - i)},$$ which follows from the Weyl character formula.
Since this expression counts the number in integer points in a Gelfand-Tsetlin polytope, by scaling, we see that the volume of the polytope of GT patterns in the lower triangle is exactly $V_n(\nu_n)/V_n(\tau_n)$ (see for example, Theorem 15.1 in \cite{Postnikov} or Theorem 1.1 in \cite{Karola}). 
The Lemma follows.
\end{proof}

\begin{figure}
\begin{center}
\includegraphics[scale=1.0]{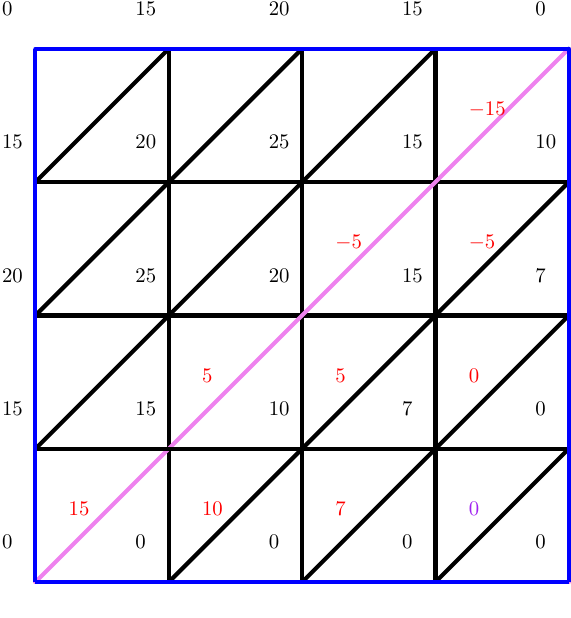}
\caption{An augmented hive for $n = 4$. The right triangle below the main diagonal of the $4\times 4$ square corresponds to a random Gelfand-Tsetlin (GT) pattern (printed in red at the center of a diagonal edge) with top row $\nu_n$, given by the difference of the number
at a vertex and the one southwest to it.
The right triangle above the main diagonal corresponds to a random hive with boundary conditions, $(\la_4, \mu_4, \nu_4)$.}
\label{fig:aug-hive}
\end{center}
\end{figure}

\begin{defn}[$C^k$ hives] A $C^k$ hive for $k \geq 0$ is a rhombus concave map from $T$ to $\R$ that is the restriction of a $C^k$ function on some open neighborhood of $T$.
\end{defn}
\begin{defn}
A $C^0$ augmented hive is a function $a$ from $\Box:= [0,1]\times [0, 1]$, such that for each positive integer $n$, $a_n:\Box_n \ra \R$ defined by $a_n(nx,  ny) = n^2a(x,  y)$ for $(nx, ny) \in ([n]\cup\{0\}) \times ([n]\cup\{0\}),$  is a discrete augmented hive.
\end{defn}

\begin{defn}\lab{def:77old}
Let $\la = \partial^- \a$, $\mu = \partial^- \b$ on $(0, 1]$ and $\nu = \partial^- \g,$ on $(0, 1]$ be bounded, strongly decreasing functions.
We denote by $H(\la, \mu; \nu)$, the set of all $C^0$ hives from $T$ to $\R$, whose restrictions to the three sides are $\a, \beta$ and $\gamma$,  and set  $H(\la, \mu) := \cup_\nu H(\la, \mu; \nu)$. We denote the set of $C^0$ augmented hives with these boundary conditions by $A(\la, \mu; \nu).$
\end{defn}
\begin{comment}
Suppose that there exists $\de > 0$ such that there exists a $C^3$ hive $h_1$ with boundary values $\a, \b, \g_0$, with the following two properties.
Firstly, for every $n \geq 1$, the discrete Hessian of $h_1$ on every lattice rhombus of $T \cap \left(\frac{\Z}{n}\right)^2$ of side length $1/n$ is at most $- \de.$ If 
$\kappa(x): = x(1-x)$ for $x \in [0, 1]$, this first condition is equivalent to  the existence of a $C^3$ hive $h_2$ with boundary conditions $\a - \de \kappa$, $\b - \de \kappa$ and $\g_0 - \de\kappa$. Secondly, suppose that $$\|h_1\|_{C^3} \leq 1.$$
\end{comment}

\begin{defn}[Covering numbers]
Given a metric space $X$ equipped with a distance $d$, and a subset $S \subseteq X$, we define the $\eps-$covering number of $S$ with respect to $d$ to be the smallest integer $N,$ such that there exist $N$ balls of diameter at most $\epsilon,$ whose union covers $S$.  If no such finite $N$ exists, the $\eps-$covering number of $S$ with respect to $d$ is said to be infinite.
\end{defn}
\begin{lemma}\lab{lem:Bronshtein}
Let $\la = \partial^- \a$, $\mu = \partial^- \b$ on $(0, 1]$ and $\nu = \partial^- \g,$ at all points of $(0, 1]$ be bounded, strongly decreasing functions.
Let $\la_n(i) := n^2(\a(\frac{i}{n})-\a(\frac{i-1}{n}))$, for $i \in [n]$, and similarly,
$\mu_n(i) := n^2(\b(\frac{i}{n})-\b(\frac{i-1}{n}))$, and 
$\nu_n(i) := n^2(\g(\frac{i}{n})-\g(\frac{i-1}{n}))$. 
The $n^2\eps-$covering number with respect to the $\ell^\infty$ norm of $H_n(\la_n, \mu_n; \nu_n)$ is at most $\exp\left(\frac{C'}\eps\right)$. Here $C'$ is a finite constant depending only on $\la, \mu$ and $\nu$.
\end{lemma}
\begin{proof}
For every discrete hive $h_n \in H_n(\la_n, \mu_n; \nu_n),$ there is a piecewise linear $C-$bounded $L-$Lipschitz concave function $h \in H(\la, \mu; \nu)$ on $T$  that satisfies $h(x, y):= n^{-2}h_n(nx, ny)$ for $(x, y) \in T\cap (\Z/n \times \Z/n),$ where $C$ and $L$ are finite constants depending only on $\la, \mu$ and $\nu.$
 Bronshtein's Theorem, Theorem 6, \cite{Brsh},  (see also \cite{Wellner}) gives an upper bound of $\exp\left(C'\eps^{-\frac{d}{2}}\right)$ on the  $\eps-$covering number with respect to the $L^\infty-$norm of the set of $C-$bounded $L-$Lipschitz convex functions on a compact domain in $\R^d$,  where $C'$ depends only on $C$ and $L$.  Therefore,  $n^2\eps-$covering number with respect to the $\ell_\infty-$norm of $H_n(\la_n, \mu_n; \nu_n)$ is at most $\exp\left(\frac{C'}\eps\right)$.
 \end{proof}
\subsection{Large deviations principle}
\begin{defn}
Given a Polish space 
$\Gamma$ let 
${\displaystyle \{\mathbb {P}_{n}\}}$ be a sequence of Borel probability measures on 
$\Gamma,$ let 
$\{a_{n}\}$ be a sequence of positive real numbers such that 
${\displaystyle \lim_{n \ra \infty}a_{n}=\infty ,}$ and let
${\displaystyle I:\Gamma\to [0,\infty ]}$ be a lower semicontinuous functional on 
$\Gamma.$ For each Borel measurable set 
${\displaystyle E\subset \Gamma,}$ let 
$\overline{E}$ and 
$E^{\circ }$ denote respectively the closure and interior of 
$E.$ The sequence 
${\displaystyle \{\mathbb {P}_{n}\}}$ is said to satisfy a large deviation principle with speed 
$\{a_{n}\}$ and rate  
$I$ if, and only if, for each Borel measurable set 
${\displaystyle E\subset \Gamma,}$
$${\displaystyle -\inf _{\g\in E^{\circ }}I(\g)\leq \liminf_{n \ra \infty}a_{n}^{-1}\log(\mathbb {P} _{n}(E))\leq \limsup_{n \ra \infty}a_{n}^{-1}\log(\mathbb {P}_{n}(E))\leq -\inf_{\g\in {\overline {E}}}I(\g).}$$

\end{defn}

\subsection{Periodic boundary conditions}
 Let $\T_n$ be the $n \times n$ torus whose vertices are in correspondence with those of  $\Box_n$, obtained by identifying the $n^{th}$ row with the $0^{th}$ row and the $n^{th}$ column with the $0^{th}$ column. 
\bit
\item Let $E_0(\T_n)$ be the set of all parallelograms $e_0\subseteq \T_n$ whose vertices are $\{(v_1, v_2),  (v_1 + 1, v_2),  (v_1 + 1, v_2 + 1),  (v_1 +2, v_2 + 1)\}.$

\item Let $E_1(\T_n)$ be the set of all parallelograms $e_1\subseteq \T_n$ whose vertices are $\{(v_1, v_2),  (v_1 + 1, v_2),  (v_1, v_2 + 1),  (v_1 +1, v_2 + 1)\}.$ 
\item Let $E_2(\T_n)$ be the set of all parallelograms $e_2\subseteq \T_n$ whose vertices are$\{(v_1, v_2),  (v_1 + 1, v_2+1),  (v_1, v_2 + 1),  (v_1 +1, v_2 + 2)\}.$ 
\eit
Let $V(\T_n)$ be the set of vertices of $\T_n.$ Where there is no ambiguity,  we will abbreviate $V(\T_n)$ to $\T_n.$

{
For $s_0, s_1, s_2 > 0$ and $f:V(\T_n) \ra \R$, we say that $ \nabla^2(f)$ satisfies $\nabla^2(f) \preccurlyeq s = (s_0, s_1, s_2)$, if  $\nabla^2(f)$ satisfies

\ben 
\item  $\nabla^2(f)(e) \leq  s_0,$ if $e \in E_0(\T_n)$.
\item $\nabla^2(f)(e)  \leq  s_1,$  if $e \in E_1(\T_n)$.
\item $\nabla^2(f)(e) \leq  s_2,$ if  $e \in E_2(\T_n)$.
\een

}
\begin{defn}
 Given  $s = (s_0, s_1, s_2)\in \R_+^3,$ let $P_n(s)$ be the bounded polytope of  all functions $g:V(\T_n) \ra \R$ such that $\sum_{v \in V(\T_n)} g(v) = 0$ and $\nabla^2(g)\preccurlyeq s$. 
 \end{defn}
 
The following Lemma appears as Lemma 3.8 in \cite{random_concave}.  For the sake of completeness, the proof has been reproduced in the section below.
\begin{lemma}\lab{lem:2.8}
Let $\min_i s_i = 2.$ Then, as $n\ra \infty$, $|P_n(s)|^{\frac{1}{n^2}}$  converges to a limit in the interval $[1, 2e]$.
\end{lemma}
\subsection{Hives}
\newcommand{\tih}{\tilde{h}}
 
Denoting probability densities  with respect to $\Leb_{n-1, 0}$ by $\rho_n$, it is known through the work of Knutson and Tao (see  \cite{KT2} and also Equation (4) in \cite{Zuberhorn}) that 

\beq\lab{eq:2.4new} \rho_n\left[\spec(X_n + Y_n) = \nu_n\right] =  \frac{V_n(\nu_n)V_n(\tau_n)}{V_n(\la_n)V_n(\mu_n)} |H_n(\la_n, \mu_n; \nu_n)|.\eeq

Then, denoting $X_n + Y_n$ by $Z_n$, for all positive $\epsilon$, we have
\beq \p_n\left[\spec(Z_n) \in B_\I^n(\nu_n, \eps n^2)\right] =\left( \frac{V_n(\tau_n)^2}{V_n(\la_n)V_n(\mu_n)} \right)\int\limits_{B_\I^n(\nu_n, \eps n^2)} \left(\frac{V_n(\nu'_n)}{V_n(\tau_n)}\right) |H_n(\la_n, \mu_n; \nu'_n)|\Leb_{n-1, 0}(d\nu'_n).\nonumber\\\lab{eq:3.2}\eeq

We view $$\bigcup\limits_{\nu'_n \in \Spec_n}A_n(\la_n, \mu_n; \nu'_n)$$
as a $(n-1)^2 + (n-1)$ dimensional polytope contained in $\R^{(n+1)^2}$ with coordinates $(a_{ij})_{i, j \in [n+1]},$ corresponding to entries in an augmented hive, as in Figure~\ref{fig:aug-hive}. Thus, $j$ increases from bottom to top, and $i$ increases from left to right. Let $\Pi_{diag, n}:\R^{(n+1)^2} \ra \R^{n-1}$, denote the evaluation map that takes a $(n+1) \times (n+1)$  matrix $A$ to its diagonal (without the endpoints) $(a_{i, i})_{2\leq i \leq n}.$ Let $\Pi_{diag}$ denote the continuous analogue of $\Pi_{diag, n}.$

Recall from the beginning of Section~\ref{sec:prelim} that $\Spec_n$ denotes the cone of all possible vectors $x = (x_1, \dots, x_n)$ in $\R^n$ such that $$x_1 \geq x_2 \geq \dots \geq x_n,$$ and $\sum_i x_i = 0.$ 
\begin{defn}Let $\eps_0 > 0$ and let $\la_n, \mu_n, \nu_n \in \Spec_n.$ Let $$G_n(\la_n, \mu_n; \nu_n, \eps_0 n^2) := \left(\bigcup\limits_{\nu'_n \in \Spec_n}A_n(\la_n, \mu_n; \nu'_n)\right)\cap \Pi_{diag, n}^{-1}(B_\I^n(\nu_n, \eps_0n^2)).$$
Similarly, let $$G(\la, \mu; \nu, \eps_0) := \left(\bigcup\limits_{\nu'}A(\la, \mu; \nu')\right)\cap \Pi_{diag}^{-1}(B_\I(\nu, \eps_0)).$$
\end{defn}
The volume of this $n^2 - n$ dimensional polytope is,
\beq |G_n(\la_n, \mu_n; \nu_n, \eps_0 n^2) | & = & \int\limits_{B_\I^n(\nu_n, \eps_0 n^2)}|A_n(\la_n, \mu_n; \nu'_n)|\Leb_{n-1, 0}(d\nu'_n).\lab{eq:G}\\
& = & \int\limits_{B_\I^n(\nu_n, \eps_0 n^2)} \left(\frac{V_n(\nu'_n)}{V_n(\tau_n)}\right) |H_n(\la_n, \mu_n; \nu'_n)|\Leb_{n-1, 0}(d\nu'_n).\eeq We need to bound this volume from above and below in a suitable manner.

\begin{defn}
Given a $C^0$ hive $h \in H(\la, \mu; \nu)$, let $L_n(h)$ denote the point  in $H_n(\la_n, \mu_n; \nu_n)$ whose $ij^{th}$ coordinate is given by $$(L_n(h))_{ij} := n^2 h\left(\frac{i}{n}, \frac{j}{n}\right),$$ for all $i, j \in [n+1]$ such that $i \leq j$.
Let $U_n$ denote the set of lattice points $(x, y) \in \Z^2$ belonging to the triangle in a $[n+1]\times [n+1]$ index set of an augmented hive, such that $0 < y < x \leq n+1$. 
%Let $W_n$ denote the set of lattice points $(x, y) \in \Z^2$ belonging to the triangle in a $[n+1]\times [n+1]$ index set of an augmented hive, such that $0 < x \leq  y \leq n$.
Let $U$ denote the set of points $(x, y) \in [0, 1]^2$ such that $0 < y < x \leq 1.$
\end{defn}

\section{Convex geometry}
Let $1 \leq \ell \in \Z$. Given sets $K_i\subseteq \R^m$ for $i \in [\ell]$, let their Minkowski sum $\{x_1 + \dots + x_\ell \big| \forall i \in [\ell], x_i \in K_i\},$ be denoted by $K_1 + \dots + K_\ell.$
We state the Brunn-Minkowski inequality below.
\begin{thm}[Brunn-Minkowski] 
Let $K$ and $L$ be compact convex subsets of $\R^m$. 
\beq |K + L|^\frac{1}{m} \geq |K|^\frac{1}{m} + |L|^\frac{1}{m}.\eeq 
\end{thm}
We will also need the following integral generalization of the Brunn-Minkowski inequality that follows from replacing the above sum by a weighted integral.
\begin{corollary}\lab{cor:Brunn}
Let $f$ be a probability density defined on a convex set $J \subset \R^d$ and let the image of a (not necessarily orthogonal) projection $\Pi_d$ of $K$ on to $\R^d$ of a compact $m$-dimensional convex set $K \subseteq \R^m \supseteq \R^d$ be equal to $J.$ Suppose that for every $x \in J,$ the $m-d$-dimensional volume of the preimage of $x$ under $\Pi_d$ is proportional to $f(x).$ Then, 
\beqs \left(\vol_{m-d}\left(\Pi_d^{-1}\left(\frac{\Pi_d \int_{v \in K} vdv}{\vol_m(K)}\right)\right)\right)^{\frac{1}{m-d}} \geq  \int_{v \in J} f(\Pi_d v) (\vol_{m-d}(\Pi_d^{-1}(\Pi_d v)))^{\frac{1}{m-d}}dv.
\eeqs
\end{corollary}
%It can be shown that 
%$$\lim_{\eps \ra 0^+} \frac{|L + \eps K| - |L|}{\eps}$$ exists. We will call this the anisotropic surface area $S_K(L)$ of $L$ with respect to $K$. 

% Dinghas \cite{Dinghas, Figalli} showed that the following anisotropic isoperimetric inequality can be derived from the Brunn-Minkowski inequality.
%\beq\lab{eq:2.2} S_K(L) \geq m|K|^{\frac{1}{m}} |L|^{\frac{m-1}{m}}.\eeq 

We shall need the following result of Pr\'{e}kopa (\cite{prekopa}, Theorem 6).
\begin{thm}[Pr\'{e}kopa]\lab{thm:prekopa}
Let $f(x, y)$ be a function of $\R^n \oplus \R^m$ where $x \in \R^n$ 
and $y \in \R^m$. Suppose that $f$ is logconcave in $\R^{n+m}$ and let
$A$ be a convex subset of $\R^m$. Then the function of the variable x:
$$\int_A f(x, y) dy$$
is logconcave in the entire space $\R^n$.
\end{thm}

We note the following theorem of Fradelizi \cite{Fradelizi}. 
\begin{thm}[Fradelizi]\lab{thm:frad}
 The density at the center of mass of a logconcave density on $\R^{n}$ is greater or equal to $e^{- n}$ multiplied by the supremum of the density.
\end{thm}

We will also need the following theorem of Vaaler (Corollary on page 554, \cite{Vaaler}).
\begin{thm}[Vaaler]\lab{thm:vaal}
The volume of a central section of the unit  cube $[-\frac{1}{2}, \frac{1}{2}]$ is greater or equal to $1$.
\end{thm}
\subsection{Volume of the polytope $P_n(s)$}
For the convenience of the reader we have reproduced material from  Section 3 of the preprint \cite{random_concave} here culminating in the proof of Lemma 3.8 in \cite{random_concave}, which is Lemma~\ref{lem:2.8} here.
Given a $k$-dimensional polytope $P,$ we denote by $|P|$ its $k$-dimensional Lebesgue measure $\vol_k P.$ We will need to show that $|P_m(s)|^{1/m^2}$ is less than $(1 + o_m(1))|P_n(s)|^{\frac{1}{n^2}},$ for $ n \geq m$. We achieve this by conditioning on a ``double layer boundary" and the use of the Brunn-Minkowski inequality.
%We will identify $\Z + \Z\omega$ with $\Z^2$ by mapping $x + \omega y$, for $x, y \in \Z$ onto $(x, y) \in \Z^2.$
Given a linear operator $T:\R^n \ra \R^n$, we denote the operator norm of $T$, namely $\max_{x \in \R^n}\|Tx\|_2/\|x\|_2,$ by $\|T\|_{op}.$ Let $\one$ denote the vector whose every coordinate is $1.$ In this section, we will assume that  $2 = s_0 \leq s_1 \leq s_2.$ 
\begin{lemma}\lab{lem:op}
Let $P \subseteq \R^n$ be a $n-d$ dimensional polytope.  Let  $T:\R^n \ra \R^n$ be a  linear map such that $\|T\|_{op} \leq M.$   Suppose that  $V$ is an  $n-k$ dimensional subspace  of $\R^n$ on which $T$ acts as identity.  Then \beq \frac{\vol_{n-d} T(P)}{\vol_{n-d} P} \leq M^k.\eeq
\end{lemma}
\begin{proof}
We may assume that the dimension of the linear span of $T(P)$ is the same as the dimension of the linear span of $P$, since otherwise, $\vol_{n-d} T(P) = 0,$ and the lemma holds. Let $V'$ denote the linear span of $P,$ and let it $n' =\dim V'$, and let $n' - k' := \dim V' \cap V.$
We see that $$k' = \left(\dim V' - \dim V'\cap V\right) = \dim V'/(V' \cap V) \leq \dim \R^n/V = k.$$ Let $v_1, \dots, v_{n'-k'}$ be an orthonormal basis for $V' \cap V$, and let 
$v_1, \dots, v_{n'}$ be an orthonormal basis of $V'.$ Denoting the exterior product by $\wedge$, we then see that 
\beqs \frac{\vol_{n-d} T(P)}{\vol_{n-d} P} & \leq & \frac{|Tv_1 \wedge \dots \wedge T v_{n'}|}{|v_1 \wedge \dots \wedge v_{n'}|}\\
& \leq & M^{k'}\\
& \leq & M^k.\eeqs
\end{proof}
\begin{defn}Let $\tP_n(s)$ be defined to be the image of $P_n(s)$ under the affine transformation $T: g \mapsto g - g(0) \one.$  Thus, given  $s = (s_0, s_1, s_2)\in \R_+^3,$  where $2 = s_0 \leq s_1 \leq s_2,$ we let $\tP_n(s)$ be the bounded polytope of  all functions $\tilde{g}:V(\T_n) \ra \R$ such that $\tilde{g}(0) = 0$ and $\nabla^2(g)\preccurlyeq s$. 
\end{defn}
\begin{lemma}\lab{lem:28}
The $n^2-1$ dimensional Lebesgue measures of $\tP_n(s)$ and  $P_n(s)$ satisfy
$$|\tP_n(s)|^{1/n^2}\left(1 - \frac{C\log n}{n^2}\right) \leq |P_n(s)|^{1/n^2} \leq |\tP_n(s)|^{1/n^2}.$$
\end{lemma}
\begin{proof}
The map $T:\R^{V(\T_n)} \ra \R^{V(\T_n)}$ where  $T: g \mapsto g - g(0) \one,$ acts as identity on every 
point in the $n^2-1$ dimensional subspace $\{g \in \R^{V(\T_n)}|g(0) = 0\}.$ We see that $T$ maps $P_n(s)$ on to $\tP_n(s).$ Denoting the identity map by $I$, we see that $$\|(T - I) (g)\|_2 \leq \|g\|_2 \|\one\|_2,$$ and therefore, $\|T-I\|_{op} \leq \sqrt{|V(\T_n)|} + 1.$ Applying Lemma~\ref{lem:op} to $T$ with $k = d = 1$, and $M = \sqrt{|V(\T_n)|} + 1$ we see that $$ |\tP_n(s)|^{1/n^2}\left(1 - \frac{C\log n}{n^2}\right) \leq |\tP_n(s)|^{1/n^2}(n+1)^{-\frac{1}{n^2}} \leq |P_n(s)|^{1/n^2} .$$
Next, define $\tilde{T}:g \mapsto g - (n^{-2}\sum_{v \in V(\T_n)} g(v))\one.$  We see that $\tilde{T}$ maps $\tP_n(s)$ on to $P_n(s).$ Also, $\tilde{T}$ is an orthogonal projection on to $\{g \in \R^{V(\T_n)}|\sum_{v \in V(\T_n)} g(v) = 0\}.$ Therefore, $\|\tilde{T}\|_{op} = 1.$  Applying Lemma~\ref{lem:op} to $\tilde{T}$ with $k = d = 1$, and $M = 1,$ we see that $$ |P_n(s)|^{1/n^2} \leq |\tP_n(s)|^{1/n^2} .$$
 
\end{proof}

%Let $\{0\} \subseteq \b_1 \neq \{0\},$ be a subset of $V(\T_{n_1}).$ 
%The following was proved in Part I, \cite{PartI}.
%\begin{notation}In what follows, we use $\omega$ to denote the the complex number $\exp\left(\frac{2\pi\iota}{3}\right)$. We identify $\mathbb C$ with $\R \oplus  \R \iota$, and thereby represent by $\omega$ the corresponding point in $\R^2.$\end{notation}
Recall from the beginning of this section that given a $k$-dimensional polytope $P,$ we denote by $|P|$ its $k$-dimensional Lebesgue measure $\vol_k P.$ Thus in particular, the presence of $\{0\}$ in $\bb_1$ in Lemma~\ref{lem:3-} does not cause $|\Pi_{\bb_1} \tP_{n_1}(s)|$ to go to $0.$

Given $n_2$ that is a multiple of $n_1,$ the quotient map from $\Z^2$ to $\Z^2/(n_1 \Z^2) = \T_{n_1}$ factors through $\Z^2/(n_2 \Z^2) =\T_{n_2}$. We denote the respective resulting maps from $\T_{n_2}$ to $\T_{n_1}$ by $\phi_{n_2, n_1}$, from $\Z^2$ to $\T_{n_2}$ by $\phi_{0, n_2}$ and from $\Z^2$ to $\T_{n_1}$ by $\phi_{0, n_1}$.
Given a set of boundary nodes $\bb \subseteq V(\T_n)$, and $ x \in \R^{\bb}$, we define $Q_{ \bb}(x)$ to be the fiber polytope over $x$, that is the preimage of $x$ under the projection map $\Pi_{\bb}$ of $\tP_n(s)$ onto $\R^{\bb}.$ Note that $Q_{\bb}(x)$ implicitly depends on $s$. 
\begin{lemma}\lab{lem:3-}
Let $\{0\} \subseteq \bb_1 \neq \{0\},$ be a subset of $V(\T_{n_1}).$ Then,
$$0 \leq \ln |\Pi_{\bb_1} \tP_{n_1}(s)| \leq (|\bb_1| -1)\ln (Cn_1^2).$$
\end{lemma}
\begin{proof}
 Given any vertex $v_1$ in $\bb_1$ other than $0$, there is a lattice path $\mathrm{path}(v_1)$ (i.e. a path $0= a_1, \dots, a_k = v_0$, where each $a_i - a_{i-1}$ is in the set $\{(0, 1), (1, 0)\}$) that goes from $0$ to some vertex $v_0 \in \phi^{-1}_{0,n_1}(v_1)$ that consists of two straight line segments, the first being a horizontal segment from $0$ to some point in $\Z^+ \times  \{0\}$, and the second being a vertical segment having the direction $(0, 1)$. It is clear that this $v_0$ can be chosen to have $\ell_1$-norm at most $2n_1$ by taking an appropriate representative of $\phi^{-1}_{0,n_1}(v_1).$
Since $2 = s_0 \leq s_1 \leq s_2,$ as specified in the beginning of this section, we see that $\{0\}\times [0,1]^{\bb_1\setminus\{0\}} \subseteq \Pi_{\bb_1} \tP_{n_1}(s) \subseteq \{0\} \times \R^{\bb_1\setminus \{0\}}.$ Let $f_1 \in  \tP_{n_1}(s)$. Along $\mathrm{path}(v_1)$, at each step, the slope of $f$ increases by no more than $s_0 + s_1$ in the horizontal stretch, due to the condition $(\Delta_0 + \Delta_1)(f_1) \leq s_0 + s_1,$ and slope of $f$ increases by no more than $s_1 + s_2$ in the vertical stretch, due to the condition $(\Delta_1 + \Delta_2)(f_1) \leq s_1 + s_2.$ By performing a translation of the torus that makes $v_1$ the origin, and repeating the same argument after reversing the roles of $0$ and $v_1$, this implies that $f_1$ is $Cn_1$-Lipschitz. Therefore, $\|f_1\|_\infty$ is at most $Cn_1^2.$ Thus $\Pi_{\bb_1} \tP_{n_1}(s)$ is contained inside a $|\bb_1| -1$ dimensional cube of side length no more than $Cn_1^2.$ 
We have thus proved the lemma.
\end{proof}

\begin{lemma}\lab{lem:3}
Let $n_1$ and $n_2$ be positive integers satisfying $n_1 | n_2$. Then 
\beq 1 \leq |\tP_{n_1}(s)|^{\frac{1}{n_1^2}} \leq |\tP_{n_2}(s)|^{\frac{1}{n_2^2}}\left(1 + \frac{C\log n_1}{n_1}\right).\eeq 
\end{lemma}
\begin{proof}The lower bound of $1$ on $|\tP_{n_1}(s)|^{\frac{1}{n_1^2}}$ follows from $[0,1]^{V(\T_{n_1})\setminus\{0\}} \subseteq \tP_n(s).$
We define the set $\bb_{1} \subseteq V(\T_{n_1})$ of ``boundary vertices'' to be all vertices that are either of the form $(0, y)$ or $(1, y)$ or $(x, 0)$ or $(x, 1)$, where $x, y$ range over all of $\Z/(n_1 \Z)$. We define the set $\bb_{2}$ to be $\phi_{n_2, n_1}^{-1}(\bb_{1}).$ For $x \in \R^{\bb_1}$, let $F_1(x):= |Q_{\bb_1}(x)|,$ and for  $x \in \R^{\bb_2}$, let $F_2(x):= |Q_{\bb_2}(x)|.$
We now have \beq|\tP_{n_1}(s)| = \int\limits_{\R^{\bb_1}}F_1(x)dx = \int\limits_{\Pi_{\bb_1} \tP_{n_1}(s)} F_1(x) dx.\eeq
%Let  $\phi^*_{n_2, n_1}$ be the linear map from $\R^{V(\T_{n_1})}$ to $\R^{V(\T_{n_2})}$ induced by $\phi_{n_2, n_1}$. 
Let  $\psi_{\bb_1, \bb_2}$ be the linear map from $\R^{\bb_1}$ to $\R^{\bb_2}$ induced by $\phi_{n_2, n_1}$. 
%Let $\mu_{21}$ be the push-forward of the Lebesgue measure restricted to $\Pi_{\b_1} P_{n_1}(s)%$ on to $\R^{\b_2}$ associated with  the map $\psi_{\b_1, \b_2}$. 
Then, for $x \in \R^{\bb_1},$ \beq F_2(\psi_{\bb_1, \bb_2}(x)) = F_1(x)^{\left(\frac{n_2}{n_1}\right)^2}.\lab{eq:2.6}\eeq
Note that that $\tP_{n}(s)$ is $n^2-1$ dimensional, has an $\ell_\infty$ diameter of $O(n^2)$ and contains a $n^2-1$ dimensional unit $\ell_\infty$-ball as a consequence of $2 = s_0 \leq s_1 \leq s_2.$ So the $|\bb_1|-1$ dimensional polytopes  $\Pi_{\bb_1} \tP_{n_1}(s)$, and $\psi_{\bb_1, \bb_2}(\Pi_{\bb_1} \tP_{n_1}(s))$ contain $|\bb_1|-1$ dimensional $\ell_\infty$ balls of radius $1.$ 
\begin{claim} \lab{cl:2.2}
Let $S_{\bb_1, \bb_2}(\frac{1}{n_1^{4}})$ be the set of all $y \in \R^{\bb_2}$ such that there exists $x \in  \Pi_{\bb_1} \tP_{n_1}((1 - \frac{1}{n_1^2})s)$ for which  $ y - \psi_{\bb_1, \bb_2}(x) \perp \psi_{\bb_1, \bb_2}(\R^{\bb_1})$  and $\|y - \psi_{\bb_1, \bb_2}(x) \|_\infty < \frac{1}{n_1^{4}}.$ Then, $y \in S_{\bb_1, \bb_2}(\frac{1}{n_1^{4}})$ implies the following. \ben \item $y \in \Pi_{\bb_2} \tP_{n_2}((1 - \frac{1}{2n_1^2})s)$ and \item 
 $|Q_{\bb_2} (y)| \geq c^{(\frac{n_2}{n_1})^2} |Q_{\bb_2} (\psi_{\bb_1, \bb_2}(x))|.$\een
 Here $c, C$ are absolute constants that do not depend on $n_1$ and $n_2.$
\end{claim}
\begin{proof}
The first assertion of the claim follows from the triangle inequality. To see the second assertion, 
let the vector $w \in \R^{V(\T_{n_2})}$ equal $0$ on all the coordinates indexed by $V(\T_{n_2})\setminus \bb_2$ and equal $\psi_{\bb_1, \bb_2}(x) - y$ on coordinates indexed by $\bb_2$.
We know that $x \in  \Pi_{\bb_1} \tP_{n_1}((1 - \frac{1}{n_1^2})s)$. Therefore, \ben \item[$(\ast)$] $Q_{\bb_2} (\psi_{\bb_1, \bb_2}(x)) -w$ has dimension $n_2^2 - |\bb_2|$, and contains an axis aligned cube of side length $\frac{c}{n_1^2},$ and hence a euclidean ball of radius $\frac{c}{n_1^2}$.\een Since every constraint defining $\tP_{n_2}(s)$ has the form $x_a + x_b - x_c - x_d \leq s_i,$ or $x_0 = 0$, \ben \item[$(\ast \ast)$] the affine spans of the codimension $1$ faces of the fiber polytope $Q_{\bb_2}(y)$ are respectively translates of the affine spans of the corresponding codimension $1$ faces of $Q_{\bb_2} (\psi_{\bb_1, \bb_2}(x)) -w$ by  euclidean distances that do not exceed $\frac{C}{ n_1^{4}}.$ \een
%By construction, the affine spans of $Q_{\b_2} (\psi_{\b_1, \b_2}(x)) -w$ and  $Q_{\b_2} (y)$ coincide. 
 If $K$ is an polytope in $\R^n$ containing a ball $B(0, r)$ and $K'$ is an affine polytope obtained by translating the faces of $K$ by a distance at most $h$, then $(1 - \frac{h}{r})K' \subset K$.
Indeed,  it suffices to prove this for half-spaces that contain $B(0, r)$, because $K$ is an intersection of such half-spaces. But if one translates a half-space by $h$ whose bounding hyperplane is at a distance $R$ from the origin,  then the distance to the origin of the bounding hyperplane of the translated half-space is at most $R + h$. Since $r \leq R$,  $(1 - \frac{h}{r})K' \subset K$.

Therefore, by $(\ast)$ and $(\ast \ast)$, some translate of $(1 - \frac{C}{n_1^{2}})Q_{\bb_2} (\psi_{\bb_1, \bb_2}(x))$ is contained inside $Q_{\bb_2} (y)$ (see Figure~\ref{fig:astast}), completing the proof of Claim~\ref{cl:2.2}.
\end{proof}

\begin{figure}
\begin{center}
\includegraphics[scale=1.0]{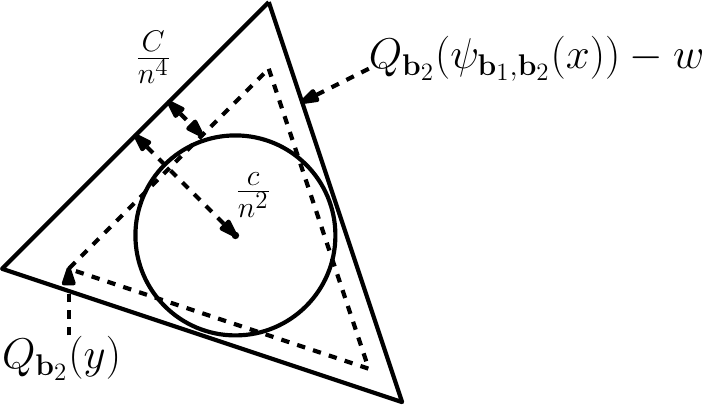}
\caption{Some translate of $(1 - \frac{C}{n_1^{2}})Q_{\bb_2} (\psi_{\bb_1, \bb_2}(x))$ is contained inside $Q_{\bb_2} (y)$.}
\label{fig:astast}
\end{center}
\end{figure}
%Also, $F_2$ being the volume of a section of $P_{n_2}$, which is in turn contained in a $n_2^2$ dimensional cube of sidelength $Cn_2^2$, is by \cite{Ball1}, bounded above by $2^{\b_2/2} (Cn_2^2)^{n_2^2- |\b_2|},$ and so we have (\ref{eq:2.8}) below.
Let $K$ denote the intersection of the origin symmetric cube of radius $\frac{1}{n_1^{4}}$ in $\R^{\bb_2}$ with the orthocomplement of $\psi_{\bb_1, \bb_2}(\R^{\bb_1})$. By the lower bound of $1$ on the volume of a central section of the unit  cube due to Vaaler (Theorem~\ref{thm:vaal}), it follows that the volume of $K$ is at least $\left(\frac{1}{n_1^{4}}\right)^{|\bb_2| - |\bb_1|}.$
The inequalities below now follow from (\ref{eq:2.6}) and Claim~\ref{cl:2.2}.
\beqs |\tP_{n_2}(s)|  & = & \int\limits_{\Pi_{\bb_2} \tP_{n_2}(s)}|Q_{\bb_2}(y)|dy\\
                             & \geq & \int\limits_{\Pi_{\bb_2} \tP_{n_2}((1 - \frac{1}{2n_1^2})s)} F_2(y) dy\lab{eq:2.8}\\
& \geq & \int\limits_{ S_{\bb_1, \bb_2}(\frac{1}{n_1^4})} F_2(y) dy\\
& \geq & \vol(K)\int\limits_{\psi_{\bb_1, \bb_2}(\Pi_{\bb_1} \tP_{n_1}((1 - \frac{1}{n_1^2})s))} c^{(\frac{n_2}{n_1})^2}F_2(z) dz\\
& \geq &  \vol(K)\int\limits_{\Pi_{\bb_1} \tP_{n_1}((1 - \frac{1}{n_1^2})s)}c^{(\frac{n_2}{n_1})^2} F_1(x)^{\left(\frac{n_2}{n_1}\right)^2}dx\nonumber\\
& \geq &  c^{(\frac{n_2}{n_1})^2}\left(\frac{1}{n_1^{4}}\right)^{|\bb_2|-|\bb_1|}\int\limits_{\Pi_{\bb_1} \tP_{n_1}((1 - \frac{1}{n_1^2})s)} F_1(x)^{\left(\frac{n_2}{n_1}\right)^2}dx.\eeqs

By Lemma~\ref{lem:3-}, $ 1 \leq |\Pi_{\bb_1} \tP_{n_1}(s)| \leq n_1^{Cn_1}$, for some universal positive constant $C > 1.$ Also, \beq c  |\Pi_{\bb_1} \tP_{n_1}(s)| & \leq & (1 - \frac{1}{n_1^2})^{|\bb_1|-1}|\Pi_{\bb_1} \tP_{n_1}(s)|\nonumber\\ & = & |\Pi_{\bb_1} \tP_{n_1}((1 - \frac{1}{n_1^2})s)| \lab{eq:c-3.5}\\ & \leq & |\Pi_{\bb_1} \tP_{n_1}(s)|.\nonumber\eeq
%Since $n_2 \geq n_1$, 

We note that 
\beqs \int\limits_{\Pi_{\bb_1} \tP_{n_1}((1 - \frac{1}{n_1^2})s)} F_1(x)^{\left(\frac{n_2}{n_1}\right)^2}dx & \geq & |\Pi_{\bb_1} \tP_{n_1}((1 - \frac{1}{n_1^2})s)|^{1 - (n_2/n_1)^2}  \\ &\times&\left(\int\limits_{\Pi_{\bb_1} \tP_{n_1}((1 - \frac{1}{n_1^2})s)} F_1(x)dx\right)^{\left(\frac{n_2}{n_1}\right)^2}\\
& \geq &  |\Pi_{\bb_1} \tP_{n_1}(s)|^{1 - (n_2/n_1)^2}  | \tP_{n_1}((1 - \frac{1}{n_1^2})s)|^{\left(\frac{n_2}{n_1}\right)^2}\\
& \geq &  |\Pi_{\bb_1} \tP_{n_1}(s)|^{1 - (n_2/n_1)^2}  \left(c| \tP_{n_1}(s)|\right)^{\left(\frac{n_2}{n_1}\right)^2} \quad \text{by}\,\, (\ref{eq:c-3.5})\\
& \geq & (Cn_1^{Cn_1})^{1 - (n_2/n_1)^2} | \tP_{n_1}(s)|^{\left(\frac{n_2}{n_1}\right)^2}. \eeqs
Thus,
\beq | \tP_{n_1}(s)|^{\left(\frac{n_2}{n_1}\right)^2} \leq (Cn_1^{Cn_1})^{(n_2/n_1)^2-1} \left(n_1^{4}\right)^{|\bb_2|-|\bb_1|}|\tP_{n_2}(s)|,  \eeq
which gives us \beq| \tP_{n_1}(s)|^{\left(\frac{1}{n_1}\right)^2} & \leq & (Cn_1^{Cn_1})^{(1/n_1^2)-(1/n_2^2)} \left(n_1^{4}\right)^{\frac{|\bb_2|-|\bb_1|}{n_2^2}}|\tP_{n_2}(s)|^{\frac{1}{n_2^2}}\\
& \leq & |\tP_{n_2}(s)|^{\frac{1}{n_2^2}}n_1^{\frac{C}{n_1}}\\
& \leq & |\tP_{n_2}(s)|^{\frac{1}{n_2^2}}\left(1 + \frac{C \log n_1}{n_1}\right).\eeq
\end{proof}
For a positive integer $n$, let $[n]$ denote the set of positive integers less or equal to $n$, and let $[n]^2$ denote $[n]\times [n]$.
In what follows, we will use $v$ to denote an arbitrary vertex in $V(\T_{n_3})$. Then, by symmetry, \beq \frac{\int_{P_{n_3}(s)} x(v) dx}{|P_{n_3}(s)|} & = &  \left(\frac{1}{n_3^2}\right)\sum_{v' \in V(\T_{n_3})}  \frac{\int_{P_{n_3}(s)} x(v') dx}{ |P_{n_3}(s)|} \\
& = &   \frac{\int_{P_{n_3}(s)} \left(\frac{\sum_{v' \in V(\T_{n_3})} x(v')}{n_3^2}\right)dx}{|P_{n_3}(s)|} \\
& = & 0.\lab{eq:2.10} \eeq
The linear map $u:P_{n_3}(s) \rightarrow \tP_{n_3}(s)$ defined  by $u(x)(v) = x(v) - x(0)$ is surjective and injective. Therefore,  
\beq \frac{\int_{\tP_{n_3}(s)} x(v) dx}{|\tP_{n_3}(s)|} & = &   \frac{\int_{P_{n_3}(s)} u(x)(v) dx}{ |P_{n_3}(s)|} \\
& = &  \frac{\int_{P_{n_3}(s)} x(v) dx}{|P_{n_3}(s)|} -  \frac{\int_{P_{n_3}(s)} x(0) dx}{|P_{n_3}(s)|}\\
& = & 0.\eeq
\begin{lemma}\lab{lem:4}
Let $C < n_2 < n_3$. Then, 
\beq |P_{n_2}(s)|^{\frac{1}{n_2^2}} \geq |P_{n_3}(s)|^{\frac{1}{n_3^2}}\left(1 - \frac{C (n_3 - n_2) \ln n_3}{n_3}\right).\eeq  
\end{lemma}
\begin{proof}  Let $\rho:V(\T_{n_2}) \ra [n_2]^2\subseteq \Z^2$ be the unique map that satisfies $\phi_{0, n_2} \circ \rho = id$ on $V(\T_{n_2})$. We embed $V(\T_{n_2})$ into $V(\T_{n_3})$ via  
  $\phi_{0, n_3} \circ \rho,$ and define $\bb$ to be $ V(\T_{n_3})\setminus (\phi_{0, n_3} \circ \rho(V(\T_{n_2}))).$  Note that $0 \in \bb$, since $0 \not \in [n_2].$ Recall that $Q_{ \bb}(x)$ was defined to be the fiber polytope over $x$, that arises from the projection map $\Pi_{\bb}$ of $\tP_n(s)$ onto $\R^{\bb}.$
Thus,
\beq \int\limits_{\R^{\bb\setminus \{0\}}} \left(\frac{|Q_{\bb}(x)| }{|\tP_{n_3}(s)|}\right)x dx & = & \Pi_\bb  \left(\frac{\int_{\tP_{n_3}(s)} v dv}{|\tP_{n_3}(s)|}\right)\\ 
& = & 0.\lab{eq:int=0}\eeq
By Theorem~\ref{thm:prekopa}, $\frac{|Q_{\bb}(x)| }{|\tP_{n_3}(s)|}$ is a logconcave function of $x\in \Pi_\bb \tP_{n_3}(s),$ and 
 is a non-negative and integrable function of $x$ that integrates to $1.$ By the corollary to the Brunn-Minkowski inequality given in Corollary~\ref{cor:Brunn}, it follows that
\beqs \int\limits_{\R^{\bb\setminus \{0\}}} \left(\frac{|Q_{\bb}(x)|}{|\tP_{n_3}(s)|}\right)|Q_{\bb}(x)|^{\frac{1}{n_3^2 - |\bb|}} dx & \leq & \left|Q_{\bb}\left( \,\,\int\limits_{\R^{\bb\setminus \{0\}}} \left(\frac{|Q_{\bb}(x)| }{|\tP_{n_3}(s)|}\right)x dx\right)\right|^{\frac{1}{n_3^2 - |\bb|}}\\
& = & |Q_{\bb}(0)|^{\frac{1}{n_3^2 - |\bb|}}  \quad \text{by}\,\,(\ref{eq:int=0}).\eeqs
Therefore, 
\beqs \int\limits_{\Pi_{\bb} \tP_{n_3}(s)} |Q_{\bb}(x)|^{1 + \frac{1}{n_3^2 - |\bb|}} \left(\frac{dx}{|\Pi_\bb\tP_{n_3}(s)|}\right) \leq \left(\frac{| \tP_{n_3}(s)|}{|\Pi_{\bb} \tP_{n_3}(s)|}\right)|Q_{\bb}(0)|^{\frac{1}{n_3^2 - |\bb|}}.\eeqs
%Let us rescale all these polytopes with  homothety of a fixed scale such that $\frac{| \tP_{n_3}(s)|}{|\Pi_{\b} \tP_{n_3}(s)|} = 1.$
By the monotonic increase of $L^p$-norms as $p$ increases from $1$ to $\infty$,  for the probability measure $\mu(dx) = \frac{dx}{|\Pi_\bb\tP_{n_3}(s)|}$, we see that 
\beq \int\limits_{\Pi_{\bb}\tP_{n_3}(s)} |Q_{\bb}(x)|^{1 + \frac{1}{n_3^2 - |\bb|}} \frac{dx}{|\Pi_\bb \tP_{n_3}(s)|} & \geq &  \left(\int\limits_{\Pi_{\bb}\tP_{n_3}(s)} |Q_{\bb}(x)| \frac{dx}{|\Pi_\bb \tP_{n_3}(s)|}\right)^{1 + \frac{1}{n_3^2 - |\bb|}}\\ & = & \left(\frac{| \tP_{n_3}(s)|}{|\Pi_{\bb} \tP_{n_3}(s)|}\right)^{1 + \frac{1}{n_3^2 - |\bb|}}.\eeq
It follows that \beq |Q_{\bb}(0)| \geq \frac{| \tP_{n_3}(s)|}{|\Pi_{\bb} \tP_{n_3}(s)|}.\eeq
Suppose that $n_2 + 2 < n_3$.
Let $\rho_+:V(\T_{n_2+2}) \ra [n_2+2]^2\subseteq \Z^2$ be the unique map that satisfies $\phi_{0, n_2+2} \circ \rho_+ = id$ on $V(\T_{n_2+2})$. We embed $V(\T_{n_2+2})$ into $V(\T_{n_3})$ via  
  $\phi_{0, n_3} \circ \rho_+,$ and define $\tilde{\bb}$ to be $ V(\T_{n_3})\setminus (\phi_{0, n_3} \circ \rho_+(V(\T_{n_2+2}))).$ 
We observe that $|\tP_{n_2+2}(s(1 + \frac{2}{(n_2+2)^2}))|$ is greater or equal to $|Q_{\bb}(0)|(\frac{1}{(n_2+2)^2}))^{|\bb|-|\tilde{\bb}|},$ since $\phi_{0, n_3} \circ \rho_+,$ induces an isometric map from $Q_\bb(0) + [0, \frac{1}{(n_2+2)^2}]^{\bb\setminus\tilde{\bb}}$ into $\tP_{n_2+2}(s(1 + \frac{2}{(n_2+2)^2})).$
Thus,
\beqs |\tP_{n_2 + 2}(s)| & = &  \left(1 + \frac{2}{(n_2+2)^2}\right)^{-(n_2+2)^2+1}\left|\tP_{n_2+2}\left(s(1 + \frac{2}{(n_2+2)^2})\right)\right|\\
&  \geq & e^{-2} |Q_{\bb}(0)|\left(\frac{1}{(n_2+2)^2}\right)^{|\bb|-|\tilde{\bb}|}\\
&  \geq & \frac{e^{-2}| \tP_{n_3}(s)|(\frac{1}{(n_2+2)^2})^{|\bb|-|\tilde{\bb}|} }{|\Pi_{\bb} \tP_{n_3}(s)|}\\
& \geq & | \tP_{n_3}(s)| (Cn_3)^{-Cn_3(n_3-n_2)}.\eeqs
Noting that $\tP_{n_2+2}(s)$ contains a unit cube and hence has volume at least $1$, we see that 
\beq |\tP_{n_2 + 2}(s)|^{\frac{1}{(n_2 + 2)^2}} & \geq & |\tP_{n_2+2}(s)|^{\frac{1}{n_3^2}}\\
& \geq & | \tP_{n_3}(s)|^{\frac{1}{n_3^2}} (C n_3)^{-C(1-\frac{n_2}{n_3})}\\
& \geq &  | \tP_{n_3}(s)|^{\frac{1}{n_3^2}} \left(1 - \frac{C (n_3 - n_2) \ln n_3}{n_3}\right).\eeq
We have thus proved the following:
let $C < n_2 +2 < n_3$. Then, 
\beq |P_{n_2 + 2}(s)|^{\frac{1}{(n_2+2)^2}} \geq |P_{n_3}(s)|^{\frac{1}{n_3^2}}\left(1 - \frac{C (n_3 - n_2) \ln n_3}{n_3}\right).\eeq  
Relabeling $n_2 + 2$ by $n_2$ yields the lemma.
\end{proof}

%$$  \De_w^{\frac{1}{m}} =\frac{|L|}{12} \exp\left( \frac{12\mathbf{L}(1)}{\pi}\right) + o(1),$$
%where $\mathbf{L}$ is the Lobachevsky function,  $$\mathbf{L}(x) = - \int_0^x \log(2\sin(t)) dt.$$
%The right hand side above is less or equal to (4 + o(1)).
%Now (\ref{eq:2.29}) tells us that $$|L| \geq (4 + o(1))^{-1} | \De_w|^{m^{-1}} =  {4}^{-1}\exp\left(\frac{6G}{\pi}\right) - o(1).$$
%This in turn tells us that if $s_0 = s_1 = s_2$ are such that $w_0^{(n)} = w_1^{(n)} = w_2^{(n)}  = 2$, then 

%While the above bound is not as tight as the one that follows, we have included it to emphasize the role of $ \De_w$.
%\end{comment}
We will need the notion of differential entropy (see page 243 of \cite{Cover}).
\begin{defn}[Differential entropy]
Let ${\displaystyle X}$ be a random variable supported on a finite dimensional Euclidean space $\R^m$,
associated with a measure $\mu$ that is absolutely continuous with respect to the Lebesgue measure. Let the Radon-Nikodym derivative of $\mu$ with respect to the Lebesgue measure be denoted $f$. The differential entropy of $X$, denoted ${\displaystyle h(X)}$ (which by overload of notation, we shall also refer to as the differential entropy of $f$, i.e. $h(f)$), is defined  as
${\displaystyle h(X)=-\int _{\R^m}f(x)\ln f(x)\,dx}$.
\end{defn}

The following Lemma is well known, but we include a proof for the reader's convenience.
\begin{lemma}\lab{lem:5}
The differential entropy of a  mean $1$ distribution with a bounded Radon-Nikodym derivative with respect to the Lebesgue measure, supported on $[0, \infty)$ is less or equal to $1$, and equality is achieved on the exponential distribution.
\end{lemma}
\begin{proof}
Let $f:[0, \infty) \ra \R$ denote a density supported on the non-negative reals, whose associated distribution $F$ has mean $1$. Let $g:[0, \infty) \ra \R$ be given by $g(x) := e^{-x}$. The relative entropy between $f$ and $g$ is given by 

\beq \partial(f||g) := \int_{[0, \infty)} f(x) \ln\left(\frac{f(x)}{g(x)}\right)dx,\eeq and can be shown to be non-negative for all densities $f$ using Jensen's inequality.
We observe that \beq \partial(f||g) & = & -h(f) +  \int_{[0, \infty)} f(x) \ln\left({e^{x}}\right)dx\\
                                          & = & - h(f) + 1, \eeq because $F$ has mean $1$.
This implies that $h(f) \leq 1 = h(g)$.
\end{proof}

\begin{lemma}\lab{lem:6}
If $2 = \min_i s_i$, \beqs |P_{n}(s)| \leq \exp\left((n-1)^2 (1 + \ln 2) + 2(n-1) \ln(C(s) n)\right).\eeqs
If one of the $s_i$ is $0,$ then the dimension of $\tP_n(s)$ is less or equal to $2n-2.$
\end{lemma}
\begin{proof}
%Define the map $\phi$ to be the restriction of $\De_1$ to squares that are contained in $\that takes $v \in \R^{[n]^2}$ to $\phi(v) \in \R^{[n-1]^2}$, where  $(\phi(v))(p, q)$ equals $v(p, q) - v(p, q-1)  - v(p-1, q) + v(p-1, q-1)$.  
In the first case, without loss of generality, we assume that $s_1 = 2$.
 Consider the matrix $\nabla_x:\R^{V(\T_n)} \ra \R^{V(\T_n)}$ that is given by $\nabla_x f(p, q):= f(p, q) - f(p-1, q)$ and $\nabla_y:\R^{V(\T_n)} \ra \R^{V(\T_n)}$ that is given by $\nabla_y f(p, q):= f(p, q) - f(p, q-1).$ 
 We denote the list of values of $\De_1$ on squares of the form $\{(p, q), (p, q-1), (p-1, q), (p-1, q-1)\}$ for $(p, q) \in [n] \times [n]$ by $\De_1|_{e \subseteq E_1 \cap \Box_n}$ Consider the map $\tilde{\Phi}: \R^{V(\T_n)\setminus \{(0, 0)\}}\ra  \R^{[n-1]^2\oplus[n-1] \oplus[n-1]}$ given by $$\tilde{\Phi} = \left(\De_1|_{ E_1 \ni e \subseteq \Box_n}, \nabla_x|_{[n-1]\times \{0\}}, \nabla_y|_{\{0\}\times [n-1]}\right).$$ 
 \begin{claim}\lab{cl:36}
  The determinant of the matrix for the map $\tilde{\Phi}$ with respect to the canonical bases for the domain and the range is either $1$ or $-1.$
 \end{claim}
 \begin{proof}
 For any $f \in span(\tP_n(s)),$ and $(p, q) \in \Box_n,$ we have $$f(p, q) =  f(p, 0) + f(0, q) - \sum_{E_1 \ni e \subseteq [p]\times [q]} (\De_1 f)(e).$$ Note that $f(p, 0)$ and $f(0, q)$ can in turn be expressed as sums of the values of $\nabla_x f$ and $\nabla_y f$ along the $x$ and $y$ axis respectively and the coefficients are all equal to $1$. Thus, the matrix for the map $\tilde{\Phi}$ with respect to the canonical bases for the domain and the range is invertible on $span(\tP_n(s))$ and has an inverse that is a matrix with integer entries. Therefore, we have proved the claim.
\end{proof}
 Next suppose $f \in \tP_n(s).$
As in the proof of Lemma~\ref{lem:3-}, we see that the slope of $f$ increases by no more than $s_0 + s_1$ in a horizontal stretch, due to the condition $(\Delta_0 + \Delta_1)(f) \leq s_0 + s_1,$ and slope of $f$ increases by no more than $s_1 + s_2$ in a vertical stretch, due to the condition $(\Delta_1 + \Delta_2)(f) \leq s_1 + s_2.$ Because these points are on a torus, we know that $\|\nabla_x f\|_\infty \leq C(s)n$ and $\|\nabla_y f\|_\infty < C(s) n.$ 
$|\tP_n(s)| \geq |P_n(s)|$ by Lemma~\ref{lem:28}, so together with Claim~\ref{cl:36}, $$|\tilde{\Phi} \tP_n(s)|= |\tP_n(s)| \geq |P_n(s)| .$$ 
As noted in the proof of Lemma~\ref{lem:28}, the map $T:\R^{V(\T_n)} \ra \R^{V(\T_n)}$ where  $T: g \mapsto g - g(0) \one,$ acts as identity on every 
point in the $n^2-1$ dimensional subspace $\{g \in \R^{V(\T_n)}|g(0) = 0\}.$ We see that $T$ maps $P_n(s)$ on to $\tP_n(s).$
Let $f$ be sampled uniformly at random from the $n^2 -1$-dimensional Lebesgue measure on $P_n(s).$
Let $\tilde{f}$ be sampled uniformly at random from the $n^2 -1$-dimensional Lebesgue measure on $\tilde{\Phi} \tP_n(s).$ Thus the distribution of $\tilde{f}$ is exactly the same 
as that of $\tilde{\Phi}T f.$
 By the formula for $T$ above, the distribution of each coordinate of $\De_1 f$ is exactly the same as the distribution of each coordinate of $\De_1 Tf,$ which is a density whose mean is $0$ and support is bounded above by $s_1.$ However, for $i \in \{0, 1, 2\}$ it is clear that  the distribution of each coordinate of $\De_i f$ is exactly the same as the distribution of each coordinate of $\De_i Tf,$ which is a density whose mean is $0.$  Now $\ln |\tilde{\Phi} \tP_n(s)|$ is the differential entropy of $\tilde{f}$ with respect to the Lebesgue measure. Since the joint differential entropy of a vector valued random variable, is less or equal to the sum of the differential entropies of its one dimensional marginals (see Corollary 10.34 of \cite{Yeung}), 
 \beqs \ln |\tilde{\Phi} \tP_n(s)| \leq (n-1)^2 (1 + \ln 2) + 2(n-1)(1 + 2 \ln(C(s) n)). \eeqs
Finally,  for the second case in the statement of the lemma, if one of the $s_i$ is zero, which without loss of generality we take to be $s_1$, then
for any $f \in span(\tP_n(s)),$  we have $\sum_{E_1 \ni e \subseteq [n]\times [n]} (\De_1 f)(e) = 0$ and $(\De_1 f)(e) \leq 0$ for each $e \in E_1,$ which implies that $(\De_1 f)(e) = 0$ for each $e \in E_1.$ Thus, $f(p, q) =  f(p, 0) + f(0, q) - \sum_{E_1 \ni e \subseteq [p]\times [q]} (\De_1 f)(e) = f(p, 0) + f(0, q),$  for $(p, q) \in \Box_n,$ implying that $f$ is constrained to lie in a subspace of dimension $2n-2.$
\end{proof}

We will use the lemmas in this section to prove Lemma~\ref{lem:2.8}.

\begin{proof}[Proof of Lemma~\ref{lem:2.8}]
By Lemma~\ref{lem:28} and Lemma~\ref{lem:3} and Lemma~\ref{lem:6}, \beq 1 \leq  \liminf\limits_{n \ra \infty} |P_n(s)|^{\frac{1}{n^2}} \leq \limsup\limits_{n \ra \infty} |P_n(s)|^{\frac{1}{n^2}} \leq 2e.\eeq

Let $C < n_1^2  \leq n_2.$ Let $n_3 = (\lfloor \frac{n_2}{n_1}\rfloor + 1)  n_1.$ 
By Lemma~\ref{lem:28} and Lemma~\ref{lem:3} and Lemma~\ref{lem:4},
\beqs |P_{n_1}(s)|^{\frac{1}{n_1^2}} & \leq &  |P_{n_3}(s)|^{\frac{1}{n_3^2}}\left(1 + \frac{C\log n_1}{n_1}\right)\\
& \leq & |P_{n_2}(s)|^{\frac{1}{n_2^2}}\left(1 - \frac{C (n_3 - n_2) \ln n_3}{n_3}\right)^{-1}\left(1 + \frac{C\log n_1}{n_1}\right)\\
& \leq & |P_{n_2}(s)|^{\frac{1}{n_2^2}}\left(1 - \frac{C n_1 \ln n_3}{n_3}\right)^{-1}\left(1 + \frac{C\log n_1}{n_1}\right)\\ 
& \leq & |P_{n_2}(s)|^{\frac{1}{n_2^2}}\left(1 - Cn_1 \left(\frac{\ln n_1^2}{n_1^2}\right)\right)^{-1}\left(1 + \frac{C\log n_1}{n_1}\right).\eeqs
This implies that \beqs |P_{n_2}(s)|^{\frac{1}{n_2^2}} \geq |P_{n_1}(s)|^{\frac{1}{n_1^2}}\left(1 - \frac{C\log n_1}{n_1}\right).\eeqs
As a consequence, \beqs \left(1 + \frac{C\log n_1}{n_1}\right)\liminf\limits_{n_2 \ra \infty} |P_{n_2}(s)|^{\frac{1}{n_2^2}} \geq |P_{n_1}(s)|^{\frac{1}{n_1^2}}. \eeqs Finally, this gives 
\beqs \liminf\limits_{n_2 \ra \infty} |P_{n_2}(s)|^{\frac{1}{n_2^2}} \geq \limsup\limits_{n_1 \ra \infty}|P_{n_1}(s)|^{\frac{1}{n_1^2}}, \eeqs implying 
\beqs 1 \leq  \liminf\limits_{n \ra \infty} |P_n(s)|^{\frac{1}{n^2}} = \lim\limits_{n \ra \infty} |P_n(s)|^{\frac{1}{n^2}} = \limsup\limits_{n \ra \infty} |P_n(s)|^{\frac{1}{n^2}} \leq 2e.\eeqs
\end{proof}

\begin{observation}\lab{obs:fn}
By the Brunn-Minkowski inequality, $\f_n(s) := (\vol_{n^2-1} P_n(s))^{\frac{1}{n^2-1}}$ is a positive Lipschitz concave function of $s$ if all the $s_i$ are nonzero, and for some $s_i = 0$, provided $n^2 -1 > 2n-2$ as is the case for $n \geq 2,$ $\f_n(s) = 0.$ 
\end{observation}

We thus conclude the following (part of which is included in the statement of Corollary 3.10 in \cite{random_concave}.) 
\begin{corollary}\lab{cor:2.10}
As $n \ra \infty$, the pointwise limit of the functions $\mathbf{f}_n$ is a function  $\mathbf{f}$ defined on $\R^3_{\geq 0}$ that is concave. If one of the $s_i$ is $0$ then for all $n$, $\f_n(s) = 0.$
\end{corollary}
Also, each $\f_n$ is a Lipschitz function of $s$ (which can be seen from its linearity along rays and its concavity as shown in Corollary~\ref{cor:2.10}, we have the following.
\begin{corollary}\lab{cor:last}
As $n \ra \infty$, the convergence of the functions $\mathbf{f}_n$ to  $\f$ is uniform on compact subsets of $\R^3_{\geq 0}$. 

Also, as $n \ra \infty$, the convergence of the functions $\frac{\mathbf{f}_n}{\f}$ to  $1$ is uniform on compact subsets of $\R^3_{> 0}$.
\end{corollary}

\begin{defn}[Surface tension $\sigma$] \lab{def:sigma}
We define $\sigma(s) :=  -\ln \f(s)$ on $\R^3_+$.
\end{defn}
 We will refer to this quantity as the surface tension, in agreement with its usage in the statistical physics literature.
 
%\underline{Note on constants:} 
%We will denote  universal constants by $C$ and $c$.

\section{Uniform integrability and submartingale convergence}

\begin{defn}
Let $h_* \in \bigcup\limits_{\nu'\in \overline{B}_\I(\nu, \eps)}H(\la, \mu; \nu')$ satisfy \beq\inf\limits_T \min\limits_{0 \leq i \leq 2}  ((-1)D_i h_*)_{ac} > 0.\lab{eq:39Di}\eeq
Let $\F_a \subset \F$ be the $\sigma-$field generated by all possible events  $\kappa$ where $\kappa$ belongs to the set $\D_a$ of all dyadic triangles or dyadic squares of side length $2^{-a}$ contained in $T$. Let $\tX = (-\hess h_*)_{ac}(\omega)$ and $\tX_a = \E[\tX | \F_a].$  Let $\tY_a(\omega) = (2|\kappa|)^{-1}(-\hess h_*)_{sing}(\kappa),$ where the  $\kappa$ are half-open (in the sense that $\kappa$ has the form  $[x_0, x_1) \times (y_0, y_1]$ if it is a square and if it is a dyadic triangle, it includes the three sides) so that $\D_a$ is a partition of $T$ and each $\omega $ uniquely determines a dyadic triangle or square $\kappa \in \D_a$, and $|\kappa|$ denotes $\Leb_2(\kappa).$
\end{defn}
Let $\tr(\hess h_*)$ denote the Radon measure given by   $\sum_{i \in \{0, 1, 2\}} (D_ih_*)(x).$

The goal of this section is to prove in Theorem~\ref{cor:24} that \beqs\lim_{a \ra \infty} \exp\left(-\sum_{\kappa \in \D_a}|\kappa| \sigma\left((-1)|\kappa|^{-1}\int_{\kappa}\hess h_*(dx)\right)\right) = \\
 %\nonumber \inf_{a \ra \infty} \exp\left(-\sum_{\kappa \in \D_a}|\kappa| \sigma\left((-1)|\kappa|^{-1}\int_{\kappa}\hess h_*(dx)\right)\right) =\\ 
\nonumber \exp\left(-\int_T \sigma((-1)(\hess h_*(x))_{ac})\Leb_2(dx)\right).\eeqs  Here $|\kappa|:= \Leb_2(\kappa).$ The main tool used is the uniform integrability of a certain submartingale:
namely that of $\{\sigma(\tX_a + \tY_a)\}_{a\geq 1}$ below.

We now proceed with the details.

Let $\Leb_2$ denote the Lebesgue measure in $\R^2$. 
The Hessian $\hess h_*$ of a concave function $h_*$ is a matrix valued Radon measure \cite{dudley}, which by the Lebesgue decomposition Theorem can be decomposed as  \beq \hess h_* = (\hess h_*)_{ac} + (\hess h_*)_{sing},\eeq where $(\hess h_*)_{ac},$ is absolutely continuous with respect to two dimensional Lebesgue measure and 
$(\hess h_*)_{sing}$ is singular with respect to the same.
 We start with $T_{(0, 1), (0, 0), (1, 1)} = T$. 
Given a dyadic triangle $T_{D,E,F}$ with vertices $D, E$ and $F$  we define its three children to be the triangles  $T_{\frac{E + F}{2}, \frac{F + D}{2}, F}$,
$T_{D, \frac{F + D}{2}, \frac{D + E}{2}}$, and the square $Q_{\frac{E + F}{2}, E, \frac{D + E}{2}, \frac{D + F}{2}}$.
Given a dyadic square $Q'$, we define its four children by subdivision into four congruent squares of half the side length.

\begin{defn}[Uniform integrability] A collection of random variables $X_i, i \in I,$ is said to be uniformly integrable if
$$\lim\limits_{M \ra \infty}\left(\sup\limits_{i \in I} \E(|X_i| \one_{|X_i| > M})\right) = 0.$$
\end{defn}
\begin{lemma}[Theorem~4.6.1, \cite{Durrett-v5}]\lab{lem:UI}
Given a probability space $(\Omega,\mathcal{F},\PP)$ and a random variable $X \in L^1,$ the collection $\{\mathbb{E}[X |G ]\big | G  \subset \mathcal{F}\}$, is uniformly integrable.
\end{lemma}
The following lemma is due to de la Vall\'{e}e-Poussin, and may be found in  (p.19, Theorem T22,  \cite{Mayer}) or \cite{Poussin}.
\begin{lemma}[de la Vall\'{e}e-Poussin]\lab{lem:24}
The family $\{X_a\}_{a \in A}$ is uniformly integrable if and only if there exists a non-negative increasing convex function $G(t)$ such that $\lim\limits_{t \ra \infty} \frac{G(t)}{t} = \infty$ and $\sup\limits_a \E[G(|X_a|)] < \infty.$
\end{lemma}
\begin{lemma}[Theorem~4.6.4  \cite{Durrett-v5}]\lab{lem:27}
For a submartingale, the following are equivalent.
\ben 
\item  It is uniformly integrable.
\item It converges almost surely and in $L^1$.
\item It converges in $L^1$.
\een
\end{lemma}

\begin{lemma}[Theorem~4.6.8, \cite{Durrett-v5}]\lab{lem:4.6.8}
 Suppose $\{\F_a\}_{a\geq 0}$ is an increasing sequence of $\sigma-$fields and $\F_\infty$ is the $\sigma-$field generated by $\bigcup_a\F_a$. Let $X \in L^1$. Then, as $a \ra \infty$,
$\E(X|\F_a) \ra \E(X|\F_\infty)$ almost surely and in $L^1.$
\end{lemma}

Consider the probability space $(\Omega, \mathcal{F}, \PP),$ where $\Omega = T$, $\mathcal{F}$ is the collection of Lebesgue measurable subsets of $T$ and $\PP$ equals $2\Leb_2$ (\ie $\Leb_2$ restricted to $T$ and normalized to be a probability measure). 
Let $\tr((\hess h_*)_{ac})$ denote the part of the Radon measure given by   $\sum_{i \in \{0, 1, 2\}} (D_ih_*)(x)$ that is absolutely continuous with respect to the Lebesgue measure in the Lebesgue decomposition.

\begin{lemma}\lab{lem:26}
Let $h_* \in \bigcup\limits_{\nu'\in \overline{B}_\I(\nu, \eps)}H(\la, \mu; \nu').$ %satisfy \beq \inf\limits_T \min\limits_{0 \leq i \leq 2}  ((-1)D_i h_*)_{ac} > 0. \lab{eq:39Di}\eeq
Then, $\tr((\hess h_*)_{ac})$ belongs to $L^1.$
\end{lemma}
\begin{proof}
Since $h_*$ is Lipschitz and rhombus concave, $\tr(\hess h_*)$ is a (nonpositive valued) Radon measure, and $\int_T \tr(\hess h_*)$ is finite. Therefore $\tr((\hess h_*)_{ac}) \in L^1.$  
\end{proof}
%\begin{theorem}[Martingale Convergence Theorem]\lab{thm:MCT}
%Let $\{Z_n\}_{n\geq 0}$ be a submartingale such that $\sup_n \E|Z_n| < \infty.$ Then with probability one the limit
 %$\lim\limits_{n \ra \infty} Z_n := Z_\infty $
%exists, is finite,  and has finite first moment.
%\end{theorem}
Recall that for $s \in \R^3_+,$ $$\sigma(s) = - \ln \left(\lim\limits_{n \ra \infty} (\vol_{n^2-1} P_n(s))^{\frac{1}{n^2-1}}\right).$$
\begin{lemma}Let $h_* \in \bigcup\limits_{\nu'\in \overline{B}_\I(\nu, \eps)}H(\la, \mu; \nu')$ satisfy $$\inf\limits_T \min\limits_{0 \leq i \leq 2}  ((-1)D_i h_*)_{ac} > 0.$$
As $a \ra \infty$, 
 $\sigma(\tX_a + \tY_a)$ converges almost surely and in $L^1$ to $\sigma(\tX).$
\end{lemma}
\begin{proof} We see that $$\E[\tX_{a+1} + \tY_{a+1}|\F_a] = \tX_a + \tY_a,  $$ and so $\{\tX_a + \tY_a\}_{a\geq 1}$ is a martingale.
Since $\sigma$ is convex, this implies that 
$$\E[\sigma(\tX_{a+1} + \tY_{a+1})|\F_a] \geq \sigma(\tX_a + \tY_a),$$ and so 
$\{\sigma(\tX_a + \tY_a)\}_{a\geq 1}$ is a submartingale. Similarly $\{\sigma(\tX_a)\}_{a\geq 1}$ is a submartingale.
We will now prove that $\{\sigma(\tX_a + \tY_a)\}_{a\geq 1}$ and $\{\sigma(\tX_a)\}_{a\geq 1}$ are uniformly integrable. Taking $G(t) := \exp(t)$ and applying Lemma~\ref{lem:24}, we see that it suffices to show that $\sup_a \E\exp\{|\sigma(\tX_a + \tY_a)|\}_{a \geq 1}<\infty$ and $\sup_a \E\exp\{|\sigma(\tX_a)|\}_{a \geq 1}<\infty$. 
By Lemma~\ref{lem:2.8},  for all  $x \in T$, \beq \log (\min_i (- D_i h_*)_{ac}(x))  & \leq & - \sigma((-1)(\hess h_*)_{ac}(x)) + \log 2\\ &<& \log \left(4e\right) + \log (\tr((-1)\hess h_*)_{ac}(x)).\lab{eq:5.5}\eeq Also $\sigma(\tX_a + \tY_a) \leq \sigma(\tX_a)$. Therefore, in order to show that $\sup_a \E\exp\{|\sigma(\tX_a + \tY_a)|\}_{a \geq 1}<\infty$ and $\sup_a \E\exp\{|\sigma(\tX_a)|\}_{a \geq 1}<\infty$,  by (\ref{eq:39Di}) it suffices to show that
$$\sup_a \E\exp(-\sigma(\tX_a + \tY_a)) < \infty.$$ Note that $\exp(-\sigma)$ is a concave function by Corollary~\ref{cor:2.10} and  $\exp(-\sigma(\E(\tX_a + \tY_a)))$ is independent of $a$. Thus 
 $$\sup_a \E\exp(-\sigma(\tX_a + \tY_a)) \leq \sup_a \exp(-\sigma(\E(\tX_a + \tY_a))) = \exp(-\sigma(\E(\tX_1 + \tY_1))) < \infty,$$ the last inequality holding because $h_*$ is Lipschitz and concave, and so $\int_T \hess h_*(dx) < \infty$. We have thus proved that $\{\sigma(\tX_a + \tY_a)\}_{a\geq 1}$ and $\{\sigma(\tX_a)\}_{a\geq 1}$ are uniformly integrable. Therefore $\{\sigma(\tX_a + \tY_a)\}_{a\geq 1}$ converges almost surely and in $L^1$. It remains to show that the limit is $\sigma(\tX).$ To do so,  note that
 $\tr((\hess h_*)_{ac}) \in L^1$ by Lemma~\ref{lem:26}, and so 
  by Lemma~\ref{lem:UI}, $\{\tr(\tX_a)\}_{a\geq 1}$ converges almost surely and in $L^1$ to $\tr(\tX)$. By (\ref{eq:5.5})  
  $\{\sigma(\tX_a)\}_{a\geq 1}$ converges in $L^1$, and hence by Lemma~\ref{lem:27}, almost surely as well.  By Lemma~\ref{lem:4.6.8}, the almost sure limit of $\{\sigma(\tX_a)\}_{a\geq 1}$ is $\sigma(\tX).$
\begin{claim}
$\{\tY_a\}_{a\geq 1}$ converges almost surely to $0$.
\end{claim}
\begin{proof}
  We recall that $(\hess h_*)_{sing}$ is a coordinate-wise nonpositive finite measure that is singular with respect to Lebesgue measure.  By the Lebesgue differentiation theorem (see Theorem 3.22,  page 99 of \cite{Folland})  $$\lim_{a \ra \infty}  \frac{(-1)(\hess h_*)_{sing}(\kappa_a(\omega))}{|\kappa_a|} = 0,$$ for (Lebesgue-)almost all points $\omega \in T$ where $\kappa_a(\omega)$ is the unique element of $\D_a$ that contains $\omega.$ 
Therefore,  $\{\tY_a\}_{a\geq 1}$ converges almost surely to $0$. 
\end{proof}
It follows that $\{\sigma(\tX_a + \tY_a) - \sigma(\tX_a)\}_{a\geq 1}$ converges almost surely to $0$. 
Therefore the almost sure limit of $\{\sigma(\tX_a + \tY_a)\}_{a\geq 1}$ is $\sigma(\tX)$. By Lemma~\ref{lem:27} the almost sure limit and the $L^1$ limit coincide. This proves the Lemma.
 \end{proof}
 \begin{defn}\lab{def:77}
Let $\tilde{\de} > 0.$
Let $H^{\tilde{\de}} = H^{\tilde{\de}}(\la, \mu)$ be the set of all  $h \in H(\la, \mu)$ such that  for any $n \in \N$ and any $i = 0, 1, 2$,  for any rhombus $e \in E_i(T_n),$
$$-\De_ih_n(e) \geq \frac{\tilde{\de}}{n^2}.$$
%$$\inf\limits_T \min\limits_{0 \leq i \leq 2}  ((-1)D_i h)_{ac} \geq \tilde{\de} > 0.$$
Let $\tau$ be the rhombus concave function on $T$ such that $\tau(x, y)  = - x^2 -  y^2 + xy + \ell,$ where $\ell$ is the unique linear function that ensures that $\tau(0, 0) = \tau(0, 1) = \tau(1, 1) = 0. $ Let $\la^{\tilde{\de}} := \la(y) - \tilde{\de}\partial_y \tau(x, y)|_{x = 0},$ and $\mu^{\tilde{\de}} := \mu(x) - \tilde{\de}\partial_x \tau(x, y)|_{y=1}.$
 \end{defn}
 \begin{lemma}\lab{lem:Hlm-nonempty}
 Suppose $\la^{\tilde{\de}}$ and $\mu^{\tilde{\de}}$ are strongly decreasing functions. Then, $H^{\tilde{\de}}(\la, \mu)$ is nonempty.
 \end{lemma}
 \begin{proof}
 Consider the hive $h: T \ra \R$ that belongs to $H(\la^{\tilde{\de}}, \mu^{\tilde{\de}})$ for which $h(x,y) = h(0, y) + h(x, 1).$ Then $h + \tilde{\de} \tau \in H^{\tilde{\de}}(\la, \mu)$, thus showing that $H^{\tilde{\de}}(\la, \mu)$ is nonempty.
 \end{proof}

 \begin{theorem}\lab{cor:24}
 For any fixed $h_* \in H(\la, \mu)$,
 %such that $$\inf\limits_T \min\limits_{0 \leq i \leq 2}  ((-1)D_i h_*)_{ac} > 0,$$
\beq \lab{eq:notJa} \lim_{a \ra \infty} \exp\left(-\sum_{\kappa \in \D_a}|\kappa| \sigma\left((-1)|\kappa|^{-1}\int_{\kappa}\hess h_*(dx)\right)\right) = \\
 %\nonumber \inf_{a \ra \infty} \exp\left(-\sum_{\kappa \in \D_a}|\kappa| \sigma\left((-1)|\kappa|^{-1}\int_{\kappa}\hess h_*(dx)\right)\right) =\\ 
\nonumber \exp\left(-\int_T \sigma((-1)(\hess h_*(x))_{ac})\Leb_2(dx)\right).\eeq  Here $|\kappa|:= \Leb_2(\kappa).$
 \end{theorem}
\begin{proof}
In the preceding discussion, we just proved the following.

For positive $\tilde{\de}$ and any fixed $h \in H^{\tilde{\de}}$,
\beqs \lim_{a \ra \infty} \exp\left(-\sum_{\kappa \in \D_a}|\kappa| \sigma\left((-1)|\kappa|^{-1}\int_{\kappa}\hess h(dx)\right)\right) = \\
\nonumber \inf_{a \in \N} \exp\left(-\sum_{\kappa \in \D_a}|\kappa| \sigma\left((-1)|\kappa|^{-1}\int_{\kappa}\hess h(dx)\right)\right) = \\ 
\nonumber \exp\left(-\int_T \sigma((-1)(\hess h(x))_{ac})\Leb_2(dx)\right).\eeqs  
It remains to remove the restriction that $$\inf\limits_T \min\limits_{0 \leq i \leq 2}  ((-1)D_i h)_{ac} > 0.$$

For any fixed $h_* \in H(\la, \mu)$,  for $t \in  \N$,  let $h_*^{(t)} := h_* + t^{-1} \tau \in H^{t^{-1}},$ where $\tau$ is defined in Definition~\ref{def:77}.
In (\ref{eq:monotone}) below, we have used the fact that for  $t_1, t_2 \in \N$, if $t_1 < t_2$, then 
$$ \sigma\left((-1)|\kappa|^{-1}\int_{\kappa}(-(t_1^{-1}, t_1^{-1}, t_1^{-1}) dx+\hess h_*(dx))\right) > \sigma\left((-1)|\kappa|^{-1}\int_{\kappa}(-(t_2^{-1}, t_2^{-1}, t_2^{-1}) dx+\hess h_*(dx))\right).$$
For any fixed $a,$ we have 
\beq  \nonumber \lim_{t \ra \infty} \exp\left(-\sum_{\kappa \in \D_a}|\kappa| \sigma\left((-1)|\kappa|^{-1}\int_{\kappa}\hess (h_*^{(t)})(dx)\right)\right) = \\
\nonumber \lim_{t \ra \infty} \exp\left(-\sum_{\kappa \in \D_a}|\kappa| \sigma\left((-1)|\kappa|^{-1}\int_{\kappa}(-(t^{-1}, t^{-1}, t^{-1})dx + \hess h_*(dx))\right)\right) = \\
\lab{eq:monotone} \inf_{t \in \N} \exp\left(-\sum_{\kappa \in \D_a}|\kappa| \sigma\left((-1)|\kappa|^{-1}\int_{\kappa}(-(t^{-1}, t^{-1}, t^{-1}) dx+\hess h_*(dx))\right)\right) = \\ \nonumber 
 \exp\left(-\sum_{\kappa \in \D_a}|\kappa| \sigma\left((-1)|\kappa|^{-1}\int_{\kappa}(\hess h_*(dx))\right)\right).\eeq

Therefore,
\beqs \lim_{a \ra \infty} \exp\left(-\sum_{\kappa \in \D_a}|\kappa| \sigma\left((-1)|\kappa|^{-1}\int_{\kappa}\hess h_*(dx)\right)\right) = \\
\nonumber \inf_{a \in \N} \exp\left(-\sum_{\kappa \in \D_a}|\kappa| \sigma\left((-1)|\kappa|^{-1}\int_{\kappa}\hess h_*(dx)\right)\right) =\\
\nonumber \inf_{a \in \N} \inf_{t \in \N} \exp\left(-\sum_{\kappa \in \D_a}|\kappa| \sigma\left((-1)|\kappa|^{-1}\int_{\kappa}(-(t^{-1}, t^{-1}, t^{-1})dx +\hess h_*(dx))\right)\right) = \\
\nonumber  \inf_{t \in \N}\inf_{a \in \N} \exp\left(-\sum_{\kappa \in \D_a}|\kappa| \sigma\left((-1)|\kappa|^{-1}\int_{\kappa}(-(t^{-1}, t^{-1}, t^{-1})dx +\hess h_*(dx))\right)\right) = \\
\nonumber  \inf_{t \in \N} \exp\left(-\int_T \sigma((-1)(-(t^{-1}, t^{-1}, t^{-1})dx + \hess h_*(x))_{ac})\Leb_2(dx)\right).\eeqs

We identify two cases:
\ben \item $\int_T \sigma((-1)(\hess h_*(x))_{ac})\Leb_2(dx) < +\infty.$\\
By Lebesgue's dominated convergence Theorem
\beqs  \inf_{t \in \N} \exp\left(-\int_T \sigma((-1)(-(t^{-1}, t^{-1}, t^{-1}) + \hess h_*(x))_{ac})\Leb_2(dx)\right) = \\ \exp\left(-\int_T \sigma((-1)(\hess h_*(x))_{ac})\Leb_2(dx)\right).
\eeqs

\item $\int_T \sigma((-1)(\hess h_*(x))_{ac})\Leb_2(dx) = +\infty.$\\
In this case,  
\beqs  \inf_{t \in \N} \exp\left(-\int_T \sigma((-1)(-(t^{-1}, t^{-1}, t^{-1}) + \hess h_*(x))_{ac})\Leb_2(dx)\right) = \\
\lim_{M \ra \infty} \inf_{t \in \N} \exp\left(-\int_T \min(M, \sigma((-1)(-(t^{-1}, t^{-1}, t^{-1}) + \hess h_*(x))_{ac})))\Leb_2(dx)\right).\eeqs
Again, by Lebesgue's dominated convergence Theorem, this equals
\beqs \lim_{M \ra \infty}  \exp\left(-\int_T \min(M, \sigma((-1)(\hess h_*(x))_{ac})))\Leb_2(dx)\right),\eeqs
which equals $0.$
\een
 This completes the proof of this Theorem.

%\noindent{ Case 1.}\\

%\noindent{ Case 2.}\\

\end{proof}
\begin{comment}
\begin{thm}
%Let $h \in H(\la, \mu; \nu)$ be a surface tension %minimizer. Then,
\beqs\limsup\limits_{n \ra \infty}\p_n\left[\spec(X_n + Y_n) \in B_\I^n(\nu_n, \eps n^2)\right]^\frac{2}{n^2} \leq
 \left(\frac{1}{V(\la)V(\mu)}\right)\sup\limits_{\tilde{\nu} \in \overline{B}_\I(\nu, \eps)}\sup\limits_{h \in H(\la, \mu; \tilde{\nu})} V(\tilde{\nu})\exp\left(-\int_T \sigma((-1)\hess h)\right).\eeqs
\end{thm}
\end{comment}

\section{Lower bound}

\begin{defn}
Consider the functional $$\J: \bigcup\limits_{\nu'} H(\la, \mu; \nu')\ra \R$$ given by 
$$\J(h) := V(\nu')\exp\left(-\int_T 2\sigma((-1)(\hess h(x))_{ac})\Leb_2(dx)\right)$$ for any $h \in H(\la, \mu; \nu').$
%Let the supremum of $\J$ on $ \bigcup\limits_{\nu'\in {B}_\I(\nu, \eps)}H(\la, \mu; \nu')$ be denoted $\f_H(\la, \mu; \nu, \eps)$. 
\end{defn}

The goal of this section is to prove in Lemma~\ref{lem:44-new} the following:\\ let $X_n$ and $Y_n$ be independent random Hermitian matrices with spectra $\la_n$ and $\mu_n$ respectively.  Let $Z_n = X_n + Y_n.$ Then,
$$\liminf\limits_{n \ra \infty}\p_n\left[\spec(Z_n) \in B_\I^n(\nu_n, \eps n^2)\right]^\frac{2}{n^2} \geq  \sup\limits_{\nu'\in {B}_\I(\nu, \eps)}\sup\limits_{h' \in H(\la, \mu; \nu')} \left(V(\la)V(\mu)\right)^{-1}\J(h').$$

The proof goes via Lemma~\ref{lem:41-new3}, in which crucial use is made of the Theorem of Fradelizi stated in Theorem~\ref{thm:frad}. This Theorem allows us to freeze the boundary of a quadratic hive, without losing much entropy.
We then patch together (in the $C^2$ case) many such near quadratic hives with frozen boundary conditions, and show that even such random haves have enough entropy.

We now proceed with the details.

\begin{defn}[Approximate identity]\lab{def:27}
Let $\theta:\R^2 \ra \R$ be a $C^3$ approximate identity supported on $[0, 1]^2$ given for $(x, y) \in [0, 1]^2$ by $$ \frac{x^4(1-x)^4y^4(1-y)^4}{\int\limits_{(x, y) \in [0, 1]^2}x^4(1-x)^4y^4(1-y)^4 dxdy}$$ and elsewhere by $0$,  and let $\theta_\eps(x, y) = \eps^{-2}\theta(\frac{x}{\eps}, \frac{y}{\eps}),$ for all $(x, y) \in \R^2.$
\end{defn}
\begin{comment}
\begin{lemma}\lab{lem:Vupper}
$$\frac{V_n(\nu_n)}{V_n(\tau_n)} \leq \frac{\left(\nu_n(1) - \nu_n(n)\right)^{{n \choose 2}}}{\prod_{k = 1}^{n-1} {k^k}}.$$
\end{lemma}
\begin{proof}
We see that 
\beqs \frac{V_n(\nu_n)}{V_n(\tau_n)} & = & \prod_{1 \leq i < j \leq n} \left(\frac{\nu_n(i) - \nu_n(j)}{j-i}\right)\\
%& = & \prod_{1 \leq i < j \leq n} (\nu_n(i) - \nu_n(j))\\
& = & \prod_{k = 1}^{n-1} \prod_{\substack{{1 \leq i < j \leq n}\\{|i-j| = k}}} \left(\frac{\nu_n(i) - \nu_n(j)}{j-i}\right).\eeqs
By the A.M-G.M inequality, 
\beqs 
\prod_{k = 1}^{n-1} \prod_{\substack{{1 \leq i < j \leq n}\\{|i-j| = k}}} \left(\frac{\nu_n(i) - \nu_n(j)}{j-i}\right)
& = & \prod_{k = 1}^{n-1} \prod_{\substack{{1 \leq i \leq n- k} \\{j = i+k}}} \left(\frac{\nu_n(i) - \nu_n(j)}{k}\right)\\
& \leq &  \prod_{k = 1}^{n-1} k^{-(n-k)}\left(\frac{\sum\limits_{1 \leq i \leq n- k}  \left(\nu_n(i) - \nu_n(i+k)\right)}{n-k}\right)^{n-k}\quad \text{by\,\, A.M-G.M}\\ 
& \leq &  \prod_{k = 1}^{n-1} k^{-(n-k)}\left(\frac{k \left(\nu_n(1) - \nu_n(n)\right)}{n-k}\right)^{n-k} \\ 
& = &  \prod_{k = 1}^{n-1} \left(\frac{ \nu_n(1) - \nu_n(n)}{n-k}\right)^{n-k}\\ 
& = & \frac{\left(\nu_n(1) - \nu_n(n)\right)^{{n \choose 2}}}{\prod_{k = 1}^{n-1} k^{k}}.\eeqs
\end{proof}
\end{comment}

\begin{lemma}\lab{lem:17}
$${V_n(\tau_n)} \geq  \exp\left({n \choose 2}\log n + n^2 \int_{0}^{1} x\log (1 - x)dx \right).$$
\end{lemma}
\begin{proof}
We see that 
\beqs {V_n(\tau_n)} & = & \prod_{1 \leq i < j \leq n} ({j-i})\\
%& = & \prod_{1 \leq i < j \leq n} (\nu_n(i) - \nu_n(j))\\
& = & \prod_{k = 1}^{n-1} \prod_{\substack{{1 \leq i < j \leq n}\\{|i-j| = k}}} ({j-i}).\eeqs
Now,
\beqs 
\prod_{k = 1}^{n-1} \prod_{\substack{{1 \leq i < j \leq n}\\{|i-j| = k}}} ({j-i})
& \geq & \prod_{k = 1}^{n-1} (n-k)^k\\
& \geq & \exp\left({n \choose 2}\log n + n^2 \int_{0}^{1} x\log(1 - x)dx \right).\eeqs
\end{proof}
Recall from Section~\ref{sec:prelim} that $T_n$ denotes $(nT) \cap (\Z\times \Z),$ where $T$ is the (closed) convex hull of $(0, 0),(0, 1)$ and $(1, 1).$ For $i = 0, 1, 2$ $E_i(T_n)$ is defined in the manner analogous to $E_i(\T_n)$. Recall from the introduction that 
we suppose $\alpha, \beta$ are Lipschitz strongly concave functions from $[0, 1]$ to $\R$ and $\g$ is a concave function from $[0, 1]$ to $\R$, such that $\a(0) = \g(0) = 0$, and  $\a(1) = \b(0) = 0$ and $\b(1) = \g(1) = 0.$ 
  %Suppose also that the $\a, \b, \g$ can be extended to  $C^\infty$ functions on $\R$. 
%We also assume $\a$ and $\b$ to satisfy $\|\a\|_{C^3} \leq 1$ and $\|\b\|_{C^3} \leq 1.$ 
 Let $\la = \partial^- \a$, $\mu = \partial^- \b$ on $(0, 1]$ and $\nu = \partial^- \g,$ at all points of $(0, 1]$, where $\partial^-$ is the left derivative, which is monotonically decreasing.
Let $\la_n(i) := n^2(\a(\frac{i}{n})-\a(\frac{i-1}{n}))$, for $i \in [n]$, and similarly,
$\mu_n(i) := n^2(\b(\frac{i}{n})-\b(\frac{i-1}{n}))$, and 
$\nu_n(i) := n^2(\g(\frac{i}{n})-\g(\frac{i-1}{n}))$. 

\begin{lemma}\lab{lem:16} Given a bounded monotonically decreasing function $\la = \partial^-\a$, from $(0, 1]$ to $\R$, 
$$\lim_{n \ra \infty} \left(\frac{V_n(\la_n)}{V_n(\tau_n)}\right)^\frac{2}{n^2} = \exp\left(2\int_{T\setminus\{(t, t)|t \in [0, 1]\}}\log\left(\frac{\la(x) - \la(y)}{x-y}\right)dxdy\right).$$  
\end{lemma}
\begin{proof} We may assume that $\la$ is strictly decreasing, since otherwise, the limit is clearly $0$. %First suppose that $(-1) \hess \a$ is a continuous density on $[0, 1]$ taking values in $[\de, \de^{-1}]$ for some $\de \in (0, 1].$ 
Without loss of generality, after scaling, we may assume that $\la(0) - \la(1) = 1.$ 
Observe that $$\log \frac{V_n(\la_n)}{n^{n \choose 2}} = \sum_{1 \leq j < i \leq n} \log\left(\frac{\la_n(j) - \la_n(i)}{n}\right).$$
For $(x, y) \in T\setminus\{(t, t)|t \in [0, 1]\}$, let $$g(x, y) := \log\left({\la(x) - \la(y)}\right).$$ Let $$g_n:T\setminus\{(t, t)|t \in [0, 1]\} \ra \R,$$ be the function given by $$g_n(x, y) = \log\left(\frac{\la_n(\lceil n y \rceil) - \la_n(\lfloor n x \rfloor)}{n}\right).$$ Note that $g$ is nonpositive and the $g_n$ are nonpositive on $T\setminus\{(t, t)|t \in [0, 1]\}$ for any $n$.
 On every compact subset of $T\setminus\{(t, t)|t \in [0, 1]\},$ $g$ is bounded, and the $g_n$ are bounded. Also, for any $(x, y) \in T \setminus\{(t, t)|t \in [0, 1]\}$ and $n' > n$,  $g_n(x, y) \geq g_{n'}(x, y)$ and the $g_n$  converge pointwise almost everywhere to $g$ as $n \ra \infty$, because any monotonically decreasing function $\la$ on $(0, 1]$ is continuous almost everywhere. Suppose first that $$\int_{T\setminus\{(t, t)|t \in [0, 1]\}}g(x, y)dxdy > -\infty.$$ Then, the $L_1$ norm of the function $g$ restricted to an $\eps$ neighborhood of  $\{(t, t)|t \in [0, 1]\}$ tends to $0$ as $\eps$ tends to $0$. Therefore, by the dominated convergence theorem (\cite{Evans}, Theorem 1.19), $$\lim_{n \ra \infty}\int_{T\setminus\{(t, t)|t \in [0, 1]\}}g_n(x, y)dxdy = \int_{T\setminus\{(t, t)|t \in [0, 1]\}}g(x, y)dxdy.$$ Next suppose $\int_{T\setminus\{(t, t)|t \in [0, 1]\}}g(x, y)dxdy = -\infty$. Then, the integral of $g$ on $T\setminus \{(t_1, t_2)||t_1 - t_2| \geq \eps\}$  (which is a compact subset of $T \setminus\{(t, t)|t \in [0, 1]\}$) tends to $-\infty$ as $\eps$  tends to $0$. Since $g_n$ is nonpositive, this implies that $$\lim_{n \ra \infty} \int_{T\setminus\{(t, t)|t \in [0, 1]\}}g_n(x, y)dxdy = -\infty$$ as well.
 By Lemma~\ref{lem:17}, $\lim_{n \ra \infty} \frac{2}{n^2}\log \frac{V_n(\tau_n)}{n^{n \choose 2}}$ is bounded below by an absolute constant. 
 We see by comparing the numerator and denominator term by term that $\lim_{n \ra \infty} \frac{2}{n^2}\log \frac{V_n(\tau_n)}{n^{n \choose 2}}$ is bounded above by a universal constant.   
 
 Applying the above argument to $\la_n = \tau_n,$ we see that
 $$\lim_{n \ra \infty} \left(\frac{V_n(\la_n)}{V_n(\tau_n)}\right)^\frac{2}{n^2} = \exp\left(\lim_{n \ra \infty} \frac{2}{n^2}\log \frac{V_n(\la_n)}{n^{n \choose 2}} - \lim_{n \ra \infty} \frac{2}{n^2}\log \frac{V_n(\tau_n)}{n^{n \choose 2}}\right).$$
 The lemma follows.
\end{proof}
\begin{defn} \lab{def:32}
Given a bounded, monotonically decreasing function $\la = \partial^-\a$, from $(0, 1]$ to $\R$, let
$V(\la):= 
\lim_{n \ra \infty} \left(\frac{V_n(\la_n)}{V_n(\tau_n)}\right)^\frac{2}{n^2},$ which by Lemma~\ref{lem:16} equals  $$\exp\left(2\int_{T\setminus\{(t, t)|t \in [0, 1]\}}\log\left(\frac{\la(x) - \la(y)}{x-y}\right)dxdy\right).$$ 
\end{defn}

\begin{lemma}\lab{lem:10} Let $\a$ and $\beta$ be strongly concave Lipschitz functions from $[0, 1]$ to $\R$, taking the value $0$ at the endpoints. Let $\la = \partial^-\a$ and $\mu = \partial^-\beta$. Then, $\lim\limits_{n \ra \infty} \left(\frac{V_n(\tau_n)^2}{V_n(\la_n)V_n(\mu_n)}\right)^\frac{2}{n^2} = \left(\frac{1}{V(\la)V(\mu)}\right).$
\end{lemma}
\begin{proof}
This follows immediately from Lemma~\ref{lem:16}.
\end{proof}
\begin{defn} Given sequences $\la_n, \mu_n$ of length $n$, we define the $X_n(\la_n)$ and $Y_n(\mu_n)$ to be independent unitarily invariant random Hermitian matrices with spectra $\la_n$ and $\mu_n$ respectively.
\end{defn}
Denote $X_n(\la_n) + Y_n(\mu_n)$ by $Z_n$. 
Let $0 < \eps_1 < 1/16$. 
\begin{defn}\lab{def:heps1}
Given $h \in H(\la, \mu; \nu)$, let $h_{\eps_1}:T \ra \R$ denote the $C^3$ function defined as follows, where $a(x, y)$ is the unique affine function that results in the function $h_{\eps_1}$ taking the value $0$ at the corners of $T$:
\beqs h_{\eps_1}(x, y) := (h \star \theta_{\eps_1})((1 - 4\eps_1)x + \eps_1, (1 - 4\eps_1)y + 3\eps_1) - \eps_1(x^2 + y^2 - xy) + a(x,y).\eeqs
\end{defn}

\begin{defn}\lab{def:12} Define $\ta, \tb, \tg$ the be the boundary values of $h_{\eps_1}$ on the vertical, horizontal and diagonal side of $T$ respectively.
\end{defn}
Since $h$ is Lipschitz, so is $h_{\eps_1}$. 

\begin{lemma}\lab{lem:7.5} For any $\eps_0 < \eps$, if $\eps_1$ is sufficiently small, then  
for $\tilde{\la} = \partial(\tilde{\a})$, $\tilde{\mu} = \partial(\tilde{\b}),$ and $\tilde{\nu}=\partial(\tilde{\g})$ we have that $\tilde{\a} \in B_\infty(\a, \frac{\eps-\eps_0}{4})$, $\tilde{\b} \in B_\infty(\b, \frac{\eps-\eps_0}{4})$ and $\tilde{\g} \in B_\infty(\g, \frac{\eps - \eps_0}{2})$, are $C^3$, concave on $[0, 1]$ and equal $0$ on the endpoints $0$ and $1$. Further, $h_{\eps_1}$ is a $C^3$-hive.
\end{lemma}
\begin{proof}
The convolution of a $C^3$ function on $[0, \eps]\times [0, \eps]$ with a $C^0$ rhombus concave function defined on $T$ is $C^3$ and rhombus concave at all points where the convolution is well defined; namely those points in $T$  whose $\ell_\infty$ distance from the diagonal of $T$ is at least $\eps.$ The convex combination of a set of rhombus concave functions on $T$ is rhombus concave on $T$. The lemma follows.
\end{proof}
\begin{lemma}\lab{lem:8-}
For all $\tilde{\la} = \partial^-(\tilde{\a})$, $\tilde{\mu} = \partial^-(\tilde{\b}),$ and $\tilde{\nu}=\partial^-(\tilde{\g})$ such that $\tilde{\a} \in B_\infty(\a, \frac{\eps-\eps_0}{4})$, $\tilde{\b} \in B_\infty(\b, \frac{\eps-\eps_0}{4})$ and $\tilde{\g} \in B_\infty(\g, \frac{\eps - \eps_0}{2})$, are $C^3$, concave on $[0, 1]$ and equal $0$ on the endpoints $0$ and $1$, 
\beqs\p_n\left[\spec(X_n(\la_n) + Y_n(\mu_n)) \in B_\I^n(\nu_n, \eps n^2)\right] \geq \p_n\left[\spec(X_n(\tilde{\la}_n) + Y_n(\tilde{\mu}_n)) \in B_\I^n\left(\tilde{\nu}_n, \eps_0 n^2\right)\right].\eeqs
\end{lemma}
\begin{proof}
We couple the tuple of random variables $(X_n(\la_n), X_n(\tla_n))$ so that they have the same eigenvectors
 corresponding to the eigenvalues, when listed in decreasing order. We also couple the tuple of random variables $(Y_n(\mu_n), Y_n(\tmu_n))$ so that they have the same eigenvectors
 corresponding to the eigenvalues, when listed in decreasing order. By design, $n^{-2}\|\spec(X_n(\tla_n) - X_n(\la_n))\|_\I \leq \frac{\eps - \eps_0}{4},$ and $n^{-2}\|\spec(Y_n(\tmu_n) - Y_n(\mu_n))\|_\I \leq \frac{\eps - \eps_0}{4}.$ The Ky Fan inequalities, (see Equation (1.55), \cite{Tao_book}), state that, for each $k \in [n],$ for $n \times n$ Hermitian matrices $X$ and $Y,$ the sum of the top $k$ eigenvalues of $X + Y$ is less or equal to the sum of the top $k$ eigenvalues of $X$ added to the sum of the top $k$ eigenvalues of $Y.$  Consequently,
 \beqs n^{-2} \|\spec((X_n(\la_n) - X_n(\tla_n)) + (Y_n(\mu_n) - Y_n(\tmu_n))) \|_\I  & \leq &   n^{-2}\|\spec(X_n(\la_n) - X_n(\tla_n))\|_\I\\ & + & n^{-2}\|\spec(Y_n(\mu_n) - Y_n(\tmu_n))\|_\I\\
 & \leq & \frac{\eps - \eps_0}{2}.\eeqs
 Thus if $$\spec(X_n(\tilde{\la}_n) + Y_n(\tilde{\mu}_n)) \in B_\I^n\left(\tilde{\nu}_n, \eps_0 n^2\right),$$ then $$X_n(\la_n) + Y_n(\mu_n)) \in B_\I^n\left(\tnu_n, \frac{\eps_0 + \eps}{2} n^2\right) \subseteq B_\I^n\left(\nu_n,  \eps n^2\right),$$ and therefore \beqs\p_n\left[\spec(X_n(\la_n) + Y_n(\mu_n)) \in B_\I^n(\nu_n, \eps n^2)\right] \geq \p_n\left[\spec(X_n(\tilde{\la}_n) + Y_n(\tilde{\mu}_n)) \in B_\I^n\left(\tilde{\nu}_n, \eps_0 n^2\right)\right].\eeqs
\end{proof}

We will need the following lemma.
\begin{lemma}\lab{lem:8}
$$ \liminf\limits_{n \ra \infty}\p_n\left[\spec(Z_n) \in B_\I^n(\nu_n, \eps n^2)\right]^\frac{2}{n^2}\geq \left(\frac{1}{V(\tilde{\la})V(\tilde{\mu})}\right)\liminf\limits_{n \ra \infty} |G_n(\tla_n, \tmu_n; \tnu_n, \eps_0 n^2)|^\frac{2}{n^2}.$$
\end{lemma}
\begin{proof} 
 
Observe that by (\ref{eq:3.2}),
\beqs \p_n\left[\spec(X_n(\tilde{\la}_n) + Y_n(\tilde{\mu}_n)) \in B_\I^n\left(\tilde{\nu}_n, \eps_0 n^2\right)\right]
& = & \int\limits_{B_\I^n(\tilde{\nu}_n, \eps_0 n^2)}|H_n(\tilde{\la}_n, \tilde{\mu}_n; \nu'_n)|\left(\frac{V_n(\nu'_n)V_n(\tau_n)}{V_n(\tilde{\la}_n)V_n(\tilde{\mu}_n)}\right)\Leb_{n-1, 0}(d\nu'_n)\\
& = & \left(\frac{V_n(\tau_n)^2}{V_n(\tilde{\la}_n)V_n(\tilde{\mu}_n)}\right)\int\limits_{B_\I^n(\tilde{\nu}_n, \eps_0 n^2)}|A_n(\tilde{\la}_n, \tilde{\mu}_n; \tilde{\nu}_n)|\Leb_{n-1, 0}(d\nu'_n).
\eeqs
Therefore, by Lemma~\ref{lem:10} and (\ref{eq:G}), 
\beqs \liminf\limits_{n \ra \infty}\p_n\left[\spec(Z_n) \in B_\I^n(\nu_n, \eps n^2)\right]^\frac{2}{n^2}\geq \left(\frac{1}{V(\tilde{\la})V(\tilde{\mu})}\right)\liminf\limits_{n \ra \infty} |G_n(\tla_n, \tmu_n; \tnu_n, \eps_0 n^2)|^\frac{2}{n^2}.
\eeqs
\end{proof}

For a fixed $a \in \N$, we define the set of boundary nodes $\bb \subseteq V(T_n)$ as follows. 
We first define $\bb'$ to be the set of all nodes $(n x, n y) \in V(T_n)$ such that there exists an $i \in \Z$ such that $|x - i 2^{-a}| \leq \frac{4}{n}$ or there exists a $j \in \Z$ such that $|y - j 2^{-a}| \leq \frac{4}{n},$ or $|y - x| \leq \frac{4}{n}.$ Now $V(\T_n)\setminus \bb'$ consists of a set of disjoint lattice rectangles separated by boundaries with at least $4$ layers together with some boundary triangles which are also separated from the lattice rectangles by at least $4$ layers.  

\begin{lemma}\lab{lem:27.1}
There exists a set of boundary nodes $\bb'' \supseteq \bb'$,  that satisfies the following conditions.
\ben
\item $\bb''$ is contained in the set of all lattice points $(nx, ny)$ such that, either, there exists an $i \in \Z,$ $|x - i 2^{-a}| \leq \frac{5}{n}$, or there exists a $j \in \Z,$ such that $|y - j 2^{-a}| \leq \frac{5}{n},$ or $|y - x| \leq \frac{5}{n}.$
\item $V(T_n)\setminus \bb''$ is the union of a collection of disjoint lattice {\it squares} $\kappa_{lat}$ (not merely rectangles) or isosceles right triangles $\kappa'_{lat}$.
\item  The Minkowski sum of each lattice square $\kappa_{lat}$ with $\{(x, y)| \max(|x|, |y|) \leq 4/n\}$
is contained in a unique dyadic square $\kappa$ of side length $2^{-a}$.
\item The set $V(\T_n)\setminus \bb''$ consists of a set of disjoint lattice squares  and lattice triangles,
of side length in the interval $[2^{-a}n - 16,  2^{-a}n ]$.
\een
\end{lemma}
\begin{proof}
We will give a constructive proof. Augment $\bb'$ in the following way. Take a dyadic square $\kappa$. If the vertices in $n\kappa$ that are not in $\bb'$ form a rectangle with the number of rows exceeding the number of columns by exactly one or  the number of columns exceeding the number of rows by exactly one rather than a square, augment $\bb'$ by adding one column layer of vertices in $n\kappa$ to $\bb'$, or row layer respectively.
Do this sequentially for each dyadic square $\kappa.$
Since we add only one row or column to each dyadic square 1. is satisfied. We see that 2. is satisfied by construction. Even $\bb'$ satisfies 3., and so $\bb''$ does as well. Since we add at most one row or column of vertices to each dyadic square, 4. is satisfied as well.

The above process takes care of the lattice triangles as well, since they are automatically isosceles.
\end{proof}

\begin{defn}\lab{def:37}
We set $\bb$ to $\bb''$, (whose existence is guaranteed by the above lemma).   Given $\tla, \tmu, \tnu \in \Spec_n$, let $h \in H(\tla, \tmu; \tnu)$. We define $h_\bb \in \R^\bb$ to be the unique vector whose coordinate corresponding to $b \in \bb$ is given by the value that $n^2 h$ takes on $(1/n) b.$
\end{defn}

\begin{figure}
    \centering
        \includegraphics[width=2.5in]{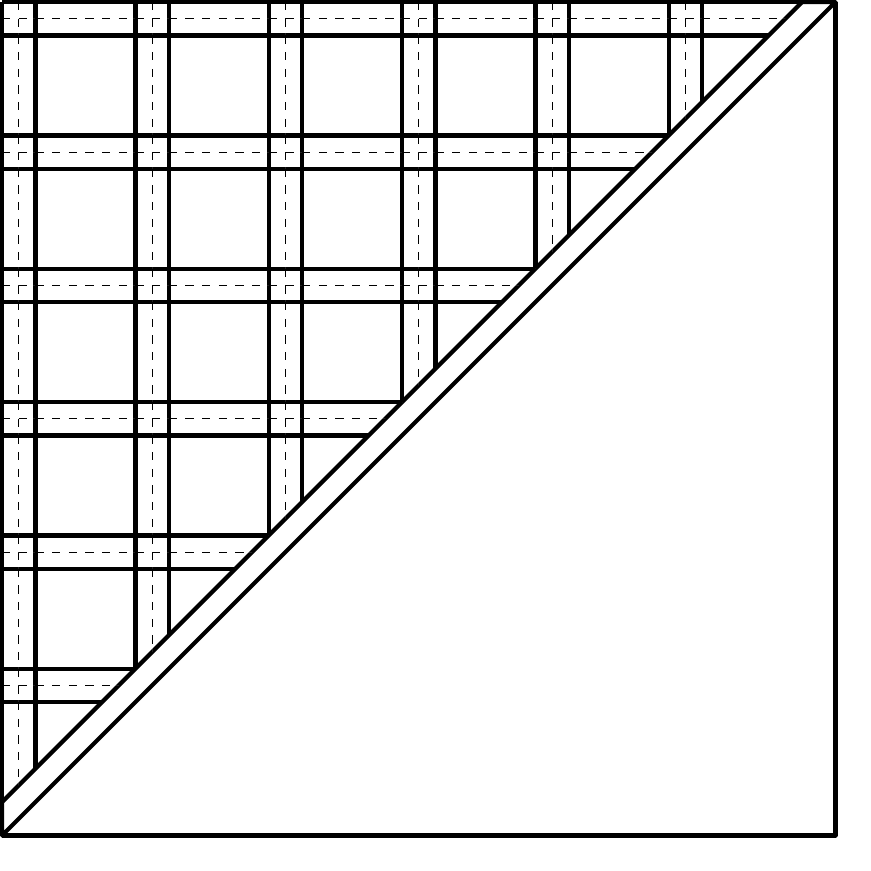} 
        \caption{Partition with multiple layers separating the squares and triangles.  Dashed lines indicate borders of dyadic pieces.} \label{fig:doublelayer}
    \end{figure}
    We will need the definition below as well.
\begin{defn}\lab{def:6.1} We define the polytope $\tilde{P}_n(\bb, h)$ as the preimage of $(h_\bb)$ under the coordinate projection $\Pi_\bb$ of $G_n(\tla_n, \tmu_n; \tnu_n, \eps_0 n^2)$ onto $\R^\bb$. We define the polytope $\tilde{Q}_n(\bb, h)$ as the preimage of $(h_\bb) + [-\frac{1}{2n^6}, \frac{1}{2n^6}]^\bb$ under the coordinate projection $\Pi_\bb$ of $G_n(\tla_n, \tmu_n; \tnu_n, \eps_0 n^2)$ onto $\R^\bb$.
\end{defn}
In the above definition, for large enough $n$, the role of $\frac{1}{2n^6}$ is simply to clamp down the values on $\bb$ so that they are essentially fixed, while at the same not clamping them down to the extent that they affect the asymptotics of the per vertex differential entropy.

\begin{lemma}\lab{lem:41-new3}
Let $\tih$ denote a $C^2$ hive, contained in $H(\tla, \tmu; \tnu)$ such that $$\inf_{x \in T} \min_i(-D_i\tih(x)) > 0.$$ Then, the volumes of the polytopes $\tilde{Q}_n(\bb, \tih)$ given by Definition~\ref{def:6.1} satisfy the following asymptotic lower bound as $n$ tends to $\infty$. Let $\eps_2$ be a  positive constant.
 Then, $$\liminf_{n\ra \infty} |\tilde{Q}_n(\bb, \tih)|^\frac{2}{n^2} \geq \exp(-\eps_2)V(\nu')\exp\left(\int_T 2\sigma(-D_0\tih(x), -D_1\tih(x), -D_2\tih(x))\Leb_2(dx)\right)$$   for all $a$ (which is implicit in the definition of $\bb$) larger than some number depending on $\tih$ and $\eps_2$.
\end{lemma}
\begin{proof}Note that $L_n(\tih) \in H_n(\tla_n, \tmu_n; \tnu_n)$. 
We will now consider two cases corresponding to dyadic squares and triangles respectively.

\noindent 1. {\it Let $\kappa$ be a dyadic square in $\D_a$.}\\
 All the lattice points in $n \kappa$ for a dyadic square $\kappa \in \D_a$ that are not in $\bb$ are separated from the lattice points in every abutting dyadic square or triangle scaled by $n$ by at least four layers of $\bb$. Let the subset of $V(T_n)\cap n\kappa$ of all lattice points at an $\ell_\infty$ distance of at most $2$ from the set $\left(V(T_n) \setminus \bb\right) \cap n\kappa$ be denoted $V_\kappa$. Let this correspond to a lattice square with $n_1 + 4$ vertices on each side (of which $n_1$ do not belong to $\bb$, and $4$ vertices do).
 Denote the projection of the translated polytope $\left(\tilde{P}_n(\bb,  \tih) - L_n( \tih)\right)$ onto $\R^{V_\kappa}$ by $Q_{n_1}$.
 %By Lemma~\ref{lem:7new},  
 
 For $i \in \{0, 1, 2\},$ let $\inf_{x \in \kappa} (-1)D_i \tih(x) = \de_i$. 
Let $\tilde{R}_{n_1}$ be the polytope consisting of the set of all functions $g:V(\T_{n_1+4}) \ra \R$ such that $\sum_{v \in V(\T_{n_1+4})} g(v) = 0$ and $\nabla^2(g)(e) \leq \de_i$  where  $e \in E_i(\T_{n_1+4})$. %Suppose that for each $i \in \{0, 1, 2\}$, we have $\sum_{e \in E_i(\T_{n_1+4})} t(e) = n^2 s_i$.
We identify the vertices of $V_\kappa$ with those of $\T_{n_1+4}$ by pasting together the sides of $V_\kappa$ in the natural way to transform it into a discrete torus.
Let $R_{n_1}$ be the polytope consisting of the set of all functions $g:V(\T_{n_1+4}) \ra \R$ such that for $v \in \bb$, $g(v) = 0$ and  $\sum_{v \in V(\T_{n_1+4})} g(v) = 0$ and $\nabla^2(g)(e) \leq \de_i$  where  $e \in E_i(\T_{n_1+4})$. 

Note  that  $$ R_{n_1} \subseteq \{x| \sum_i x_i = 0\}\cap Q_{n_1} \subseteq \Pi_{\{x| \sum_i x_i = 0\}} Q_{n_1}.$$

Further, for any $x\in (1 - n^{-6})R_{n_1}$, the line segment $\{x + [- n^{-C}, n^{-C}] \one_{\left(V(T_n) \setminus \bb\right) \cap n\kappa}\}$ is contained in $Q_{n_1}$.

This implies that for all sufficiently large $n$, \beq |Q_{n_1}| \geq n^{-C} (1 - n^{-6})^{\dim R_{n_1}} |R_{n_1}| . \eeq

 We see that $R_{n_1}$ is a section of $\tilde{R}_{n_1}$ and is the preimage of $0$ under the push forward of the uniform measure on $\tilde{R}_{n_1}$ under the evaluation map $\Pi_\bb$ that maps $g\in \tilde{R}_{n_1}$ on to $g_{\bb\cap V_\kappa}$. By Theorem~\ref{thm:frad}, and the logconcavity of the pushforward of the uniform measure on $\tilde{R}_{n_1}$ under the evaluation map $\Pi_\bb$ that follows from the convexity of $\tilde{R}_{n_1}$ and Theorem~\ref{thm:prekopa}, (noting also that the center of mass of $\tilde{R}_{n_1}$ is $0,$ by (\ref{eq:2.10})) we see that $$\frac{|R_{n_1}|}{|\tilde{R}_{n_1}|} \geq  \sup\limits_{g \in \tilde{R}_{n_1}}(2e\|g\|_\infty)^{-Cn_1 }.$$
 
However \beq \lim\limits_{n_1 \ra \infty} |\tilde{R}_{n_1}|^{\frac{1}{n_1^2}} = \exp(-\sigma(\de_0, \de_1, \de_2)),\eeq by Lemma~\ref{lem:2.8}.  

It follows that \beq\lab{eq:5.2}  \liminf\limits_{n_1 \ra \infty} |R_{n_1}|^{\frac{1}{n_1^2}} \geq \exp(-\sigma(\de_0, \de_1, \de_2)).\eeq 
\begin{comment}
Let the $\ell_2$ ball in $\R^{V(\T_{n_1+4})\cap \bb}$ of radius $\frac{1}{2n^6}$ centered at the origin be denoted $\tilde{B}_2.$
Then, for any $x \in \tilde{B}_2$, $$\Pi_\bb^{-1}(x) \cap \tilde{R}_{n_1} \supseteq \left(1 - \frac{1}{n}\right) R_{n_1} + x.$$
It follows that $$| \tilde{R}_{n_1}| \geq \left(1-\frac{1}{n}\right)^{n_1^2} |{R}_{n_1}| |\tilde{B}_2|.$$
Therefore, \beq\lab{eq:5.3}  \limsup\limits_{n_1 \ra \infty} |R_{n_1}|^{\frac{1}{n_1^2}} \leq \exp(-\sigma(\de_0, \de_1, \de_2)).\eeq 
By (\ref{eq:5.2}) and (\ref{eq:5.3}), \beq \lim\limits_{n_1 \ra \infty} |R_{n_1}|^{\frac{1}{n_1^2}} = \exp(-\sigma(\de_0, \de_1, \de_2)).\eeq 
\end{comment}
As $a \ra \infty$, the supremum over all $\kappa \in \D_a$ of $$\sup_{x \in \kappa} (-1)D_i \tih(x) - \inf_{x \in \kappa} (-1)D_i \tih(x),$$ tends to $0$ because $\tih$ is $C^2$ with a uniformly continuous Hessian. However because $\tih$ is strongly rhombus concave, $\inf_{x \in \kappa} (-1)D_i \tih(x)$ does not tend to $0$. As a consequence, as $a$ tends to infinity, the supremum over all $\kappa \in \D_a$ of 
$$\exp\left(-\sigma(\de_0, \de_1, \de_2) + |\kappa|^{-1}\int_\kappa \sigma(-D_0(x), -D_1(x), -D_2(x))\Leb_2(dx)\right)$$ tends to $0.$

\noindent 2. {\it Let $\kappa'$ be a dyadic triangle in $\D_a$.}\\
 All the lattice points in $n \kappa'$ for a dyadic triangle $\kappa' \in \D_a$ that are not in $\bb$  are separated from the lattice points in every abutting dyadic square or triangle scaled by $n$ by at least four layers of $\bb$. Let the subset of $V(T_n)\cap n\kappa'$ of all lattice points at an $\ell_\infty$ distance of at most $2$ from the set $\left(V(T_n) \setminus \bb\right) \cap n\kappa'$ be denoted $V_{\kappa'}$. Let this correspond to a lattice triangle with $n_1 + 4$ vertices on each side (of which $n_1$ do not belong to $\bb$, and $4$ vertices do).
 For $i \in \{0, 1, 2\},$ let $\inf_{x \in \kappa'} (-1)D_i \tih(x) = \de_i$. 
 %Suppose that for each $i \in \{0, 1, 2\}$, we have $\sum_{e \in E_i(\T_{n_1+4})} t(e) = n^2 s_i$.
Let $\kappa'' \cup \kappa'$ be a dyadic square $\kappa$, and let $\kappa'' \cap \kappa'  $ be the diagonal segment of $\kappa'$. This defines $\kappa''$ uniquely to be a dyadic triangle obtained by reflecting $\kappa'$ about its diagonal.
Let the subset of $V(\square_n)\cap n\kappa$ of all lattice points at an $\ell_\infty$ distance of at most $2$ from the set $\left(V(\square_n) \setminus \bb\right) \cap n\kappa$ be denoted $V_\kappa$.
We identify the vertices of $V_{\kappa}$ with those of $\T_{n_1+4}$ by pasting together the sides of $V_{\kappa}$ in the natural way to transform it into a discrete torus.

Let $\mathcal{R}$ denote the reflection of $\Box_n$ about the line $\{(x, y)|y = x\}.$

We augment $\bb \cap \kappa$ to $\tilde{\bb}$, by including in $\tilde{\bb}$ a band surrounding $n(\kappa'\cap \kappa'')$ (and symmetric about it) and some additional nodes in $n\kappa''$ that ensure that $\tilde{\bb} \cap (n\kappa)$ is symmetric under the action of $\mathcal{R}$. More precisely, $\bb''$ is defined to be all those vertices in $n\kappa$ that either equal a vertex $v$ or a vertex $\mathcal{R}(v)$  for $v \in \bb \cap \kappa'$. 

 Denote the projection of the polytope $$\left(\tilde{P}_n({\bb}, L_n( \tih)) - L_n( \tih)\right) \times_{\bb \cap \mathcal{R}(\bb)} \left(\tilde{P}_n({\bb}, L_n( \tih)) - L_n( \tih)\right),$$ (viewed as a polytope of functions from $V(\square_n)$ to $\R$ with the constraints satisfying reflection symmetry across the diagonal) onto $\R^{V_{\kappa}}$ by $Q'_{n_1}$.

 We include $\tilde{\bb}$ as a subset of $V(\T_{n_1 + 4})$ by embedding the $(n_1 + 4)\times (n_1+4)$  square into the torus in the natural way.
Let $R_{n_1}$ be the polytope consisting of the set of all functions $g:V(\T_{n_1+4}) \ra \R$ such that for $v \in \tilde{\bb}$, $g(v) = 0$ and  $\sum_{v \in V(\T_{n_1+4})} g(v) = 0$ and $\nabla^2(g)(e) \leq \de_i $  where  $e \in E_i(\T_{n_1+4})$.
%By construction,  $R_{n_1} \subseteq Q'_{n_1}$.

 Let $\tilde{R}_{n_1}$ be the polytope consisting of the set of all functions $g:V(\T_{n_1+4}) \ra \R$ such that $\sum_{v \in V(\T_{n_1+4})} g(v) = 0$ and $\nabla^2(g)(e) \leq \de_i $  where  $e \in E_i(\T_{n_1+4})$.

Note  that  $$ R_{n_1} \subseteq \{x| \sum_i x_i = 0\}\cap Q'_{n_1} \subseteq \Pi_{\{x| \sum_i x_i = 0\}} Q'_{n_1}.$$

Further, for any $x\in (1 - n^{-6})R_{n_1}$, the line segment $\{x + [-1/n^C, 1/n^C] \one_{\left(V(\square_n) \setminus \bb\right) \cap n\kappa}\}$ is contained in $Q'_{n_1}$.

This implies that for all sufficiently large $n$, \beq |Q'_{n_1}| \geq n^{-C} (1 - n^{-6})^{\dim R_{n_1}} |R_{n_1}| . \eeq

We see that $R_{n_1}$ is a section of $\tilde{R}_{n_1}$ and is the preimage of $0$ under the push forward of the uniform measure on $\tilde{R}_{n_1}$ under the evaluation map $\Pi_{\tilde{\bb}}$ that maps $g\in \tilde{R}_{n_1}$ on to $g_{\tilde{\bb}}$. By Theorem~\ref{thm:frad} due to Fradelizi, and the logconcavity of the pushforward of the uniform measure on $\tilde{R}_{n_1}$ under the evaluation map $\Pi_\bb$ that follows from  Theorem~\ref{thm:prekopa}, we see that $$\frac{|R_{n_1}|}{|\tilde{R}_{n_1}|} \geq  \sup\limits_{g \in \tilde{R}_{n_1}}(2e\|g\|_\infty)^{-(Cn_1)}.$$
However \beq \lim\limits_{n_1 \ra \infty} |\tilde{R}_{n_1}|^{\frac{1}{n_1^2}} = \exp(-\sigma(\de_0, \de_1, \de_2)),\eeq by Theorem~\ref{lem:2.8}. 

\begin{comment}
Let the $\ell_2$ ball in $\R^{V(\T_{n_1+4})\cap \tilde{\bb}}$ of radius $\frac{1}{2n^6}$ centered at the origin be denoted $\tilde{\tilde{B}}_2.$
Then, for any $x \in \tilde{\tilde{B}}_2$, $$\Pi_\bb^{-1}(x) \cap \tilde{R}_{n_1} \supseteq \left(1 - \frac{1}{n}\right) R_{n_1} + x.$$
It follows that $$| \tilde{R}_{n_1}| \geq \left(1-\frac{1}{n}\right)^{n_1^2} |{R}_{n_1}| |\tilde{\tilde{B}}_2|.$$
Therefore, \beq\lab{eq:5.3b}  \limsup\limits_{n_1 \ra \infty} |R_{n_1}|^{\frac{1}{n_1^2}} \leq \exp(-\sigma(\de_0, \de_1, \de_2)).\eeq 
\end{comment}
As in the case of a dyadic square, it follows that \beq \liminf\limits_{n_1 \ra \infty} |R_{n_1}|^{\frac{1}{n_1^2}} \geq \exp(-\sigma(\de_0, \de_1, \de_2)).\eeq As $a \ra \infty$, the supremum over all dyadic triangles $\kappa' \in \D_a$ of $$\sup_{x \in \kappa'} (-1)D_i \tih(x) - \inf_{x \in \kappa} (-1)D_i \tih(x),$$ tends to $0$ because $\tih$ is $C^2$ with a uniformly continuous Hessian. However because $\tih$ is strongly rhombus concave, $\inf_{x \in \kappa'} (-1)D_i \tih(x)$ does not tend to $0$. As a consequence,  as $a$ tends to infinity, the supremum over all $\kappa' \in \D_a$ of 
$$\exp\left(-\sigma(\de_0, \de_1, \de_2) + |\kappa'|^{-1}\int_{\kappa'} \sigma(-D_0h(x), -D_1h(x), -D_2h(x))\Leb_2(dx)\right)$$ tends to $0.$
 
 Also suppose $h_n \in H_n(\tla_n, \tmu_n; \tnu_n)$ has the property that  for each $i$, each $(-1)\De_i h_n$ is  uniformly  bounded above and below by positive constants on $T_n$. Then, any hive $h'_n$ that differs in $\ell_\infty$ norm from $h_n$ by less than $1/(2n^4)$, for sufficiently large $n$, satisfies the rhombus inequalities and thus is a valid hive. 
Taken together with Lemma~\ref{lem:16}, (and by the symmetry of $\tilde{\bb}$ under reflection about $\kappa'\cap \kappa''$)  we obtain a lower bound of $\exp(- \eps_2)V(\tnu)\exp\left(\int_T \sigma(-D_0h(x), -D_1h(x), -D_2h(x))\Leb_2(dx)\right)$ for all $a$ larger than some number depending on $\tih$, on $\liminf_{n\ra \infty} |\tilde{P}_n(\bb, \tih)|^\frac{2}{n^2}$, since the volume of the polytope $\tilde{P}_n(\bb, \tih)$ is equal to the product of the volumes of all the different $Q_{n_1}$  and $Q'_{n_1}$ for various dyadic squares and triangles.  

The lemma follows from Claim~\ref{cl:inlem41}.

\begin{claim} \lab{cl:inlem41} $$\liminf_{n\ra \infty} |\tilde{Q}_n(\bb, \tih)|^{\frac{2}{n^2}} \geq\liminf_{n\ra \infty} |\tilde{P}_n(\bb, \tih)|^{\frac{2}{n^2}}$$ \end{claim}
\begin{proof}
For any $x \in [-\frac{1}{2n^6}, \frac{1}{2n^6}]^\bb,$ the preimage of $(h_\bb) +x $ under the coordinate projection $\Pi_\bb$ contains the polytope $x + h_{\bb} + (1 - \frac{1}{n})(\tilde{P}_n(\bb, \tih) - L_n(\tih)),$ because $\tilde{P}_n(\bb, \tih)$ contains a Euclidean ball of constant radius (belonging to the affine span of $\tilde{P}_n(\bb, \tih)$).
It follows that $$|\tilde{Q}_n(\bb, \tih)| \geq |\tilde{P}_n(\bb, \tih)| \left(1 - \frac{1}{n}\right)^{\dim \tilde{P}_n(\bb, \tih)} \left|\left[-\frac{1}{2n^6}, \frac{1}{2n^6}\right]^\bb\right| .$$ Since, $a$ is constant as $n \ra \infty$,  $\bb$ contains a number of nodes that is linear in $n$, so the claim follows.
\end{proof}

\end{proof}

\begin{lemma}\lab{lem:41-new} For every $h \in H(\la, \mu; \nu)$, let $h_{\eps_1}$ be as defined in Definition~\ref{def:heps1} and  $\ta, \tb, \tg$ be the boundary values of $h_{\eps_1}$ on the vertical, horizontal and diagonal side of $T$ respectively. Let $\tilde{\la} = \partial(\tilde{\a})$, $\tilde{\mu} = \partial(\tilde{\b}),$ and $\tilde{\nu}=\partial(\tilde{\g})$. Given a fixed positive $\eps_2$, for sufficiently small values of $\eps_1$ (depending on $\eps_2)$, the resulting functions $\ta, \tb, \tg$ are such that 

 $$\liminf\limits_{n \ra \infty} |G_n(\tla_n, \tmu_n; \tnu_n, \eps_0 n^2)|^\frac{2}{n^2} \geq  \exp(-2\eps_2) V(\tnu)\exp\left(-\int_T 2\sigma((-1)(\hess h)_{ac})\Leb_2(dx)\right).$$
 %Note: Write about the $\eps_i$.
\end{lemma}
\begin{proof}
By (\ref{eq:G}), 
\beqs |G_n(\tla_n, \tmu_n; \tnu_n, \eps_0 n^2) | & = & \int\limits_{B_\I^n(\tilde{\nu}_n, \eps_0 n^2)}|A_n(\tilde{\la}_n, \tilde{\mu}_n; \nu'_n)|\Leb_{n-1, 0}(d\nu'_n)\\
& = & \int\limits_{B_\I^n(\tilde{\nu}_n, \eps_0 n^2)}|H_n(\tilde{\la}_n, \tilde{\mu}_n; \nu'_n)|\left(\frac{V_n(\nu'_n)}{V_n(\tau_n)}\right)\Leb_{n-1, 0}(d\nu'_n).\eeqs
 If $-\int_T \sigma((-1)(\hess h)_{ac})\Leb_2(dx) > - \infty,$ then,  \beq \lim_{\eps'\ra 0} -\int_{(\partial T)_{\eps'}} \sigma((-1)(\hess h)_{ac})\Leb_2(dx) = 0\lab{eq:bdryT}\eeq where $(\partial T)_{\eps'}$ is an $\eps'$ neighborhood of  $\partial T$. This is due to the logarithmic nature of $\sigma$ ($\sigma(s) = - \ln \f(s)$ by Definition~\ref{def:sigma}) and the finite upper bound on the value that $(-1) \tr (\hess h)$ (which is a Radon measure by \cite{dudley}) assigns to $T$. The convexity of $\sigma$ and (\ref{eq:bdryT}) together imply that if $-\int_T \sigma((-1)(\hess h)_{ac})\Leb_2(dx) > - \infty,$ for any positive $\eps_2$, there is a positive threshold $\eps'$ such that for all $\eps_1 < \eps'$, $$-\int_T \sigma((-1)(\hess h)_{ac})\Leb_2(dx) < \eps_2 - \left( \int_T \sigma((-1)\hess h_{\eps_1})\Leb_2(dx)\right).$$
As a consequence,  (whether or not $-\int_T \sigma((-1)(\hess h)_{ac})\Leb_2(dx)$ is finite,)
\beqs V(\tnu)\exp\left(-\int_T 2\sigma((-1)(\hess h)_{ac})\Leb_2(dx)\right) \leq \exp(\eps_2)V(\tnu)\exp\left(-\int_T 2\sigma((-1)\hess h_{\eps_1})\Leb_2(dx)\right).\eeqs
Thus, it suffices to show that for any fixed $\eps_2 > 0$, when $\eps_1$ is sufficiently small and $h_{\eps_1}$ is defined by Definition~\ref{def:heps1}, $$V(\tnu)\exp\left(-\int_T 2\sigma((-1)\hess h_{\eps_1})\Leb_2(dx)\right) \leq \exp(\eps_2) \liminf\limits_{n \ra \infty} |G_n(\tla_n, \tmu_n; \tnu_n, \eps_0 n^2)|^\frac{2}{n^2}.$$
We shall analyze the polytope $\tilde{Q}_n(\bb, h_{\eps_1})$ given by Definition~\ref{def:6.1}. Since 
$$\tilde{Q}_n(\bb, h_{\eps_1}) \subseteq G_n(\tla_n, \tmu_n; \tnu_n, \eps_0 n^2),$$ it suffices to get a suitable lower bound on $|\tilde{Q}_n(\bb, h_{\eps_1})|$ to prove this Lemma.
This follows from Lemma~\ref{lem:41-new3}.
\end{proof}

\begin{lemma}\lab{lem:30-Jun30_2023}
Let $f$ from $[x_0, x_1]$ to $\R$ be  Lipschitz, and strongly concave,  and $\eps_7 > 0$ be fixed. Then, there is an $\eps_8 > 0$ such that 
 for any  Lipschitz concave function $\tilde{f}$ from $[x_0, x_1]$ to $\R$ that satisfies $\|f - \tilde{f}\|_\infty < \eps_8$, the following holds.  The set of all $x' \in (x_0, x_1)$ such that $$\{\partial^+\tilde{f}(x'), \partial^-\tilde{f}(x')\} \subseteq [-\eps_7 + \partial^+f(x'), \partial^-f(x') + \eps_7],$$
 has Lebesgue measure  at least $x_1 - x_0 - \eps_7$,  
\end{lemma}
\begin{proof}
It is known that the Hessian of a Lipschitz concave function is a Radon measure (see \cite{dudley}). It follows that the measure of all $x \in (x_0 + \de, x_1 -\de)$ such that $\int_{x - \de}^{x + \de}(-1)\hess f(dx) \geq \eps_7/4$ tends to $0$ as $\de$ tends to $0$.
Therefore, there exists  $\eps_{7.5} > 0$ such that for any $\eps' \leq \eps_{7.5},$ the measure of the set $S$ of all $x \in (x_0 + \eps', x_1 - \eps')$ such that \beq\int_{x - \eps'}^{x + \eps'}(-1)\hess f(dx) \geq \eps_7/4\eeq is less than $\eps_7.$ 
By the last fact, there exists an $\eps_8 > 0$ such that if  $\tilde{f}$ is a  Lipschitz concave function from $[x_0, x_1]$ to $\R$ that satisfies $\|f - \tilde{f}\|_\infty < \eps_8$, the following holds. For all $x \not\in S,$ $$\partial^-\tilde{f}(x) < \partial^-f(x - \eps_{7.5}) + \frac{\eps_{7.5}}{2} < \partial^-f(x) + \frac{\eps_{7.5}}{2} + \frac{\eps_{7}}{4},$$ and 
$$\partial^+\tilde{f}(x) > \partial^+f(x + \eps_{7.5}) - \frac{\eps_{7.5}}{2} > \partial^+f(x) - \frac{\eps_{7.5}}{2} -  \frac{\eps_{7}}{4}.$$
The Lemma follows by ensuring that $\eps_{7.5} < \eps_7$.
\end{proof}

\begin{lemma}\lab{lem:41} Let $h \in H(\la, \mu; \nu)$ and $h_{\eps_1} \in H(\tla, \tmu; \tnu)$ be defined in accordance with Definition~\ref{def:heps1}, where $\la = \partial^- \a,$ $\mu = \partial^-\b$ and $\a$ and $\b$ are strongly concave. Given a fixed positive $\eps_2$, for all sufficiently small values of $\eps_1$, 
 $$\left(\frac{1}{V(\tilde{\la})}\right) > \exp( - \eps_2/2)\left( \frac{1}{V(\la)}\right),$$ and
  $$\left(\frac{1}{V(\tilde{\mu})}\right) >  \exp(- \eps_2/2)\left( \frac{1}{V(\mu)}\right).$$ 
    More generally, 
   for any fixed $\eps_2 > 0$,
 for any strongly decreasing  $\hat{\la}$ on $[0, 1]$ such that $\hat{\la} = \partial^-\hat{\a}$, there exists $\eps_3 > 0$ such that if  $\|\a - \hat{\a}\|_\infty < \eps_3,$ then 
  $$V(\hat{\la}) <  \exp(\eps_2/2)V(\la) .$$
\end{lemma}
\begin{proof}
As before, we write $\la = \partial^- \a,$ $\mu = \partial^-\b$ and $\nu = \partial^- \g$, $\tla = \partial^- \tilde{\a},$ $\tmu = \partial^-\tilde{\b}$ and $\tnu = \partial^- \tilde{\g}$ on $(0, 1)$. Here our choices of $\a, \b, \gamma, \tilde{\a}, \tilde{\b}, \tilde{\g}$ are those that are uniquely specified by their being continuous on $[0, 1]$ and the conditions $\a(0) = \tilde{\a}(0) = 0$, and $\b(0) = \tilde{\b}(0) = 0$, and $\gamma(0) = \tilde{\gamma}(0) = 0.$
Due to the fact that $h$ is Lipschitz, as $\eps_1 \ra 0$, $\max(\|\a - \tilde{\a}\|_\infty, \|\b - \tilde{\b}\|_\infty, \|\g - \tilde{\g}\|_\infty)$ tends to $0$. We also know by the conditions on $\la, \mu$, that $\a$ and $\b$ are strongly concave.  Recall from Definition~\ref{def:heps1} that  $h_{\eps_1}:T \ra \R$ denotes following the $C^3$ function where $a(x, y)$ is the unique affine function that results in the function $h_{\eps_1}$ taking the value $0$ at the corners of $T$:
\beqs h_{\eps_1}(x, y) = (h \star \theta_{\eps_1})((1 - 4\eps_1)x + \eps_1, (1 - 4\eps_1)y + 3\eps_1) - \eps_1(x^2 + y^2 - xy) + a(x,y).\eeqs By the rhombus inequalities, $\partial^-(\a)(0)$ is greater or equal to $\partial_y^{-}(h)(x, y)$ for any $0 \leq x \leq y \leq 1.$ Similarly, 
$\partial^-(\a)(1)$ is less or equal to $\partial_y^{-}(h)(x, y)$ for any $0 \leq x \leq y \leq 1.$ Therefore,
$\la(0) \geq \tla(0) - C \eps_1,$ and $\la(1) \leq \tla(1) + C\eps_1.$
We will show that $V(\tla)/V(\la)$ can be made less than $\exp(\eps_2/2)$ by making $\eps_1$ sufficiently small. 
By Lemma~\ref{lem:16}, it suffices to show that if $\eps_1$ is sufficiently small, then
$$\int_{T\setminus\{(t, t)|t \in [0, 1]\}}\log\left(\frac{\tla(x) - \tla(y)}{\la(x)-\la(y)}\right)dxdy < \eps_2/2.$$
If $$\int_{T\setminus\{(t, t)|t \in [0, 1]\}}\log\left({\tla(x) - \tla(y)}\right)dxdy = - \infty,$$ we are done, so since $$\int_{T\setminus\{(t, t)|t \in [0, 1]\}}\log\left({\tla(x) - \tla(y)}\right)dxdy \leq \Leb_2(T) \log\left({\tla(0) - \tla(1)}\right) < \infty,$$ we assume $\int_{T\setminus\{(t, t)|t \in [0, 1]\}}\log\left({\tla(x) - \tla(y)}\right)dxdy$ is finite.
By the strict concavity of $\a$, $\eps_3$ can be chosen such that we have
$$\int_{T\cap\{(t', t)| |t'-t| < \eps_3\}}\bigg|\log\left({\tla(x) - \tla(y)}\right)\bigg|dxdy < \eps_2/4.$$
By choosing $\eps_1$ and hence $\|\a - \tilde{\alpha}\|_\infty$ to be sufficiently small, we can ensure by Lemma~\ref{lem:30-Jun30_2023} that $$\int_{T\setminus\{(t', t)||t-t'| < \eps_3\}}\log\left(\frac{\tla(x) - \tla(y)}{\la(x)-\la(y)}\right)dxdy < \eps_2/4.$$
It follows that 
$$\int_{T\setminus\{(t, t)|t \in [0, 1]\}}\log\left(\frac{\tla(x) - \tla(y)}{\la(x)-\la(y)}\right)dxdy < \eps_2/2.$$
The same argument proves that $V(\tmu)/V(\mu)$ can be made less than $\exp(\eps_2/2)$ by making $\eps_1$ arbitrarily small.
%Let $\tnu = \partial^-\tilde{\gamma}.$ Then, $\tilde{\gamma}$ is strongly concave, and so the same argument above proves that $V(\nu)/V(\tnu)$ can be made less than $\exp(\eps_2/2)$ by making $\eps_1$ arbitrarily small, yielding us the lemma.
In the argument above, we used the nearness in $L^\infty$, and 
$\la(0) \geq \tla(0) - C \eps_1,$ and $\la(1) \leq \tla(1) + C\eps_1,$
but not the specifics of the convolution. Therefore, there exists $\eps_3 > 0$ such that for any strongly decreasing $\hat{\la}$ such that $\hat{\la} = \partial^-\a$ and $\|\a - \hat{\a}\|_\infty < \eps_3,$ 

$$\int_{T\cap\{(t, t')\in [\sqrt{\eps_3}, 1]\times [0, 1-\sqrt{\eps_3}]\}} \log\left(\frac{\hat{\la}(x) - \hat{\la}(y)}{\la(x)-\la(y)}\right)dxdy < \eps_2/10.$$

We need to handle the integral on the two strips $T\cap ([0, \sqrt{\eps_3}]\times[0,1])$ and $T\cap ([0, 1]\times [1 - \sqrt{\eps_3}, 1]).$
 We only handle the first strip, since the second is analogous.
By the concavity of $\a$ and $\hat{\a},$ we see that for any $\eps$ such that $0< \eps < \sqrt{\eps_3}$, $|\hat{\la}(\eps) - \la(\eps)| < \frac{C\eps_3}{\eps}.$
Thus, $$\int_{T\cap\{(t, t')\in [0, \sqrt{\eps_3}]\times [0, 1]\}}\log\left(\frac{\hat{\la}(x) - \hat{\la}(y)}{\la(x)-\la(y)}\right)dxdy = O(\int_0^{\sqrt{\eps_3}} \log \left(\frac{1}{\eps}\right) d\eps).$$
The RHS equals $O(\sqrt{\eps_3} \log (1/\eps_3) + \sqrt{\eps_3})$, which tends to $0$ as $\eps_3$ tends to $0$. Therefore, for a sufficiently small choice of $\eps_3$, we have
 
  $$\left( \frac{1}{V(\la)}\right) < \exp(\eps_2/2)\left(\frac{1}{V(\hat{\la})}\right).$$
\end{proof}

\begin{lemma}\lab{lem:41.5Jul23-2023} Let $h \in H(\la, \mu; \nu)$ and $h_{\eps_1} \in H(\tla, \tmu; \tnu)$ be defined in accordance with Definition~\ref{def:heps1}, where $\la = \partial^- \a,$ $\mu = \partial^-\b$ and $\a$ and $\b$ are strongly concave. Given a fixed positive $\eps_2$, for all sufficiently small values of $\eps_1$, 
  $$V(\tnu) >  \exp(-\eps_2/2)V(\nu).$$
  \end{lemma}
\begin{proof} Let $q(x, y) := x^2 + y^2 - xy.$ Note that $$h_{\eps_1}(x, y) = ((h - \eps_1 q) \ast \theta_{\eps_1})((1 - 4\eps_1)x + \eps_1, (1 - 4\eps_1)y + 3\eps_1) + a'(x, y),$$ for a suitable affine function $a'$. For a given $\eps_1$, let $\de_x, \de_y \in [-\eps_1, 0]$.
Define the hive $\check{h}(x, y) := (h - \eps_1 q)((1 - 4\eps_1)x + \eps_1 + \de_x, (1 - 4\eps_1)y + 3\eps_1 + \de_y) + \check{a}(x, y),$ where $\check{a}$ is chosen to 
make the value of $\check{h}$ on the vertices of $T$ to be $0$, and let $\check{\la}, \check{\mu}$ and $\check{\nu}$ be such that $\check{h} \in H(\check{\la}, \check{\mu}; \check{\nu}).$ As a functional on the space of strongly decreasing functions from $[0, 1]$ to $\R$, $V$ is log-concave. Therefore, to prove this lemma, it suffices to show that when  $\eps_1$ sufficiently small, for all $\de_x, \de_y \in [-\eps_1, 0]$, 
$$V(\nu) \leq V(\check{\nu})\exp(\eps_2/100).$$ To see this, 
note that for all $x < - 100\eps_1 + y$, by the rhombus concavity of $h$, $$\check{\nu}(x)  - \check{\nu}(y) \geq \nu(x + 10 \eps_1) - \nu(y - 10\eps_1).$$
Without loss of generality, we may rescale $h$ and assume that $\nu(0) \leq \frac{1}{2}$ and $\nu(1) \geq - \frac{1}{2}.$

However, because of the presence of the term $ - \eps_1 q$ in the expression for $\check{h}$, $$\bigg|\int\limits_{\substack{(x, y) \in T\\ |x - y| < 100\eps_1}} \log (\check{\nu}(x)  - \check{\nu}(y))dx dy\bigg| < C\eps_1 \log \frac{1}{\eps_1}.$$ Let $R := T \cap (\{x \leq 10 \eps_1\} \cup \{y \geq 1 - 10 \eps_1\} \cup\{|x - y| < 80 \eps_1\}).$ By the rescaling,
 $$\int\limits_R \log (\nu(x)  - \nu(y))dx dy  <  C(\log M) \eps_1, $$ where $M$ is the Lipschitz-constant of $h$.
Thus, when $\eps_1$ is sufficiently small, $$V(\nu) \leq V(\check{\nu})\exp(\eps_2/100),$$ and we are done.
\end{proof}
\begin{lemma}\lab{lem:44-new}Let $\la:[0, 1]\ra \R$ and $\mu:[0, 1] \ra \R$ be strongly decreasing functions that integrate to $0$ over $[0, 1].$ Let $\nu:[0, 1] \ra \R$ be a decreasing function that integrates to $0$ over $[0, 1]$.  Let $X_n$ and $Y_n$ be independent random Hermitian matrices with spectra $\la_n$ and $\mu_n$ respectively.  Let $Z_n = X_n + Y_n.$ Then,
$$\liminf\limits_{n \ra \infty}\p_n\left[\spec(Z_n) \in B_\I^n(\nu_n, \eps n^2)\right]^\frac{2}{n^2} \geq  \sup\limits_{\nu'\in {B}_\I(\nu, \eps)}\sup\limits_{h' \in H(\la, \mu; \nu')} \left(V(\la)V(\mu)\right)^{-1}\J(h').$$
\end{lemma}
\begin{proof}
This is a consequence of Lemma~\ref{lem:8}, Lemma~\ref{lem:41-new}, Lemma~\ref{lem:41} and Lemma~\ref{lem:41.5Jul23-2023}.
\end{proof}
\section{Upper bound}
This section concludes with Lemma~\ref{lem:55-new} which states the following.\\
Let $Z_n = X_n + Y_n.$ Then,
$$\limsup\limits_{n \ra \infty}\p_n\left[\spec(Z_n) \in {B}_\I^n(\nu_n, \eps)\right]^\frac{2}{n^2} \leq \sup\limits_{\nu'\in {B}_\I(\nu, \eps)}\sup\limits_{h' \in H(\la, \mu; \nu')} \left(V(\la)V(\mu)\right)^{-1}\J(h').$$
The strategy to proving this is to break $T$ into dyadic squares, and prove an upper bound on the entropy over each square, when one looks at random hives inside a small tube around a fixed discrete concave function. The key fact used in proving the upper bound is Lemma~\ref{lem:22}, where it is shown, roughly speaking,  that if one considers functions on a torus with a variable upper bound on the Hessian, one can do no better entropy-wise than if this upper bound were a constant.

In Lemma~\ref{lem:46}, which is the lemma of this section that relates discrete hives to continuous rhombus concave functions, there is a requirement on 
 $h \in \bigcup\limits_{\nu'\in \overline{B}_\I(\nu, \eps)}H(\la, \mu; \nu')$  that it satisfies $$\inf\limits_T \min\limits_{0 \leq i \leq 2}  ((-1)D_i h)_{ac} >\tilde{\de} > 0.$$ The way this is achieved is by a convex combination with a fixed rhombus concave function with this property. But to then transfer our results to the general case, we need the below lemma.
 
\begin{lemma}\lab{lem:11}Fix $h \in \bigcup\limits_{\nu'\in \overline{B}_\I(\nu, \eps)}H(\la, \mu; \nu')$ 
that satisfies $$\inf\limits_T \min\limits_{0 \leq i \leq 2}  (-1)(D_i h)_{ac} >\tilde{\de} > 0.$$
 There exists $\nu' \in \overline{B}_\I(\nu, \eps)$ and $h_* \in H(\la, \mu, \nu')$ such that
 for all positive $N$ and $0 \leq \de < 1$, we have
\beqs \limsup\limits_{n \ra \infty} \bigg|G_n(\la_n, \mu_n; \nu_n, \eps n^2)\cap \left(B_{\infty}^{n + 1 \choose 2}\left(L_n(\de h + (1 - \de) h_\ast), \frac{n^2}{N}\right)\oplus \R^{U_n}\right)\bigg|^\frac{1}{n^2} \eeqs \beqs\geq \limsup\limits_{n \ra \infty} |G_n(\la_n, \mu_n; \nu_n, \eps n^2)|^\frac{1}{n^2}(1 - \de).\eeqs
\begin{comment}
\beq \limsup\limits_{n \ra \infty} |G_n(\la_n, \mu_n; \nu_n, \eps n^2)|^\frac{2}{n^2} \leq \f_{H, a}(\la, \mu; \nu, \eps).\eeq
\end{comment}
\end{lemma}
\begin{proof}
By Bronshtein's Theorem  on the $L^\infty$ covering numbers of a space of bounded Lipschitz convex functions at a scale $N^{-1} $ (see Lemma~\ref{lem:Bronshtein}),
for every $N\geq 1$, there is a $C^0$ hive $h_N \in H(\la, \mu; \nu'_N)$ for some $\nu'_N \in B_\infty(\nu, \eps)$ such that 
\beqs \bigg|G_n(\la_n, \mu_n; \nu_n, \eps n^2)\cap\left( B_{\infty}^{n + 1 \choose 2}\left(L_n(h_N), \frac{n^2}{N}\right)\oplus \R^{U_n} \right)\bigg| \geq \exp(- CN)|G_n(\la_n, \mu_n; \nu_n, \eps n^2)|,\eeqs where $C$ is finite constant depending only on $\la$ and $\mu$. 
For each $N$, this implies that 
\beqs \limsup\limits_{n \ra \infty} \bigg|G_n(\la_n, \mu_n; \nu_n, \eps n^2)\cap \left(B_{\infty}^{n + 1 \choose 2}\left(L_n(h_N), \frac{n^2}{N}\right)\oplus \R^{U_n} \right)\bigg|^\frac{2}{n^2} = \limsup\limits_{n \ra \infty} |G_n(\la_n, \mu_n; \nu_n, \eps n^2)|^\frac{2}{n^2} .\eeqs 
Due to the compactness of $\bigcup\limits_{\nu} H(\la, \mu; \nu)$ in the $L^\infty$ topology, there is a subsequence of $(h_N)_{N \geq 1}$ that has a $L^\infty$ limit $h_\ast$, which belongs to $H(\la, \mu; \nu'')$, for some $\nu'' \in \overline{B}_\I(\nu, \eps)$.  
Then for all positive $N$, we have
\beqs \limsup\limits_{n \ra \infty} \bigg|G_n(\la_n, \mu_n; \nu_n, \eps n^2)\cap \left(B_{\infty}^{n + 1 \choose 2}\left(L_n(h_\ast), \frac{n^2}{N}\right)\oplus \R^{U_n}\right)\bigg|^\frac{2}{n^2} = \limsup\limits_{n \ra \infty} |G_n(\la_n, \mu_n; \nu_n, \eps n^2)|^\frac{2}{n^2} .\eeqs
By the Brunn-Minkowski inequality,  (taking account the fact that both the bodies involved below are nonempty)
$$\bigg|G_n(\la_n, \mu_n; \nu_n, \eps n^2)\cap \left(B_{\infty}^{n + 1 \choose 2}\left(L_n(\de h + (1 - \de) h_\ast), \frac{n^2}{N}\right)\oplus \R^{U_n}\right)\bigg|^\frac{1}{n^2}$$ is greater or equal to 
$$\bigg|G_n(\la_n, \mu_n; \nu_n, \eps n^2)\cap \left(B_{\infty}^{n + 1 \choose 2}\left(L_n(h_\ast), \frac{n^2}{N}\right)\oplus \R^{U_n}\right)\bigg|^\frac{1}{n^2}(1-\de) \,\,+ $$
$$\bigg|G_n(\la_n, \mu_n; \nu_n, \eps n^2)\cap \left(B_{\infty}^{n + 1 \choose 2}\left(L_n(h), \frac{n^2}{N}\right)\oplus \R^{U_n}\right)\bigg|^\frac{1}{n^2}\de.$$
%Therefore, this lemma follows from Lemma~\ref{lem:upper}.
But,
$$\bigg|G_n(\la_n, \mu_n; \nu_n, \eps n^2)\cap \left(B_{\infty}^{n + 1 \choose 2}\left(L_n(h), \frac{n^2}{N}\right)\oplus \R^{U_n}\right)\bigg|^\frac{1}{n^2}\de \geq 0,$$ by virtue of its arising from the volume of a nonempty set.
%because $$\inf\limits_T \min\limits_{0 \leq i \leq 2} (-1) (D_i h)_{ac} >\tilde{\de} > 0.$$ 
This completes the proof of the lemma.
\end{proof}
Let $h_\ast \in \bigcup\limits_{\nu'\in \overline{B}_\I(\nu, \eps)}H(\la, \mu; \nu')$.

\begin{defn}\lab{def:34}
For $a \geq 1$, consider the functional $$\J_a: \bigcup\limits_{\nu'} H(\la, \mu; \nu')\ra \R$$ given by 
$$\J_a(h) := V(\nu')\exp\left(-2 \sum_{\kappa \in \D_a}|\kappa| \sigma\left((-1)|\kappa|^{-1}\int_{\kappa^o}(\hess h)(dx)\right)\right)$$ for any $h \in H(\la, \mu; \nu').$ Here $|\kappa|:= \Leb_2(\kappa),$ and $\kappa^o$ is the interior of $\kappa$.
%Let the supremum of $\J_a$ on $ \bigcup\limits_{\nu'\in \overline{B}_\I(\nu, \eps)}H(\la, \mu; \nu')$ be denoted $\f_{H, a}(\la, \mu; \nu, \eps)$. 
\end{defn}
Note that $\J_a(h_*)$ is {\it not} in general equal to the expression of which the limit is being taken in the LHS in (\ref{eq:notJa}), due to the possibility of the singular part of $\nabla^2 h_*$ being supported on $\partial \kappa.$ 
However, by the convexity of $\sigma$,  and (\ref{eq:notJa}) we have the following sets of inequalities, and hence Corollary~\ref{cor:cor}.
For any fixed $h_* \in \bigcup\limits_{\nu'\in \overline{B}_\I(\nu, \eps)}H(\la, \mu; \nu')$

 %such that $\inf\limits_T \min\limits_{0 \leq i \leq 2}  ((-1)D_i h_*)_{ac} > 0,$
\beqs 
\exp\left(-\sum_{\kappa \in \D_a}|\kappa| \sigma\left((-1)|\kappa|^{-1}\int_{\kappa}\hess h_*(dx)\right)\right) \geq \\ \nonumber
\exp\left(-\sum_{\kappa \in \D_a}|\kappa| \sigma\left((-1)|\kappa|^{-1}\int_{\kappa^o}\hess h_*(dx)\right)\right) \geq \\ \nonumber
\exp\left(-\sum_{\kappa \in \D_a}|\kappa| \sigma\left((-1)|\kappa|^{-1}\int_{\kappa}(\hess h_*)_{ac}\Leb_2(dx)\right)\right)\geq  \\ \nonumber \exp\left(-\int_T \sigma((-1)\hess h_*(x))_{ac}\Leb_2(dx)\right).\eeqs 

Thus we have the following  corollary to Theorem~\ref{cor:24}.
\begin{corollary}\lab{cor:cor}
For any fixed $h_* \in \bigcup\limits_{\nu'\in \overline{B}_\I(\nu, \eps)}H(\la, \mu; \nu')$
 %such that $\inf\limits_T \min\limits_{0 \leq i \leq 2}  ((-1)D_i h_*)_{ac} > 0,$
\beq \lab{eq:notJa1.5} \lim_{a \ra \infty} \exp\left(-\sum_{\kappa \in \D_a}|\kappa| \sigma\left((-1)|\kappa|^{-1}\int_{\kappa^o}(\hess h_*)_{ac}\Leb_2(dx)\right)\right) = \\ \nonumber \exp\left(-\int_T \sigma((-1)(\hess h_*(x))_{ac})\Leb_2(dx)\right).\eeq 
%Also as a consequence of (\ref{eq:notJa}) we see that 
\beq \lab{eq:notJa2} \lim_{a \ra \infty} \exp\left(-\sum_{\kappa \in \D_a}|\kappa| \sigma\left((-1)|\kappa|^{-1}\int_{\kappa^o}\hess h_*(dx)\right)\right) = \\ \nonumber \exp\left(-\int_T \sigma((-1)(\hess h_*(x))_{ac})\Leb_2(dx)\right).\eeq  Here $|\kappa|:= \Leb_2(\kappa).$
\end{corollary}

\begin{lemma}\lab{lem:22}
Let $t:\left(E_0(\T_{n_1})\cup E_1(\T_{n_1})\cup E_2(\T_{n_1})\right) \ra \R_+$ be a function from the rhombi of $\T_{n_1}$ to the nonnegative reals.
Let $Q_{n_1}(t)$ be the polytope consisting of the set of all functions $g:V(\T_{n_1}) \ra \R$ such that $\sum_{v \in V(\T_{n_1})} g(v) = 0$ and $\nabla^2(g)(e) \leq t(e)$ for each $e \in E_0(\T_{n_1})\cup E_1(\T_{n_1})\cup E_2(\T_{n_1})$. Suppose that for each $i \in \{0, 1, 2\}$, we have $\sum_{e \in E_i(\T_{n_1})} t(e) = n_1^2 s_i$. Then, $|Q_{n_1}(t)| \leq |P_{n_1}(s)|.$
\end{lemma}
\begin{proof}
Let the translation operator $T_v$ act on functions $g:V(\T_{n_1}) \ra \R$ such that $\sum_{v \in V(\T_{n_1})} g(v) = 0$ by $T_v g(w) := g(w - v).$ Let $T_v(Q_{n_1}(t))$ be defined to be the set of all $T_v g$ as $g$ ranges over $Q_{n_1}(t)$.
Then, denoting Minkowski sums by $\oplus$, we see 
by the Brunn-Minkowski inequality and the fact that the $T_v(Q_{n_1}(t))$ all have the same volume that 
\beqs |Q_{n_1}(t)|^{\frac{1}{n_1^2-1}} \leq \Bigg|\left(\frac{1}{n_1^2}\right)\left(\bigoplus_{v \in V(\T_{n_1})}  T_v(Q_{n_1}(t))\right) \Bigg|^{\frac{1}{n_1^2-1}}.\eeqs From the containment \beqs \left(\frac{1}{n_1^2}\right)\left(\bigoplus_{v \in V(\T_{n_1})}  T_v(Q_{n_1}(t))\right) \subseteq P_{n_1}(s),\eeqs it follows that 
$$ \Bigg|\left(\frac{1}{n_1^2}\right)\left(\bigoplus_{v \in V(\T_{n_1})} T_v(Q_{n_1}(t))\right) \Bigg|^{\frac{1}{n_1^2-1}} \leq  |P_{n_1}(s)|^{\frac{1}{n_1^2-1}}.$$ 
This completes the proof of the lemma.

\end{proof}
\begin{comment}
We will map $V(\T_n)$ onto $(\Z/n\Z) \times (\Z/n\Z)$ via the unique $\Z$ module isomorphism that maps $1$ to $(1, 0)$ and $\omega$ to $(0, 1)$. 
Given $n_1|n_2$, the natural map from $\Z^2$ to $\Z^2/(n_1 \Z^2) = V(\T_{n_1})$ factors through $\Z^2/(n_2 \Z^2) =V(\T_{n_2})$. We denote the respective resulting maps from $V(\T_{n_2})$ to $V(\T_{n_1}) $ by $\phi_{n_2, n_1}$, from $\Z^2$ to $V(\T_{n_2})$ by $\phi_{0, n_2}$ and from $\Z^2$ to $V(\T_{n_1})$ by $\phi_{0, n_1}$.
\end{comment}

%Finally we define $\bb$ to be $\tilde{\bb} + o$, \ie a translation of %$\tilde{\bb}$ by the offset $o$.
%Given $\bb$, define $(h_\bb)_{quant}$ to be the closest point to $h_\bb$, every coordinate of which is an  integer multiple of $\frac{1}{M}$, where $M := 2n^6$. 

%\subsection{Slack of a facet}

\begin{lemma}\lab{lem:46} %Let ${h_\ast} \in \bigcup\limits_{\nu'\in \overline{B}_\I(\nu, \eps)}H(\la, \mu; \nu')$. 
 Fix $ a \in \N$ and $h \in \bigcup\limits_{\nu'\in \overline{B}_\I(\nu, \eps)}H(\la, \mu; \nu')$ that satisfies $$\inf\limits_T \min\limits_{0 \leq i \leq 2}  ((-1)D_i h)_{ac} >\tilde{\de} > 0.$$ Fix $\eps_9 > 0$.  There is a positive $n_0$, such that following is true for all  $n > n_0$. Let $\bb$ be defined by Definition~\ref{def:37} and let  $\tilde{Q}_n$ be as in Definition~\ref{def:6.1}.  For every  $h'' = \tilde{\de}h + (1 - \tilde{\de}){h_{**}} \in H(\la, \mu, \nu''),$ where $h_{**} \in \bigcup\limits_{\nu'\in \overline{B}_\I(\nu, \eps)}H(\la, \mu; \nu')$,  $$|\tilde{Q}_n(\bb, h'')|^\frac{2}{n^2} < \J_{a}(h'')\exp( \eps_9).$$
\end{lemma}
\begin{proof}
We proceed to study $\tilde{Q}_n(\bb, h'')$ in order to obtain an upper bound on its volume. 
\ben 
\item Let $\kappa$ be a dyadic square in $\D_a$.
 All the lattice points in $n\kappa$ for a dyadic square $\kappa \in \D_a$ that are not in $\bb$ are are separated from the lattice points in every abutting dyadic square by at least four layers of $\bb$. Let the subset of $V(T_n)\cap n\kappa$ of all lattice points at an $\ell_\infty$ distance of at most $2$ from the set $\left(V(T_n) \setminus \bb\right) \cap n\kappa$ be denoted $V_\kappa$.
 Denote the projection of the translated polytope $\left(\tilde{Q}_n(\bb, h'') - L_n(h'')\right)$ onto $ \R^{V_\kappa}$ by $Q_\kappa$. Then
 $\Pi_{\sum_i x_i = 0} Q_\kappa \subseteq Q_{n_1}(t)$, where $n_1^2$ is the number of vertices in $V_\kappa$, and for a rhombus $e$, $t(e)$ is $- \hess(h'')(e)$,   and $Q_{n_1}(t)$ is the polytope defined in Lemma~\ref{lem:22}.  Suppose that for each $i \in \{0, 1, 2\}$, we have $\sum_{e \in E_i(\T_{n_1})} t(e) = n^2 s_i$. Then,  applying Lemma~\ref{lem:22}, we see that $|Q_{n_1}(t)| \leq |P_{n_1}(s)|.$ That $s$ is dominated by $\frac{16n_1}{n^6}  -  |\kappa|^{-1}\int_{\kappa^o}\hess h''(dx),$ can be seen as follows using the fact that $V_\kappa \subseteq n\kappa^o.$  Suppose first that $h''$ is a $C^2$ function on  $n \kappa^o. $ Then this domination holds by Green's Theorem (see Lemma~\ref{lem:3.1}).  Next suppose $h''$ is an arbitrary Lipschitz rhombus concave function on $n\kappa^o.$ The required result follows from the smooth case after convolution with a $C^\infty$ approximate identity of sufficiently small support.  Therefore,  by Lemma~\ref{lem:2.8} and Observation~\ref{obs:fn},  for all sufficiently large $n$ (and hence $n_1$) \beq |P_{n_1}(s)|^\frac{1}{n_1^2 - 1} \leq \frac{Cn_1}{n^6} + |P_{n_1}(-|\kappa|^{-1}\int_{\kappa^o}\hess h''(dx))|^\frac{1}{n_1^2- 1},\eeq where $C$ is an absolute constant.
 Assuming $ |P_{n_1}(-|\kappa|^{-1}\int_{\kappa^o}\hess h''(dx))|^\frac{1}{n_1^2- 1} = \de_0 > 0,$ we see from Lemma~\ref{lem:22} that 
 \beqs |Q_\kappa|^\frac{1}{n_1^2} & \leq &  (diam (Q_\kappa))^{\frac{1}{n_1^2}} |\Pi_{\sum_i x_i = 0} Q_\kappa|^{\frac{1}{n_1^2}}\\
 & \leq & (diam(Q_\kappa))^{\frac{1}{n_1^2}}\left(\left(\frac{Cn_1}{ n^6} + |P_{n_1}(-|\kappa|^{-1}\int_{\kappa^o}\hess h''(dx))|^\frac{1}{n_1^2-1}\right)^{n_1^2 -1}\right)^{\frac{1}{n_1^2}}.\eeqs 
The RHS can be (crudely) bounded above by $$\left(\frac{Cn_1^3}{ \de_0 n^6} + (1 + \frac{C\log n}{n_1^2}) |P_{n_1}(-|\kappa|^{-1}\int_{\kappa^o}\hess h''(dx))|^\frac{1}{n_1^2}\right),$$ 
because letting $\xi = \frac{Cn_1}{n^6},$ we have $(\xi + \de_0)^{\frac{1}{n_1^2}} \leq \de_0^{\frac{1}{n_1^2}} \left(1 + \frac{C\xi}{{n_1^2\de_0}}\right).$ 

\begin{figure}\label{fig:onesquare}
\begin{center}
\includegraphics[scale=0.5]{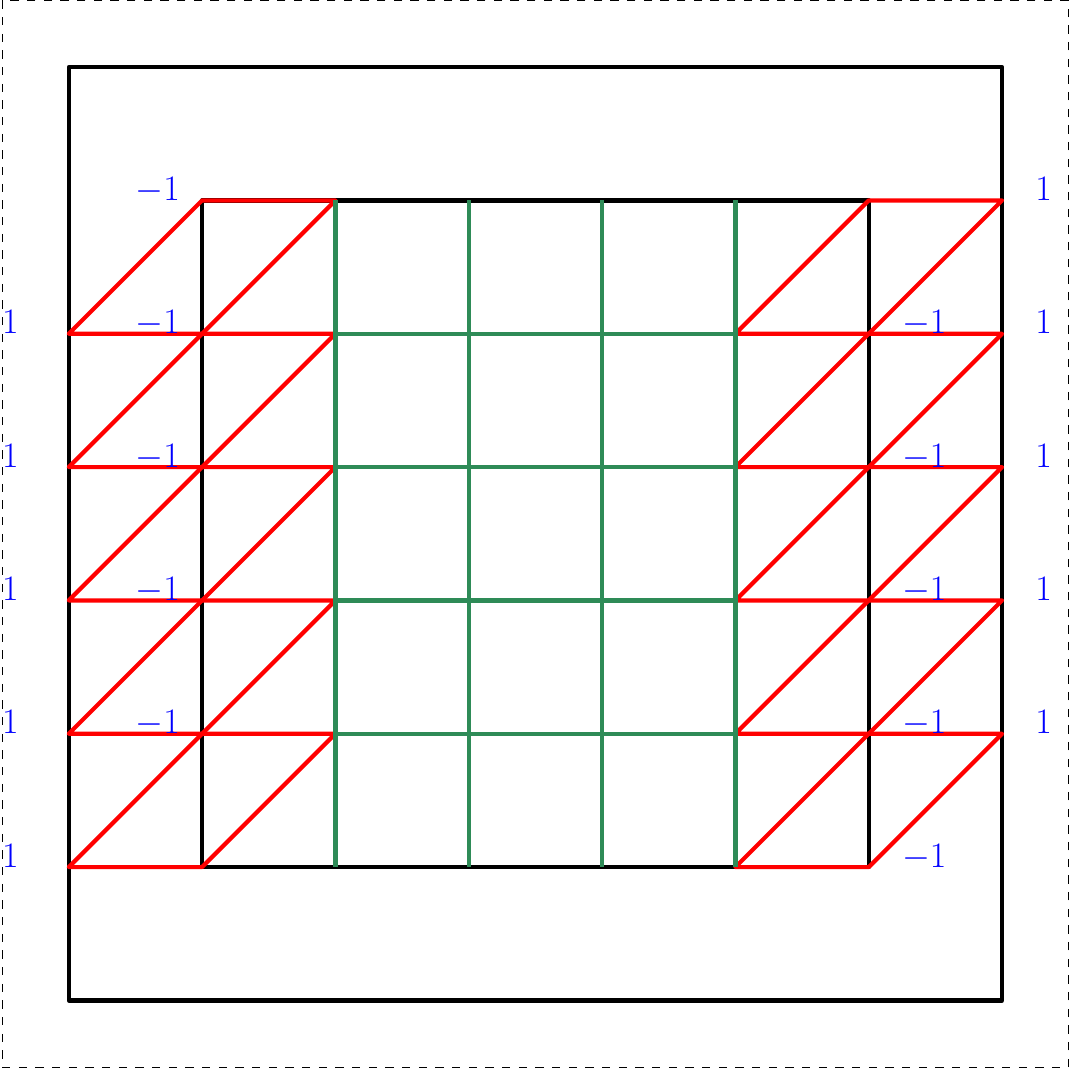}
\caption{A single dyadic square  (depicted with dashed lines) and the inner two layers of a   fixed boundary around it. The rhombi (which have been linearly transformed to become parallelograms) in red are some of the ones that intersect two layers of the boundary but are not contained in the boundary. The $3\times 3$ square in green in the center corresponds to $V_\kappa.$}
\end{center}
\end{figure}

\begin{figure}\label{fig:onetriangle}
    \centering
    \begin{subfigure}[t]{0.4\textwidth}
        \centering
        \includegraphics[width=\linewidth]{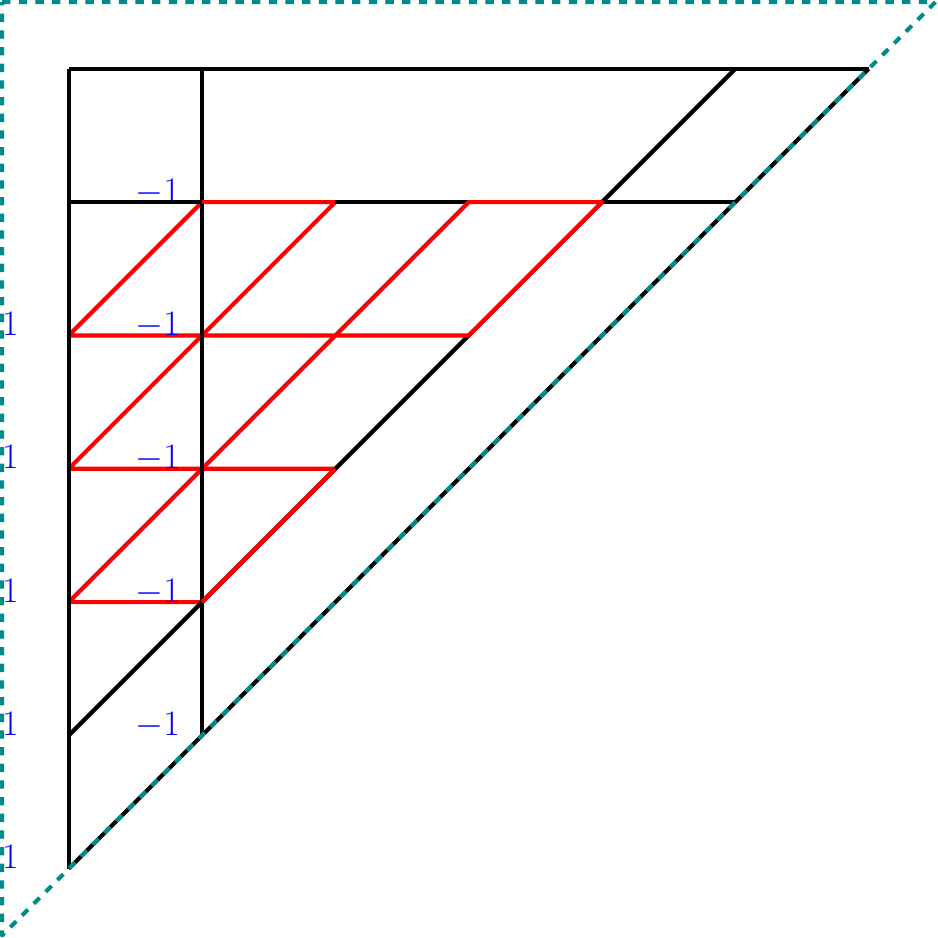} 
        \caption{A single dyadic triangle (depicted with dashed lines) and a fixed boundary around it.} \label{fig:onetriangle1}
    \end{subfigure}
    \hfill
    \begin{subfigure}[t]{0.4\textwidth}
        \centering
        \includegraphics[width=\linewidth]{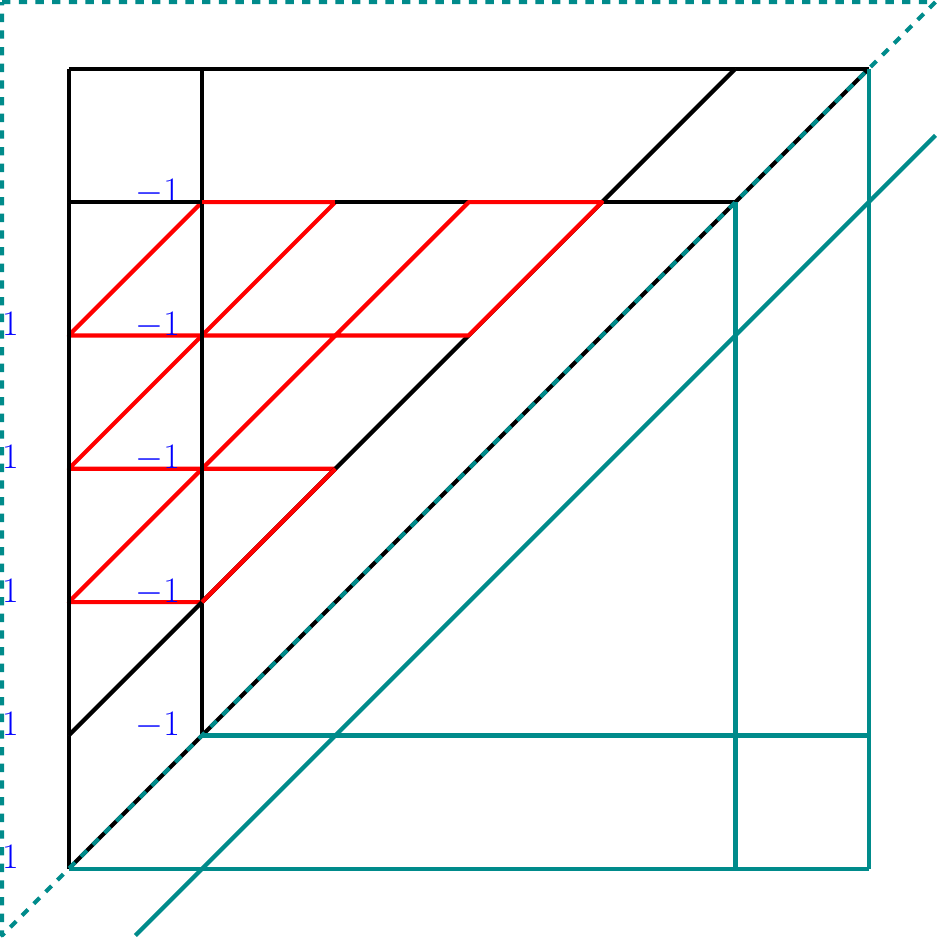} 
        \caption{completing a triangle to a square.} \label{fig:onetriangle2}
    \end{subfigure}
    \end{figure}

 \item Let $\kappa'$ be a dyadic triangle in $\D_a$. Let $\kappa''$ denote the triangle obtained by a central  reflection $\RR(\kappa') = \kappa''$, where $\RR(x) = p - x$, $p$ being the center of the hypotenuse of $\kappa$. Let $T'_n$ denote the right angled lattice triangle below the diagonal in the augmented $n \times n$ hive. 
 Let $V(T_n )\cap (n\kappa')$  be denoted $V_{\kappa'}$ and $V( T'_n)\cap (n\kappa'')$  be denoted $V_{\kappa''}$.
 Denote the projection of the translated polytope $\left(\tilde{Q}_n(\bb,  h'') - L_n(h'')\right)$ onto $ \R^{V_{\kappa'}}$ by $Q_{\kappa'}$. Let $T_\RR$ be the map from $\R^{V_{\kappa'}}$ to $\R^{V_{\kappa''}}$ that takes function of $V_{\kappa'}$ to a function of $V_{\kappa''}$ by mapping $V_{\kappa''}$ into $V_{\kappa'}$ using $\RR$.
 Denote the isomorphic copy $T_\RR(Q_{\kappa''})$ by $Q_{\kappa''}$.
 Note that $V_{\kappa'} \cap V_{\kappa''}$ consists of vertices on the diagonal, on which any points in $Q_{\kappa''}$ and $Q_{\kappa''}$, viewed as functions, take the value $0$. Suppose 
 $\Pi_{V_{\kappa'} \cup V_{\kappa''}}$ projects $\R^{V_{\kappa'}\cup V_{\kappa''}}$ onto $\{x \in \R^{V_{\kappa'} \cup V_{\kappa''}}, \sum_i x_i = 0\}$ orthogonally. Let $Q_{\kappa'\cup\kappa''}$ be the polytope $$\Pi_{V_{\kappa'} \cup V_{\kappa''}}\left(\left(Q_{\kappa'} \times_{{V_{\kappa'} \cap V_{\kappa''}}} Q_{\kappa''}\right) + \left[-\frac{1}{2n^6}, \frac{1}{2n^6}\right]^{(\bb \cap \kappa') \cup (\RR(\bb) \cap \kappa'')}\right),$$ where, $Q_\kappa:=Q_{\kappa'} \times_{{V_{\kappa'} \cap V_{\kappa''}}} Q_{\kappa''}$ is defined to be the subset of $\R^{V_{\kappa'}\cup V_{\kappa''}}$ given by the set of points whose coordinates equal their coordinates in $Q_{\kappa'}$ over $V_{\kappa'}$ and in $Q_{\kappa''}$ over  $V_{\kappa''}$, (which, on $V_{\kappa'} \cap V_{\kappa''}$, belong to $\left[-\frac{1}{2n^6}, \frac{1}{2n^6}\right]$ and agree across the two patches $V_{\kappa'}$ and $V_{\kappa''}$). Then
 \beq\lab{eq:Qk} \Pi_{V_{\kappa'} \cup V_{\kappa''}} Q_\kappa \subseteq Q_{n_1}(t),\eeq where $n_1^2$ is the number of vertices in $V_{\kappa'}\cup V_{\kappa''}$, and for a rhombus $e$ is supported in  $\kappa'$, $t(e)$ is $- \hess(L_n(h''))(e)$ or if $e$ is supported in $\kappa''$, $t(e)$ is $- \hess(L_n(h''))(\RR(e))$, or in case $e$ straddles the hypotenuse boundary of $\kappa$, $t(e)$ equals $\frac{2}{n^6}$. As before, $Q_{n_1}(t)$ is the polytope defined in Lemma~\ref{lem:22}.  Suppose that for each $i \in \{0, 1, 2\}$, we have $\sum_{e \in E_i(\T_{n_1})} t(e) = n^2 s_i$. Then,  by (\ref{eq:Qk}) and applying Lemma~\ref{lem:22}, we see that $$  (diam(Q_\kappa))^{-1}|Q_\kappa| \leq |\Pi_{V_{\kappa'} \cup V_{\kappa''}} Q_\kappa| \leq |Q_{n_1}(t)| \leq |P_{n_1}(s)|.$$

As in part 1.  of this proof,  we establish using Green's Theorem and the fact that $V_{\kappa'} \subseteq n \kappa',$ that $s$ is dominated by $\frac{16n_1}{n^6}  -  |\kappa'|^{-1}\int_{\kappa'}\hess h''(dx).$ Therefore, by Lemma~\ref{lem:2.8},  for all sufficiently large $n_1,$ \beq|P_{n_1}(s)|^\frac{1}{n_1^2 - 1} \leq \frac{Cn_1}{n^6} + |P_{n_1}(-|\kappa'|^{-1}\int_{\kappa'}\hess h''(dx))|^\frac{1}{n_1^2- 1}.\eeq
 Therefore, assuming $|P_{n_1}(-|\kappa'|^{-1}\int_{\kappa'}\hess h''(dx))|^\frac{1}{n_1^2} =  \de_0 > 0$,  \beqs |Q_{\kappa'}|^{\frac{2}{n_1^2}} & \leq & (diam(Q_{\kappa}))^{\frac{1}{n_1^2}} \left(\sqrt{|P_{n_1}(s)|}\right)^{\frac{2}{n_1^2}}\\ & \leq & \left( \left(\frac{Cn_1}{n^6} + (1 + \frac{C\log n}{n_1^2})|P_{n_1}(-|\kappa'|^{-1}\int_{\kappa'}\hess h''(dx))|^\frac{1}{n_1^2- 1}\right)^{n_1^2-1}\right)^{\frac{1}{n_1^2}}\\
 & \leq &\left(\frac{Cn_1^3}{\de_0n^6} + (1 + \frac{C\log n}{n_1^2})|P_{n_1}(-|\kappa'|^{-1}\int_{\kappa'}\hess h''(dx))|^\frac{1}{n_1^2}\right).\eeqs
 \een
 If a vector space $E$ is the orthogonal sum of subspaces $E_1 \oplus \dots \oplus E_r$, then for any full dimensional polytope $Q \subseteq E$, the volume of $Q$ is less or equal to the product of the volumes of the projections $\Pi_{E_i}Q$ of $Q$ on to the $E_i$ because $Q \subseteq \prod_{i=1}^r \Pi_{E_i} Q.$
 Thus,  by 1.  and 2. above,  \beqs |\tilde{Q}_n(\bb, h'')|^{\frac{1}{n_1^2}} & \leq & \left(\prod_{\text{subsquares\,}\kappa \in \D_a}\left(\frac{Cn_1^3}{\de_0 n^6} + (1 + \frac{C\log n}{n_1^2})|P_{n_1}(-|\kappa|^{-1}\int_{\kappa^o}\hess h''(dx))|^\frac{1}{n_1^2}\right)\right)\times\\ & & \left(\prod_{\text{subtriangles\,}\kappa' \in \D_a}
 \left(\frac{Cn_1^3}{\de_0 n^6} + (1 + \frac{C\log n}{n_1^2})|P_{n_1}(-|\kappa'|^{-1}\int_{\kappa'}\hess h''(dx))|^\frac{1}{n_1^2}\right)\right)\times\\& & V_n(\nu''_n).\eeqs
 The lemma follows from Lemma~\ref{lem:2.8}, noting that $\de_0$ is independent of $n$, and for each $\kappa' \in \D_a$, the convergence of $|P_{n_1}(-|\kappa'|^{-1}\int_{\kappa'}\hess h''(dx))|^\frac{1}{n_1^2}$ to $\f(-|\kappa'|^{-1}\int_{\kappa'}\hess h''(dx))$ as $n$ tends to infinity is uniform over the $h''$ in question by Corollary~\ref{cor:last} because $-|\kappa'|^{-1}\int_{\kappa'}\hess h''(dx)$ ranges over a compact subset of $\R_{> 0}^3$.
\end{proof}

Let $h \in \bigcup\limits_{\nu'\in \overline{B}_\I(\nu, \eps)}H(\la, \mu; \nu').$ %satisfy $\inf\limits_T \min\limits_{0 \leq i \leq 2}  ((-1)D_i h)_{ac} > 0.$
Let $\kappa$ be a dyadic square in $\D_a$ such that \beq\lab{eq:dist} \inf_{\substack{{z \in \kappa}\\{z'\in \partial T}}} |z - z'| > 0.\eeq Let the $y$ coordinates in $\kappa$ range from $y_0= y^\kappa_0$ to $y_1 = y^\kappa_1$, and $x$ coordinates range from $x_0= x^\kappa_0$ to $x_1 = x^\kappa_1$. Consider the function $f_\kappa(x) = \int_{y_0}^{y_1}h(x, y)dy$, and the function $g_\kappa(y) = \int_{x_0}^{x_1} h(x, y) dx$.  For every fixed $y'$, $h(x, y')$ is a bounded, Lipschitz, concave function of $x \in [x_0, x_1]$. Therefore $f(x) = f_\kappa(x)$ is a bounded, Lipschitz, concave function of $x$ for $x \in [x_0, x_1]$. Similarly $g(y) = g_\kappa(y)$ is a bounded Lipschitz, concave function of $y$ for $y \in [y_0, y_1]$. 

Note that $\partial^-f_\kappa(x_0), \partial^+f_\kappa(x_1),  \partial^-g_\kappa(y_0)$ and $\partial^+g_\kappa(y_1)$  can be unambiguously defined because of (\ref{eq:dist}).
\begin{lemma}\lab{lem:39} Fix $\eps_7 = \eps^\kappa_7 > 0$.  Let $x' \in [x_0, x_1]$ and $y' \in [y_0, y_1].$

 There exists a sufficiently small positive $\eps_8 = \eps^\kappa_8$  depending on $\eps_7$, $\kappa,$  $x'$, $y'$ and $h$ such that the following is true.
 Let $\tilde{f}:[x_0, x_1]\ra \R$  and $\tilde{g}:[y_0, y_1]\ra \R$ be given by  $\tilde{f}(x) = \int_{y_0}^{y_1}\tilde{h}(x, y)dy$, and $\tilde{g}(y) = \int_{x_0}^{x_1}\tilde{h}(x, y)dx$ for some $\tilde{h} \in \bigcup\limits_{\nu'\in \overline{B}_\I(\nu, \eps)}H(\la, \mu; \nu')$  such that $$\|h - \tilde{h}\|_\infty < \eps_8.$$ Then, 
\ben \item $\{\partial^+\tilde{f}(x'), \partial^-\tilde{f}(x')\} \subseteq [-\eps_7 + \partial^+f(x'), \partial^-f(x') + \eps_7].$
 \item $\{\partial^+\tilde{g}(y'), \partial^-\tilde{g}(y')\} \subseteq [-\eps_7 + \partial^+g(y'), \partial^-g(y') + \eps_7].$
 \een
\end{lemma}
\begin{proof}
Since $h$ and $\tilde{h}$ are Lipschitz and concave, it follows that the same is true of $f, \tilde{f}$, $g$ and $\tilde{g}$. 
%The lemma follows from Lemma~\ref{lem:39new}.
%Note  that $|T_7| + |S_7| = \#_7 <  \infty$. Therefore, it suffices to prove that for a given $x' \in S_7,$ $$\{\partial^+\tilde{f}(x'), \partial^-\tilde{f}(x')\} \subseteq [-\eps_7 + \partial^+f(x'), \partial^-f(x') + \eps_7],$$ and for a given $y' \in T_7$,  $$\{\partial^+\tilde{g}(y'), \partial^-\tilde{g}(y')\} \subseteq [-\eps_7 + \partial^+g(y'), \partial^-g(y') + \eps_7],$$ if
 %$\|h - \tilde{h}\|_\infty$ is less than a sufficiently small positive real  $\eps_8 = \eps^\kappa_8.$ 
Since the situations are symmetric, it suffices to prove 1.
 
 There exists $x'_0 < x'$ such that $$\partial^- f(x') < (f(x') - f(x'_0))/(x' - x'_0) < \partial^- f(x') + \eps_7/2.$$
 
 %Now $$(f(x') - f(x'_0))/(x' - x'_0) = \frac{\int_{x'_0}^{x'} (\partial^-f(x)) dx}{x' - x'_0}.$$ 

%Therefore, the measure of points $x$ in $[x'_0, x']$ such that $\partial^-f(x) < \partial^- f(x') + \eps_7/2$  is at least $(x' - x'_0)/2$, and so $\partial^-f((x' + x'_0)/2) < \partial^- f(x') + \eps_7/2.$ 
%Similarly,  we can find $x'_1 > x'$ such that $\partial^+f((x' + x'_1)/2) > \partial^+ f(x') - \eps_7/2.$ 

Since $\|f - \tilde{f}\|_\infty < \eps_8 2^{-a}$,  we see that 
 $$\partial^- \tilde{f}(x') < (f(x') - f(x'_0)+  \eps_8 2^{-a+1})/(x' - x'_0) < \partial^- f(x') + \eps_7,$$ for sufficiently small $\eps_8.$
 Similarly, when $\eps_8$ is sufficiently small, 
the case of $\partial^+ \tilde{f}(x')$ is analogous and
$$\partial^+ \tilde{f}(x') > \partial^+ f(x') - \eps_7.$$
 The result thus follows.
\end{proof}

\begin{lemma}\lab{lem:37}
 Let $\eps_{8.5}  > 0$,  be  an arbitrarily small  positive real.   Let  $h' \in \bigcup\limits_{\nu'\in \overline{B}_\I(\nu, \eps)}H(\la, \mu; \nu').$ %satisfy $$\inf\limits_T \min\limits_{0 \leq i \leq 2}  ((-1)D_i h')_{ac} > \tilde{\de} > 0.$$ %For a dyadic square $\kappa$, let $f'_\kappa(x) = \int_{y_0}^{y_1}h'(x, y)dy$, and $g'_\kappa(y) = \int_{x_0}^{x_1} h'(x, y) dx.$ 
There exists
an $ a' \in \N$ with the following property.   For every  $a \geq a'$,  the total Lebesgue measure of all dyadic squares $\kappa = [x_0, x_1) \times (y_0, y_1] \in \D_a$ such that 
\ben \item $\inf_{\substack{{z \in \kappa}\\{z'\in \partial T}}} |z - z'| > 0,$ and 
  \item $|\partial^+{f'_\kappa}(x_0) -  \partial^-{f'_\kappa}(x_0)| + |\partial^+{f'_\kappa}(x_1) -  \partial^-{f'_\kappa}(x_1)| < \eps_{8.5}|\kappa|$ and
 \item $|\partial^+{g'_\kappa}(y_0) -  \partial^-{g'_\kappa}(y_0)| + |\partial^+{g'_\kappa}(y_1) -  \partial^-{g'_\kappa}(y_1)| < \eps_{8.5}|\kappa|,$
\een
 is greater than $\frac{1}{2} - \eps_{8.5}.$
\end{lemma}
\begin{proof}
Note that $f'_\kappa$ and $g'_{\kappa}$ extend to an open neighborhood of $[x_0, x_1]$ because of 1.  and so $\partial^+{f'_\kappa}(x_1), $ $\partial^-{f'_\kappa}(x_0),$ $\partial^+{g'_\kappa}(y_1)$ and $\partial^-{g'_\kappa}(y_0)$ are all well defined.
We know from the rhombus inequalities that $h'$ is Lipschitz, and hence that $|\tr((\hess h')_{sing})(T)|$ is  finite. Let $\partial \kappa$ denote those points in $\kappa$ which do not belong to $\kappa^o$; recall that $\kappa$ is a half-open  dyadic square or a closed dyadic triangle. Therefore there is a finite upper bound independent of $a''$
on $\sum_{\kappa \in \D_{a''}} \int_{\partial \kappa} (-1) \tr(\hess h')(dx)$.  However,  for any $a \geq a''$,  $$\bigcup_{\kappa \in \D_a} \partial \kappa \supseteq \bigcup_{\kappa \in \D_{a''}} \partial \kappa.$$ Therefore,
 $\sum_{\kappa \in \D_a} \int_{\partial \kappa} (-1) \tr(\hess h')(dx)$ is a monotonically increasing function of $a$. 
 The measure $\tr((\hess h')_{sing})$ endows any point with zero measure,  because by Theorem 6.1 of \cite{dudley},  $\tr((\hess h'))$ is absolutely continuous with respect to the one dimensional Hausdorff measure.   Therefore,  for a dyadic square $\kappa = [x_0, x_1) \times (y_0, y_1]$ we know that   $|\partial^+{f'_\kappa}(x_0) -  \partial^-{f'_\kappa}(x_0)| + |\partial^+{g'_\kappa}(y_1) -  \partial^-{g'_\kappa}(y_1)|$
 is  bounded above by $\int_{\partial \kappa} (-1) \tr((\hess h')_{sing})(dx)$. 
 It follows that there exists some sufficiently large $a''$ such that for all $a > a''$,  
$$\sum_{\kappa \in \D_a} \int_{\partial \kappa} (-1) \tr(\hess h')(dx) - \sum_{\kappa \in \D_{a''}} \int_{\partial \kappa} (-1) \tr(\hess h')(dx) < \frac{\eps_{8.5}^2 }{100}.$$
For any fixed $a''$, the total Lebesgue measure of all dyadic squares $\kappa \in \D_{a}$ such that $\partial \kappa'' \cap \partial \kappa = \emptyset$ for all $\kappa'' \in \D_{a''},$ tends to $\frac{1}{2}$ (and in particular eventually exceeds $\frac{1}{2} - \frac{\eps_{8.5}}{2}$) as $a$ tends to infinity.
The Lemma now follows.
\end{proof}
\begin{lemma}\lab{lem:38}
  Fix $\eps_{8.5} > 0$,   and $h' \in \bigcup\limits_{\nu'\in \overline{B}_\I(\nu, \eps)}H(\la, \mu; \nu'),$ such that $\J(h') > \eps_{8.5}.$ %that satisfies $$\inf\limits_T \min\limits_{0 \leq i \leq 2}  ((-1)D_i h')_{ac} > \tilde{\de} > 0.$$ 
 For all sufficiently large $a''$,  there exists $N > 0$ such that 
for every $h'' \in \bigcup\limits_{\nu'\in \overline{B}_\I(\nu, \eps)}H(\la, \mu; \nu')$ such that %\ben \item 
$\|h'' - h'\|_\infty < \frac{1}{N}$, 
%and \item $\inf\limits_T \min\limits_{0 \leq i \leq 2}  (D_i h'')_{ac} > \tilde{\de} > 0,$
%\een 
the following is true.
\beqs  \J_{a''}(h'') < (\J_{a''}(h') +  \eps_{8.5}^2)\exp\left( C \eps_{8.5} + \eps_{8.5}\log \frac{C(\la, \mu)}{\eps_{8.5}}\right).\eeqs
\end{lemma}
\begin{proof} 
%Recall from Definition~\ref{def:34} that $$\J_a(h) = V(\nu')\exp\left(-\sum_{\kappa \in \D_a}|\kappa| \sigma\left((-1)|\kappa|^{-1}\int_{\kappa^o}\hess h(dx)\right)\right)$$ for any $h \in H(\la, \mu; \nu').$
%Let $\eps_{8.5}$ be fixed to a small positive real.
Let $a''$ be an arbitrary integer greater or equal to the $a'$ appearing in the statement of Lemma~\ref{lem:37}.   For each $\kappa \in \D_{a''}$,  let $\eps_8^\kappa$ be 
%as in the statement of Lemma~\ref{lem:39}. 
a small constant, whose value will be fixed later in this proof.
 Let $N$ be chosen such that  \beq \frac{1}N < \inf\limits_{\substack{\kappa \in \D_{a''}\\(x', y') \in \left(2^{-a''}\N^2\right)\cap \overline{\kappa}}} 4^{-a''}\eps_8^\kappa .\lab{eq:N-14Jul23}\eeq
Recall that by (\ref{eq:D0}),  (\ref{eq:D1}) and (\ref{eq:D2}),  we have 
that when $h$ is $C^2$, 
\beqs D_0 h & = & \partial_x\partial_y h+ \partial^2_x h \\
         D_1 h & = & - \partial_x \partial_y h\\
         D_2 h & = & \partial_x \partial_y h + \partial_y^2 h.\eeqs From this, we see by convolution with a $C^\infty$ approximate identity and Green's Theorem that for {\it any} Lipschitz rhombus concave $h$ and $\D_{a''} \ni \kappa = [x_0, x_1) \times [y_0, y_1),$ that 
\beqs \int_{\kappa^o}(-1)D_0h(dx) & = &  - h(x_0, y_0) - h(x_1, y_1) + h(x_0, y_1) + h(x_1, y_0) + \partial^+f_\kappa(x_0) - \partial^-f_\kappa(x_1),\\
 \int_{\kappa^o}(-1)D_1 h(dx) & = & h(x_0, y_0) + h(x_1, y_1) - h(x_0, y_1) - h(x_1, y_0),\\
\int_{\kappa^o}(-1)D_2 h(dx) & = & - h(x_0, y_0) - h(x_1, y_1)  + h(x_0, y_1) + h(x_1, y_0) + \partial^+g(y_0) - \partial^-g(y_1).\eeqs
\begin{observation}\lab{obs:A12Jul23}
By our choice of parameters,  for all $\kappa \in \D_{a''}$ that are contained in a set $S_1 \cup S_2$ of measure $\frac{1}{2} - \eps_{9}$ whose existence is guaranteed by Lemma~\ref{lem:39} and  Lemma~\ref{lem:37}, we either have for $\kappa \in S_1$ \beq \left| \sigma\left((-1)|\kappa|^{-1}\int_{\kappa^o}\hess h''(dx)\right)   -  \sigma\left((-1)|\kappa|^{-1}\int_{\kappa^o}\hess h'(dx)\right) \right| < C\eps_{8.5},\eeq
or have, for $\kappa \in S_2$,   \beq \lab{eq:6.7} \sigma\left((-1)|\kappa|^{-1}\int_{\kappa^o}\hess h'(dx)\right) > \frac{1}{\eps_{9}^2 },\eeq and 
 
\beq \sigma\left((-1)|\kappa|^{-1}\int_{\kappa^o}\hess h''(dx)\right) > \frac{1}{\eps_{9}^2},\eeq since $\J(h') > \eps_{8.5},$ for some $\eps_9 < \eps_{8.5}$ that depends only on $\eps_{8.5}$ and $h'$. We note that by Markov's inequality (and choosing $\eps_9$ to be small enough), the Lebesgue measure of $S_2$ can be ensured to be less than $\eps_9.$
\end{observation} 
\begin{claim}\lab{cl:5-Dec-26-2023} There exists $\eps_9$ depending on $h'$ and $\eps_{8.5}$ alone such that
we can choose $a''$ large enough that $$\exp\left(- 2 \sum_{\kappa \in \D_{a''}\setminus (S')}|\kappa| \sigma\left((-1)|\kappa|^{-1}\int_{\kappa^o}\hess h'(dx)\right)\right) V(\nu') < \J(h') + \eps_{8.5}^2$$ for any set $S' \subseteq \D_{a''}$ whose Lebesgue measure is less or equal to $2\eps_9.$ Here $h' \in H(\la, \mu; \nu')$.
\end{claim}
\begin{proof}
%This follows from the fact that $\J(h') > \eps_{8.5} > 0$ and Theorem~\ref{cor:24}.
Since $\sigma(-(\nabla^2h')_{ac}) \in L^1,$ we see that there must exist an $\eps_9$ such that for all measurable subsets $S$ of $T$ having Lebesgue measure less than $2\eps_9$, 
 $$\int_S |\sigma(-(\nabla^2h')_{ac}) |\Leb_2(dx) < c \eps_{8.5}^2.$$
 
By Theorem~\ref{cor:24}, there exists an $a_1$ such that for all $a_3 \geq a_1$,  $\J_{a_3}(h') < \J(h') + \frac{\eps_{8.5}^2}{2}.$
By Theorem~\ref{cor:24}   and the convexity of $\sigma$, there exists an $a_0$ such that for all $a_2 \geq a_0$,  we have 
\beqs \sum_{\kappa \in \D_{a_2}} \left|\int_\kappa  \left(\sigma\left((-1)|\kappa|^{-1}\int_{\kappa^o}\left(\hess h'(dx)\right)\right) - \sigma\left(-(\hess h')_{ac}\right) \right) \Leb_2(dx) \right| &=& \\
\sum_{\kappa \in \D_{a_2}} (-1)\int_\kappa \left(\sigma\left((-1)|\kappa|^{-1}\int_{\kappa^o}\left(\hess h'(dx)\right)\right) - \sigma\left(-(\hess h')_{ac}\right)\right) \Leb_2(dx) & \leq &\\
\sum_{\kappa \in \D_{a_2}} (-1)\int_\kappa \left(\sigma\left((-1)|\kappa|^{-1}\int_{\kappa}\left(\hess h'(dx)\right)\right) - \sigma\left(-(\hess h')_{ac}\right)\right) \Leb_2(dx) <
c \eps_{8.5}^2.\eeqs

Choosing $a'' \geq \max(a_0, a_1)$ satisfies the conditions of this claim.
\end{proof}
%Therefore, the  total Lebesgue measure of squares in $S_2$ is bounded above by $\eps_{9}.$
%We see that $$\exp\left(\sum_{\kappa \in S_2}  |\kappa| \sigma\left((-1)|\kappa|^{-1}\int_{\kappa^o}\hess h'(dx)\right) \right)> c(\la, \mu) \J(h').$$
%\beq \nonumber\left| \exp\left(-\sigma\left((-1)|\kappa|^{-1}\int_{\kappa^o}\hess h''(dx)\right)\right)   -  \exp\left(-\sigma\left((-1)|\kappa|^{-1}\int_{\kappa^o}\hess h'(dx)\right)\right) \right|\\ < C\eps_{8.5}.\eeq
For the set $S = S_2 \cup S_3$ of remaining dyadic triangles and dyadic cubes,  %from 

%\beq \sigma\left((-1)|\kappa|^{-1}\int_{\kappa^o}\hess h''(dx)\right) \leq \log \frac{1}{\eps_{8.5}},\eeq we have 
%the concavity of $-\sigma$ and the bound $\exp(-\sigma((-1)\nabla^2h)) \leq \frac{e \tr((-1)\nabla^2 h)}{3}$ (from Lemma~\ref{lem:2.8}), we have the bound 
 \beq \sum_{\kappa \in S} |\kappa|  \sigma\left((-1)|\kappa|^{-1}\int_{\kappa^o}\hess h''(dx)\right) \geq   (-1) \eps_{8.5}\log \frac{C(\la, \mu)}{\eps_{8.5}}, \lab{eq:Clamu} \eeq
where $C(\la,  \mu)$ is a finite constant that can be fixed (as seen from Lemma~\ref{lem:3-Nov20}) depending on $\la$ and $\mu$ alone,  which bounds from above, (up to multiplicative $C$) the Lipschitz constant and hence the magnitude of the total measure of the trace of the Hessian of  any rhombus concave function in $H(\la, \mu)$.
We now choose $\eps_8$ 
(such a value of $\eps_8$ exists by Lemma~\ref{lem:39}),
so that when $N$ is chosen to satisfy (\ref{eq:N-14Jul23}),  we have
for every $h'' \in \bigcup\limits_{\nu'\in \overline{B}_\I(\nu, \eps)}H(\la, \mu; \nu')$ such that %\ben \item 
$\|h'' - h'\|_\infty < \frac{1}{N}$, 
%and \item $\inf\limits_T \min\limits_{0 \leq i \leq 2}  (D_i h'')_{ac} > \tilde{\de} > 0,$
%\een 
the following to be true (using the last part of Lemma~\ref{lem:41}, to get the second step,  and taking $h'' \in H(\la, \mu; \nu'')$ in the first step; where we use Observation~\ref{obs:A12Jul23} and Claim~\ref{cl:5-Dec-26-2023}).
\beqs  \exp\left(-\sum_{\kappa \in \D_{a''}\setminus(S_2 \cup S_3)}|\kappa| \sigma\left((-1)|\kappa|^{-1}\int_{\kappa^o}\hess h''(dx)\right)\right) V(\nu'')< \eeqs 
\beqs \exp\left(-\sum_{\kappa \in \D_{a''}\setminus (S_2 \cup S_3)}|\kappa| \sigma\left((-1)|\kappa|^{-1}\int_{\kappa^o}\hess h'(dx)\right)\right)V(\nu')\exp\left( C \eps_{8.5}\right) < \eeqs
\beqs(\J(h') +  \eps_{8.5}^2)\exp\left( C \eps_{8.5}\right).\eeqs%Putting this together with the last part of  Lemma~\ref{lem:41},  we are done.
As stated in the beginning of this proof, we let $a''$ be an arbitrary integer greater or equal to the $a'$ appearing in the statement of Lemma~\ref{lem:37}.

%\beqs  \exp\left(-\sum_{\kappa \in \D_{a''}\setminus(S_2 \cup S_3)}|\kappa| \sigma\left((-1)|\kappa|^{-1}\int_{\kappa^o}\hess h''(dx)\right)\right) V(\nu'')\exp\left( C \eps_{8.5}\right).\eeqs 
Together with Observation~\ref{obs:A12Jul23} and (\ref{eq:Clamu}), we have that 

\beqs  \J_{a''}(h'') < (\J_{a''}(h') +  \eps_{8.5}^2)\exp\left( C \eps_{8.5} + \eps_{8.5}\log \frac{C(\la, \mu)}{\eps_{8.5}}\right).\eeqs

This completes the proof of this lemma.
\end{proof}

\begin{lemma}\lab{lem:39new}
  Fix $\eps_{8.5} > 0$,   and $h \in \bigcup\limits_{\nu'\in \overline{B}_\I(\nu, \eps)}H(\la, \mu; \nu'),$ such that $\J(h) = 0.$ %that satisfies $$\inf\limits_T \min\limits_{0 \leq i \leq 2}  ((-1)D_i h')_{ac} > \tilde{\de} > 0.$$ 
 For all sufficiently large $a''$,  there exists $N > 0$ such that 
for every $h'' \in \bigcup\limits_{\nu'\in \overline{B}_\I(\nu, \eps)}H(\la, \mu; \nu')$ such that %\ben \item 
$\|h'' - h\|_\infty < \frac{1}{N}$, 
%and \item $\inf\limits_T \min\limits_{0 \leq i \leq 2}  (D_i h'')_{ac} > \tilde{\de} > 0,$
%\een 

\beqs  \J_a(h'') <  \eps_{8.5}.\eeqs
\end{lemma}
\begin{proof}
Let $h' = h + \de \tau,$ for some positive $\de$ such that $\J(h') < \eps_{8.5}/2.$ Such a $\de$ exists by the proof of Theorem~\ref{cor:24}.
By the preceding Lemma~\ref{lem:38}, we see that there is an $L^\infty$ neighborhood of $h'$ of radius $1/N$ for which \beqs  \J_a(h'') < (\J_a(h') +  \eps_{8.5}^2)\exp\left( C \eps_{8.5}\right).\eeqs
This lemma  follows by relabelling $(\eps_{8.5}/2 + \eps_{8.5}^2) \exp\left(C\eps_{8.5}\right)$ as $\eps_{8.5},$  and noting that a $1/N$ neighborhood of $h$ is translated into a $1/N$ neighborhood of $h'$ under addition by $\de \tau.$

\end{proof}

\begin{lemma}\lab{lem:38.5}Let $\la, \mu$ be strongly decreasing functions that integrate to $0$ over $[0, 1].$
  The functional $\J$ is upper semicontinuous on $H(\la, \mu)$ with respect to the $L^\infty(T)$ metric.
\end{lemma}
\begin{proof}

  Fix $\eps_9>0.$ Then, by Lemma~\ref{lem:38} and Lemma~\ref{lem:39new}, 
for all sufficiently large $a$, there exists $N > 0$ such that 
for every $h'' \in H(\la, \mu)$ %H^{\tilde{\de}}(\la^{- \tilde{\de}},  \mu^{-\tilde{\de}})$ 
such that %\ben \item 
$\|h'' - h'\|_\infty < \frac{1}{N}$, 
\beqs  \J_{a}(h'') < \J_{a}(h')\exp(\eps_9) + \eps_9.\eeqs 
For any $a > a'$, $\J_{a'} - \J_a$ is a nonnegative functional  on $H(\la, \mu).$ 
Therefore,  $\lim_{a' \ra \infty} \J_{a'}$ (which by Corollary~\ref{cor:cor} equals $\J$) is upper semicontinuous at $h'$.   Thus,  $\J$ is an upper semicontinuous functional on 
 $H(\la, \mu)$ with respect to the $L^\infty(T)$ metric.
\end{proof}
In the following lemma, let $h_{\ast} \in \bigcup\limits_{\nu'\in \overline{B}_\I(\nu, \eps)}H(\la, \mu; \nu')$ be such that
 for all positive $N$ and $0 \leq \de < \tilde{\de}$, we have
\beqs \limsup\limits_{n \ra \infty} \bigg|G_n(\la_n, \mu_n; \nu_n, \eps n^2)\cap \left(B_{\infty}^{n + 1 \choose 2}\left(L_n(\de h + (1 - \de) h_\ast), \frac{n^2}{N}\right)\oplus \R^{U_n}\right)\bigg|^\frac{1}{n^2} \eeqs \beqs\geq \limsup\limits_{n \ra \infty} |G_n(\la_n, \mu_n; \nu_n, \eps n^2)|^\frac{1}{n^2}(1 - \de),\eeqs
where $h \in \bigcup\limits_{\nu'\in \overline{B}_\I(\nu, \eps)}H(\la, \mu; \nu')$ satisfies $$\inf\limits_T \min\limits_{0 \leq i \leq 2}  (-1)(D_i h)_{ac} > \tilde{\de} > 0.$$ Such an $h$ exists as a consequence of Lemma~\ref{lem:Hlm-nonempty} (by taking a convex combination with the rhombus concave function defined there).
The existence of such a $h_*$ is guaranteed by Lemma~\ref{lem:11}.
\begin{lemma}\lab{lem:upper} 
Let $h'' =\de h + (1 - \de) h_\ast$.
 For any fixed $\eps_{10} > 0$, the following is true. For any $a \in \N$, there exists
an $N \in \N$ such that  for all   $n$ larger than some threshold (possibly depending on $a$ and $N$), 
\beqs  \bigg|G_n(\la_n, \mu_n; \nu_n, \eps n^2)\cap \left(B_\infty^{n + 1 \choose 2}\left(L_n(h_\ast), \frac{n^2}{N}\right)\oplus \R^{U_n} \right)\bigg|^\frac{2}{n^2} (1 - \de)^2 < \J_a(h'')\exp( \eps_{10}).\eeqs
\end{lemma}
\begin{proof}  Let $\hat{Q}_n := {\de}h_n + (1 - {\de})\left(G_n(\la_n, \mu_n; \nu_n, \eps n^2)\cap \left(B_\infty^{n + 1 \choose 2}\left(L_n(h_{\ast}), \frac{n^2}{N}\right)\oplus \R^{U_n} \right)\right).$ 
Let $\hat{Q}:= {\de}h + (1 - {\de})\left(G(\la, \mu; \nu, \eps )\cap \left(B_\infty \left(h_{\ast}, \frac{1}{N}\right)\oplus \R^{U} \right)\right)$, where  $U$ is the analogue of $U_n$ in the continuous setting. 

For each $n$, we see that 
$$ \bigg|G_n(\la_n, \mu_n; \nu_n, \eps n^2)\cap \left(B_\infty^{n + 1 \choose 2}\left(L_n(h_{\ast}), \frac{n^2}{N}\right)\oplus \R^{U_n} \right)\bigg|^\frac{2}{n^2}(1 - {\de})^2 = |\hat{Q}_n|^\frac{2}{n^2}$$

Fix $\eps_9 > 0$.
Let $\bb$ be defined by Definition~\ref{def:37} and $\tilde{Q}_n$ be as in Definition~\ref{def:6.1}.
For all $h' \in \hat{Q}$, 
 the following is true from Lemma~\ref{lem:46} for all sufficiently large $n$. 
For any $h' \in \hat{Q}$, $$|\tilde{Q}_n(\bb, L_n(h'))|^\frac{2}{n^2} < \J_{a}(h')\exp( \eps_9).$$
For $N$ large enough,  this is in turn less than $\J_a(h'')\exp(2\eps_9)$ by Lemma~\ref{lem:38}, where $h'' = {\de}h + (1 - {\de})h_\ast.$
Note that the dimension of $\tilde{Q}_n(\bb, L_n(h''))$ is $n^2$.
$$G_n(\la_n, \mu_n; \nu_n, \eps n^2)\cap \left(B_\infty^{n + 1 \choose 2}\left(L_n(h_\ast), \frac{n^2}{N}\right)\oplus \R^{U_n} \right)$$ can be covered by $n^{Cn}$ sets of the form $\tilde{Q}_n(\bb, L_n(h')),$ where the $h'$ all belong to $\hat{Q}.$
%Observe from Definition~\ref{def:37} that $\lim\sup_{n \ra \infty} |\bb|/n$ is a finite constant depending only on $a$. 
Therefore, 
%by Fubini's Theorem, choosing $2\eps_9 < \eps_{10}$ 
we have, for all sufficiently large $n$,
\beqs  \bigg|G_n(\la_n, \mu_n; \nu_n, \eps n^2)\cap \left(B_\infty^{n + 1 \choose 2}\left(L_n(h_\ast), \frac{n^2}{N}\right)\oplus \R^{U_n} \right)\bigg|^\frac{2}{n^2} (1 - \de)^2 & < & \left(n^{Cn} \sup\limits_{h' \in \hat{Q}} |\tilde{Q}_n(\bb, L_n(h'))|\right)^\frac{2}{n^2}\\
& < &  \J_a(h'')\exp( \eps_{10}).\eeqs
\end{proof}

\begin{lemma}\lab{lem:55-new}
Let $\la:[0, 1]\ra \R$ and $\mu:[0, 1] \ra \R$ be strongly decreasing functions that integrate to $0$ over $[0, 1].$ Let $\nu:[0, 1] \ra \R$ be a decreasing function that integrates to $0$ over $[0, 1]$.  Let $X_n$ and $Y_n$ be independent random Hermitian matrices with spectra $\la_n$ and $\mu_n$ respectively.  Let $Z_n = X_n + Y_n.$ Then,
$$\limsup\limits_{n \ra \infty}\p_n\left[\spec(Z_n) \in {B}_\I^n(\nu_n, \eps n^2)\right]^\frac{2}{n^2} \leq \sup\limits_{\nu'\in {B}_\I(\nu, \eps)}\sup\limits_{h' \in H(\la, \mu; \nu')} \left(V(\la)V(\mu)\right)^{-1}\J(h').$$
\end{lemma}
\begin{proof}
By (\ref{eq:3.2}), we see that 
\beqs \limsup\limits_{n \ra \infty}\p_n\left[\spec(Z_n) \in {B}_\I^n(\nu_n, \eps n^2)\right]^\frac{2}{n^2} = \limsup\limits_{n \ra \infty} \left(\frac{V_n(\la_n)V_n(\mu_n)}{V_n(\tau_n)^2}\right)^{-\frac{2}{n^2}}\bigg|G_n(\la_n, \mu_n; \nu_n, \eps n^2)\bigg|^\frac{2}{n^2}.\eeqs
By Lemma~\ref{lem:16},  Lemma~\ref{lem:11},  Lemma~\ref{lem:38} and Lemma~\ref{lem:upper}, we see that 
\beqs \limsup\limits_{n \ra \infty} \left(\frac{V_n(\la_n)V_n(\mu_n)}{V_n(\tau_n)^2}\right)^{-\frac{2}{n^2}}\bigg|G_n(\la_n, \mu_n; \nu_n, \eps n^2)\bigg|^\frac{2}{n^2} & \leq & \lim\limits_{a \ra \infty}\frac{ \left(V(\la)V(\mu)\right)^{-1}\J_a(\de h + (1-\de) h_*)}{(1 - \de)^2}\\ 
& = &  \frac{ \left(V(\la)V(\mu)\right)^{-1}\J(\de h + (1-\de) h_*)}{(1 - \de)^2}.\eeqs
Therefore, 
\beqs \limsup\limits_{n \ra \infty} \left(\frac{V_n(\la_n)V_n(\mu_n)}{V_n(\tau_n)^2}\right)^{-\frac{2}{n^2}}\bigg|G_n(\la_n, \mu_n; \nu_n, \eps n^2)\bigg|^\frac{2}{n^2} & \leq & 
\liminf\limits_{\de \ra 0}  \left(V(\la)V(\mu)\right)^{-1}\J(\de h + (1-\de) h_*)\\
& \leq & \sup\limits_{\nu'\in {B}_\I(\nu, \eps)}\sup\limits_{h' \in H(\la, \mu; \nu')} \left(V(\la)V(\mu)\right)^{-1}\J(h').\eeqs

\end{proof}

\section{Large deviations for random hives}

\begin{lemma}\lab{lem:41-new2} Let $h \in H(\la, \mu; \nu)$ be a hive. Let $\tih \in H(\tla, \tmu; \tnu)$ be a $C^2$ hive such that $h - \tih$ is a hive. For any fixed $\eps_0 > 0$,

 $$ V(\tnu)\exp\left(-\int_T 2 \sigma((-1)\hess \tih)\Leb_2(dx)\right) \leq $$  $$ \liminf\limits_{n \ra \infty} \bigg|G_n(\la_n, \mu_n; \nu_n, \eps_0 n^2) \cap \left(B^{n+1 \choose 2}_\infty(L_n(h), n^2\eps_0) \oplus \R^{U_n}\right)\bigg|^\frac{2}{n^2}.$$
 %Note: Write about the $\eps_i$.
\end{lemma}
\begin{proof} ({\bf case\,\,1}): First suppose that $\tilde{h} \in H^{\tilde{\de}}(\tla, \tmu)$ for some positive $\tilde{\de}$.
We shall analyze the polytope $\tilde{Q}_n(\bb, h)$ given by Definition~\ref{def:6.1}. Since 
$$\tilde{Q}_n(\bb, h) \subseteq G_n(\la_n, \mu_n; \nu_n, \eps_0 n^2) \cap \left(B^{n+1 \choose 2}_\infty(L_n(h), n^2\eps_0) \oplus \R^{U_n}\right),$$ it suffices to get a suitable lower bound on $|\tilde{Q}_n(\bb, h)|$ to prove this Lemma. The map $h' \mapsto h' + h - \tih$
translates $\tilde{Q}_n(\bb, \tih)$ into $\tilde{Q}_n(\bb, h)$. Therefore it suffices to get a suitable lower bound on $|\tilde{Q}_n(\bb, \tih)|.$
This follows from Lemma~\ref{lem:41-new3}.
                         
({\bf case \,\,2}): Now suppose $\tilde{h} \not\in H^{\de}(\tla, \tmu)$ for all positive $\de$. For the span of this proof, let $q:T \ra \R$ be given by $q(x, y) = a' - (x^2 + y^2 -xy)$, where $a'$ is the unique affine function chosen so as to make $q(0, 0) = q(0, 1) = q(1, 1) = 0.$

 By Lemma~\ref{lem:Hlm-nonempty}, $H^{\tilde{\de}}(\la, \mu)$ is nonempty for some positive $\tilde{\de}$. Let $\hat{h} \in  H^{\tilde{\de}}(\la, \mu)$.   For all sufficiently small $\de$, $h_\de := \de \hat{h} + (1 - \de) h$ satisfies
\beq L_n(h_\de) \in G_n(\la_n, \mu_n; \nu_n, \frac{\eps_0 n^2}{10}) \cap \left(B^{n+1 \choose 2}_\infty(L_n(h), \frac{n^2\eps_0}{10}) \oplus \R^{U_n}\right). \lab{eq:Ln}\eeq 
Suppose that $\de < \tilde{\de}$. Then,  $ \hat{h} - \de q$ is a hive.
Also, for $\tilde{h}_\de := (1-\de) \tilde{h} + \de^2 q \in H(\tla_\de, \tmu_\de; \tnu_\de)$, we have that
\beq \lim\limits_{\de \ra 0} V(\tnu_\de)\exp\left(-\int_T 2\sigma((-1)\hess \tih_\de)\Leb_2(dx)\right) \geq \nonumber\\ V(\tnu)\exp\left(-\int_T 2\sigma((-1)\hess \tih)\Leb_2(dx)\right). \lab{eq:Ln2}\eeq
Also, $h_\de - \tilde{h}_\de =  (1 - \de) (h - \tih) + \de(\hat{h} - \de q)$ is a hive. By ({\bf case\,\,1}), we see that 
 $$ V(\tnu_\de)\exp\left(-\int_T 2\sigma((-1)\hess \tih_\de)\Leb_2(dx)\right) \leq $$  $$ \liminf\limits_{n \ra \infty} \bigg|G_n(\la_n, \mu_n; \de\hat{\nu}_n + (1 - \de)\nu_n, (\eps_0 - C\de)n^2) \cap \left(B^{n+1 \choose 2}_\infty(L_n(h_\de), n^2(\eps_0 - C \de)) \oplus \R^{U_n}\right)\bigg|^\frac{2}{n^2}.$$
 However,
 $$G_n(\la_n, \mu_n; \de\hat{\nu}_n + (1 - \de)\nu_n, (\eps_0 - C\de)n^2) \cap \left(B^{n+1 \choose 2}_\infty(L_n(h_\de), n^2(\eps_0 - C \de)) \oplus \R^{U_n}\right) \subseteq $$
 $$G_n(\la_n, \mu_n; \nu_n, \eps_0 n^2) \cap \left(B^{n+1 \choose 2}_\infty(L_n(h_\de), n^2 \eps_0) \oplus \R^{U_n}\right).$$
This proves the Lemma.

\end{proof}

\begin{notation}
Let $H:= \bigcup_{\la}\bigcup_{\mu}\bigcup_{\nu}H(\la, \mu; \nu)$.
As in Definition~\ref{def:77old}, let $H(\la, \mu):= \bigcup_{\nu'}H(\la, \mu; \nu')$.

For $h' \in H(\la, \mu; \nu')$, let \beq \lab{eq:I1} I_1(h') := - \log\left(\frac{\J(h')}{V(\la)V(\mu)}\right).\eeq
For the purposes of this section, let $\p_n$ denote the measure on $H_n(\la_n, \mu_n) := \bigcup_{\nu'_n}H_n(\la_n, \mu_n; \nu'_n)$ obtained by restricting a uniformly random augmented hive from 
$\bigcup_{\nu'_n}A_n(\la_n, \mu_n; \nu'_n)$ to $T_n$.
Note that 
\beqs I_1(h') =  \log\left(\frac{V(\la)V(\mu)}{V(\nu')}\right) +   \int_T 2\sigma((-1)(\hess h')_{ac})\Leb_2(dx) .\eeqs
\end{notation}

\begin{thm}[Lower bound in terms of $C^2$ hives]\lab{thm:6}
For any $\eps > 0$ and $h_\ast \in H(\la, \mu; \nu),$ 
\beqs \liminf\limits_{n \ra \infty}\left(\frac{2}{n^2}\right)\log \p_n\left[h_n \in B_\infty^{n + 1 \choose 2}\left(L_n(h_{\ast}), n^2 \eps\right)\right] \geq \sup\limits_{\substack{h' \in B_\infty (h_{\ast}, \eps)\cap H(\la, \mu)\\h' - \tih \in H\\ \tih \in C^2(T)\cap H}}  \log\left(\frac{\J(\tih)}{V(\la) V(\mu)}\right).\eeqs
\end{thm}
\begin{proof}
This follows immediately from Lemma~\ref{lem:41-new2}.
\end{proof}

The following upper bound holds for all bounded monotonically decreasing $\la$ and $\mu$.

\begin{thm}[Upper bound in terms of $C^0$ hives]\lab{thm:7}
For any $\eps > 0$ and $h_\ast \in H(\la, \mu; \nu),$
\beqs \limsup\limits_{n \ra \infty}\left(\frac{2}{n^2}\right)\log \p_n\left[h_n \in B_\infty^{n + 1 \choose 2}\left(L_n(h_{\ast}), n^2 \eps\right)\right] \leq - \inf\limits_{h' \in B_\infty (h_{\ast},  \eps)\cap H(\la, \mu)} I_1(h').\eeqs
\end{thm}
\begin{proof}
The map $\phi: h'' \mapsto \eps_0 (h_\ast) + (1 - \eps_0) h''$ by passing to $L_n$ gives rise to a map $\phi_n$ from $H_n(\la_n, \mu_n) \cap B_\infty^{n + 1 \choose 2}\left(L_n(h_{\ast}), n^2 \eps \right)$ into $H_n(\la_n, \mu_n) \cap B_\infty^{n + 1 \choose 2}\left(L_n(h_{\ast}), n^2 \eps(1 - \eps_0)\right),$ and since this map has a Jacobian determinant that equals $(1 - \eps_0)^{n \choose 2}$,
we observe the following.  \beq \lim_{\eps_0 \ra 0^+} \limsup\limits_{n \ra \infty}\left(\frac{2}{n^2}\right)\log \p_n\left[h_n \in B_\infty^{n + 1 \choose 2}\left(L_n(h_{\ast}), n^2 \eps(1 - \eps_0)\right)\right] \lab{eq:thm7.1}\eeq \beqs = \,\limsup\limits_{n \ra \infty}\left(\frac{2}{n^2}\right)\log \p_n\left[h_n \in B_\infty^{n + 1 \choose 2}\left(L_n(h_{\ast}), n^2 \eps \right)\right].\eeqs
 Fix $ a \in \N$ and $\hat{h} \in \bigcup\limits_{\nu'\in \overline{B}_\I(\nu, \eps)}H(\la, \mu; \nu')$ that satisfies
 % \ben
 %\item 
 $$\inf\limits_T \min\limits_{0 \leq i \leq 2}  ((-1)D_i \hat{h})_{ac} >\tilde{\de} > 0.$$ 
 %and 
 %\item $$\sup_T |\hat{h} - h_\ast| < \frac{\eps}{2}.$$
 %\een
 Such a choice of $\hat{h}$ exists from Lemma~\ref{lem:Hlm-nonempty} and taking convex combinations with $h_\ast$.
 Then we see that for each positive $\eps_0$, for all sufficiently small $\hat{\de}$, 
 \beqs \limsup\limits_{n \ra \infty}\left(\frac{2}{n^2}\right)\log \p_n\left[h_n \in B_\infty^{n + 1 \choose 2}\left(L_n(h_{\ast}), n^2 \eps(1 - \eps_0)\right)\right]\eeqs
 is less or equal to  \beq \limsup\limits_{n \ra \infty}\left(\frac{2}{n^2}\right)\log \p_n\left[\hat{\de} \hat{h}_n +(1 - \hat{\de})h_n \in B_\infty^{n + 1 \choose 2}\left(L_n( h_{\ast}), n^2\eps\right)\right]\lab{eq:thm7.2}.\eeq It follows that the limit as $\hat{\de}$ tends to $0$ of (\ref{eq:thm7.2}) is greater or equal to  (\ref{eq:thm7.1}).
% But also, by convexity, if $$h_n \in B_\infty^{n + 1 \choose 2}\left(L_n(h_{\ast}), n^2 \eps \right),$$ then $$\hat{\de} \hat{h}_n +(1 - \hat{\de})h_n \in B_\infty^{n + 1 \choose 2}\left(L_n(h_{\ast}), n^2 \eps \right).$$
 % Thus, the limit as $\hat{\de}$ tends to $0$ of (\ref{eq:thm7.2}) is equal to  (\ref{eq:thm7.1}).
 We now fix $\hat{\de}$ and obtain an upper bound on (\ref{eq:thm7.2}) in terms of $I_1$.
 
 Fix $\eps_9 > 0$.  There is a positive $n_0$, such that following is true for all  $n > n_0$. Let $\bb$ be defined by Definition~\ref{def:37} and let  $\tilde{Q}_n$ be as in Definition~\ref{def:6.1}.  Then, by Lemma~\ref{lem:46}, for every $\hat{\de} \leq \tilde{\de},$ $h'' = \hat{\de}\hat{h} + (1 - \hat{\de}){h_{**}} \in H(\la, \mu, \nu''),$ where $h_{**} \in \bigcup\limits_{\nu'\in \overline{B}_\I(\nu, \eps)}H(\la, \mu; \nu')$,  $$|\tilde{Q}_n(\bb, L_n(h''))|^\frac{2}{n^2} < \J_{a}(h'')\exp( \eps_9).$$
 We note that for each fixed $a$, the number of sites in $\bb$ is bounded above by $Cn$.

% the dimension of the slice $\tilde{P}_n(\bb, L_n(h''))$ exceeds ${n \choose 2} - Cn$. 
Note that the $\ell_\infty$ diameter of $\left(\hat{\de} \hat{h}_n + (1 - \hat{\de})H_n(\la_n, \mu_n)\right) \cap B_\infty^{n + 1 \choose 2}\left(L_n(h_{\ast}), n^2 \eps\right)$ is less than $n^C$ and  the sum of the  volumes of the elements in its cover $\C$ consisting of some  $n^{Cn}$ polytopes of the form $\tilde{Q}_n(\bb, L_n(h''))$
satisfies 
$$(1 - \hat{\de})^2 \left(\frac{V_n(\la_n)V_n(\mu_n)}{V_n(\tau_n)^2}\right)^{\frac{2}{n^2}}\left(\p_n\left[\hat{\de} \hat{h}_n +(1 - \hat{\de})h_n \in B_\infty^{n + 1 \choose 2}\left(L_n( h_{\ast}), n^2\eps\right)\right]\right)^\frac{2}{n^2} < $$
$$\left(\sum\limits_{\tilde{Q}_n(\bb, L_n(h'')) \in \C} \left|\tilde{Q}_n(\bb, L_n(h''))\right|\right)^{\frac{2}{n^2}} < \sup_{h'' \in \left(\hat{\de}\hat{h} + (1 - \hat{\de})H(\la, \mu)\right) \cap B_\infty\left(h_\ast, \eps\right)}\J_{a}(h'')\exp( \eps_9).$$

Note that for each $h'' = \hat{\de}\hat{h} + (1 - \hat{\de}){h_{**}} \in H(\la, \mu, \nu''),$ $\lim_{a \ra \infty}\J_a(h'') = \J(h''),$ from Corollary~\ref{cor:cor}.
By Lemma~\ref{lem:38}, Lemma~\ref{lem:39new} and the compactness of $$\left(\hat{\de} \hat{h} + (1 - \hat{\de})H(\la, \mu)\right) \cap B_\infty\left(h_{\ast},   \eps\right)$$ in the $L^\infty$ topology by the Arzela-Ascoli theorem (the rhombus concave functions in $H(\la, \mu)$ are all Lipschitz with the same finite Lipschitz constant by Lemma~\ref{lem:3-Nov20}) $$\sup_{h'' \in \left(\hat{\de}\hat{h} + (1 - \hat{\de})H(\la, \mu)\right) \cap B_\infty\left(h_\ast, \eps\right)}\J_{a}(h'')\exp( \eps_9),$$ converges as $a \ra \infty$ to 
$$\sup_{h'' \in \left(\hat{\de}\hat{h} + (1 - \hat{\de})H(\la, \mu)\right) \cap B_\infty\left(h_\ast, \eps\right)}\J(h'')\exp( \eps_9).$$ By Lemma~\ref{lem:10},  $\lim\limits_{n \ra \infty} \left(\frac{V_n(\tau_n)^2}{V_n(\la_n)V_n(\mu_n)}\right)^\frac{2}{n^2} = \frac{1}{V(\la)V(\mu)}.$
The theorem follows.

\end{proof}

\begin{figure}
    \centering
        \includegraphics[width=2.5in]{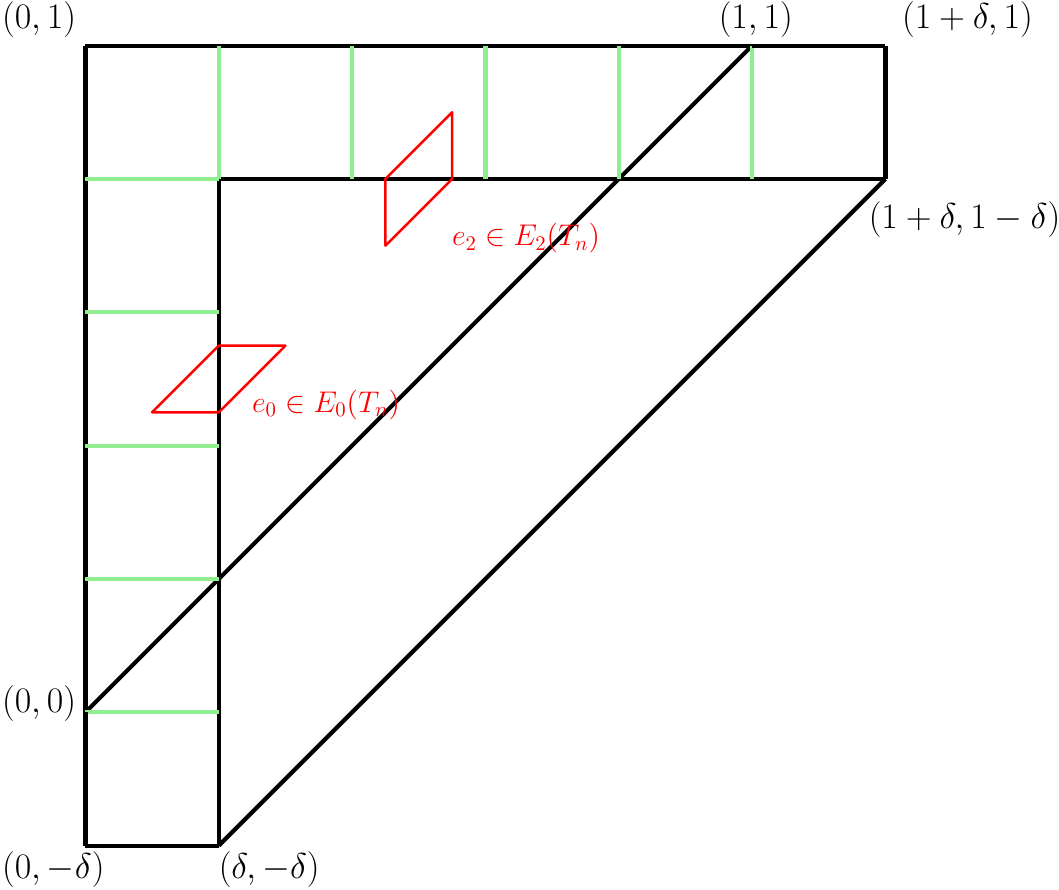}
        \caption{Construction of a  hive $h^{(2)}$ from $h^{(1)}$.  A crucial first step is inserting strips near the horizontal and vertical boundary. We then shift $h^{(1)}$ and paste it to the strips.  On a particular green line segment,  the value taken is constant.  $h^{(2)}$ is rhombus concave, as can be seen by analyzing $\De_0$ and $\De_2$ on the two rhombi $e_0 \in E_0(T_n)$ and $e_2 \in E_2(T_n)$ for some positive integer $n$ that have been depicted,   though $h^{(2)}$ is not in general differentiable.} \label{fig:C2hive1}
    \end{figure}
  %  \begin{defn}
%Given two $C^2$ functions ...{\color{red} do this}
%\end{defn}

\begin{thm}[Equality when $\la$ and $\mu$ are $C^1$]\lab{thm:8} Suppose $\la$ and $\mu$ are strongly decreasing and belong to $C^1[0, 1]$ and satisfy $\int_0^1\la(y)dy = \int_0^1 \mu(x)dx = 0.$
For any $\eps > 0$ and $h_\ast \in H(\la, \mu; \nu),$ 
\beq \lim\limits_{n \ra \infty}\left(\frac{2}{n^2}\right)\log \p_n\left[h_n \in B_\infty^{n + 1 \choose 2}\left(L_n(h_{\ast}), n^2 \eps\right)\right] = -\inf\limits_{h' \in B_\infty (h_{\ast},  \eps)\cap H(\la, \mu)} I_1(h').\eeq
\end{thm}
\begin{proof}
By the lower and upper bounds presented in Theorem~\ref{thm:6} and Theorem~\ref{thm:7}, it suffices to show that 
\beq - \inf\limits_{\substack{h'' \in B_\infty (h_{\ast}, \eps)\cap H(\la, \mu)\\h'' - \tih \in H\\ \tih \in C^2(T)}} I_1(\tih) \geq - \inf\limits_{h' \in B_\infty (h_{\ast},  \eps)\cap H(\la, \mu)} I_1(h').\eeq

We will achieve this as follows. For every $h' \in B_\infty (h_{\ast},  \eps)\cap H(\la, \mu),$  and $\de' > 0$, we will produce $h^{\#} \in B_\infty (h',  C \de')\cap H(\la, \mu),$ such that $h^{\#} \in C^2(T)$ and  $I_1(h^{\#}) < I_1(h') + \de'.$
Let $M$ be chosen such that $|\partial^2 \a(z)|$ and  $|\partial^2\b(z)|$ belong to $[\frac{1}{M}, M]$ for all $0 < z < 1$. Let $\de$ be such that  $\frac{1}{16M} > \de > 0$.
Let $$h^{(1)} := (1 - 2M\de) h' + \ell,$$ where $\ell$ is an affine function chosen so that \ben
\item[A.] $\partial_y \ell + (1 - 2M\de)\partial^- \a(1) = 0,$ and \item[B.] $\partial_x\ell + (1 - 2M\de)\partial^+ \b(0) = 0,$ and \item[C.] $\ell(0, 1) + (1 - 2M\de)h'(0, 1)= 0.$
\een
It is clear that $h^{(1)} - \ell \in B_\infty (h',  C \de)$ because $h'$ is bounded.
The reason we define $h^{(1)}$ with the multiplicative factor of $(1 - 2M\de)$  is that we will  need to perform various operations on it that transform it into a $C^2$ hive $h^{(3)}$ but in the process alter the boundary conditions a little.  In the final step, we will add a hive $\tilde{h}$ to $h^{(3)}$ to give rise to a strongly rhombus concave function $h^\#$ with exactly the right boundary conditions.
In order for the boundary values of $h^{(3)}$ to be such that  adding a suitable rhombus concave function makes the boundary values right,  we have to work with $h^{(1)}$ rather than $h$ in the first place. The additive term of $\ell$ ensures that derivatives at the point $(0, 1)$ are zero, and the value is $0$ as well. This is essential to make the rhombus inequalities hold in $h^{(2)}$ defined below. This is delicate because from left to right, across the vertical line
$\{(x, y)| x = \de,  y \leq 1-\de\},$ there can be a  jump discontinuity in $\partial_x^{-}h^{(2)}$ and from bottom to top,  across the horizontal line
$\{(x, y)| y = 1 - \de,  x \geq \de\},$ there can be a  jump discontinuity in $\partial_y^{-}h^{(2)}.$

Next, define $h^{(2)}$ on $T +_m  \left([0, \de]\times [-\de, 0]\right)$ (where $+_m$ signifies Minkowski sum) as follows. 
\ben
\item For $(x, y) \in F_1(\de):= (T + (\de, - \de))$, let $h^{(2)}(x, y) = h^{(1)}(x - \de, y + \de).$ 
\item For $(x, y) \in  F_2(\de):= [\de, 1 + \de] \times [1 - \de, 1],$  let $h^{(2)}(x, y) =  h^{(1)}(x -\de, 1)$. 
\item For $(x, y) \in F_3(\de) := [0, \de] \times [1-\de, 1],$ let $h^{(2)}(x, y) = 0.$ 
\item For $(x, y) \in F_4(\de):= [0, \de] \times [-\de, 1 - \de],$  let $h^{(2)}(x, y) = h^{(1)}(0, y + \de)$.
\een
$h^{(2)}|_T \in B_\infty (h^{(1)},  C \de)$ because $h^{(1)}$ is Lipschitz.
\begin{claim} The function $h^{(2)}$ is rhombus concave.
\end{claim}
\begin{proof}
We first consider rhombi straddling the  boundary of $F_1(\de)$ or $F_3(\de)$. The rhombi that do not straddle these boundaries can be subdivided into rhombi of the same shape, each of which is entirely contained in one of the $F_i(\de)$,  and which hence satisfy the corresponding rhombus constraint.  For rhombi contained in $F_3(\de) \cup F_2(\de)$, we see that since in this region the value of $h^{(2)}$ is independent of the $x$ coordinate,  and is concave with respect to the $y$ coordinate, the rhombus inequalities are satisfied. The same argument holds for $F_3(\de)\cup F_4(\de).$  Let us now consider the boundary of $F_1(\de).$ There are only two nontrivial cases: $e_0 \in E_0(T_n)$ intersecting $\{\de\}\times [- \de, 1 -\de]$ and $e_2 \in E_2(T_n)$ intersecting $[\de, 1 + \de] \times \{1 - \de\}$  (see Figure~\ref{fig:C2hive1}). These two are analogous, so we only handle the former.  Let the vertices of $e_0$ be $v_0, v_1, v_2, v_3$ starting from the left, and going in anti-clockwise order.
Then, $h^{(2)}(v_0) = h^{(2)}(v_1)$ by construction.  However,  $h^{(1)}$ is rhombus concave, and therefore in view of B.  above,  $h^{(2)}(v_2) \leq h^{(2)}(v_3).$ This proves the claim.
\end{proof}
Recall from Definition~\ref{def:27}, that $\theta$ supported on $[0, 1]^2$ is given for $(x, y) \in [0, 1]^2$ by $$\theta(x, y) :=  \frac{x^4(1-x)^4y^4(1-y)^4}{\int\limits_{(x, y) \in [0, 1]^2}x^4(1-x)^4y^4(1-y)^4 dxdy}$$ and elsewhere by $0$,  and $\theta_\de(x, y) = \de^{-2}\theta(\frac{x}{\de}, \frac{y}{\de}),$ for all $(x, y) \in \R^2.$ We define $\bar{\theta}_\de(x, y):= \theta_\de(-x, y).$
\ben
\item[($\star$)] Let $\de_1$ satisfying $\de^2 > \de_1 > 0$ be such that for any  point $z \in [0, 1]$, if $z' \in [0, 1]$ is such that $|z - z'| < \de_1$, then $|\partial^2 \a(z) - \partial^2\a(z')| < \frac{\de}{256},$ and $|\partial^2\b(z) - \partial^2\b(z')| < \frac{\de}{256}.$ Note that this implies that there exists a  modulus of continuity $\omega_{\partial^2\a}$ of $\partial^2\a$,  that satisfies $\omega_{\partial^2\a}(\de_1) < \frac{\de}{256}.$ An identical statement holds for $\partial^2\beta$.
\een
Next, we define $h^{(3)}$ on $T$ by \beq h^{(3)}(x, y) = (h^{(2)} \ast \bar{\theta}_{\de_1}) - \frac{\de}{16}\left(x^2 + y^2 - xy\right) + \ell',\lab{eq:sharp} \eeq where $h^{(2)}$ is viewed as a function from $T + ([0, \de]\times [-\de, 0])$ to $\R,$ and $\ell'$ is a linear function chosen so that the value taken by $h^{(3)}$ on $(0, 0)$, $(0, 1)$ and $(1, 1),$ equals $0$. 
%Note that $$\|h^{(2)} - h^{(2)}\ast \bar{\theta}_{\de_1}\|_{C^2} < .$$

Now, because $\theta$ is $C^2$ and $- \frac{\de}{16}\left(x^2 + y^2 - xy\right)$ is strongly rhombus concave, $h^{(3)}$ is a strongly rhombus concave hive.  We see that $h^{(2)} $ restricted to $T$ is a hive.  Further, in a $\de$ neighborhood of the vertical boundary it is constant along horizontal lines, and in a $\de$ neighborhood of the horizontal boundary, it is constant along vertical lines.

Also,  $h^{(3)} -\ell'  \in B_\infty (h^{(2)},  C \de)$ because $h^{(2)}$ is Lipschitz, and the convolution of a Lipschitz function with $\theta$ is $C\de'$-close to the given Lipschitz function in the $L^\infty$ norm.
\begin{claim} $\tilde{h}$ given by  $\tilde{h}(x, y) = (\a(y) - h^{(3)}(0, y)) + (\b(x) - h^{(3)}(x, 1)),$ is a hive.
\end{claim}
\begin{proof}
 Since, $\tilde{h}(x, y) = (\a(y) - h^{(3)}(0, y)) + (\b(x) - h^{(3)}(x, 1)),$ it suffices to show that $(\a(y) - h^{(3)}(0, y))$ is a concave function of $y$ and $(\b(x) - h^{(3)}(x, 1))$ is a concave function of $x$. By ($\star$), and  the equalities $\partial_x^2 (h^{(2)} \ast \bar{\theta}_{\de_1}) = (\partial_x^2 h^{(2)}) \ast \bar{\theta}_{\de_1} $ and $\partial_y^2 (h^{(2)} \ast \bar{\theta}_{\de_1}) = (\partial_y^2 h^{(2)}) \ast \bar{\theta}_{\de_1},$ we see that $\tilde{h}$ is a hive (the $\de/16$ in the definition of $h^{(3)}$ could have been replaced by any sufficiently small multiple of $\de$).
\end{proof}
We see that
   $h^\# = h^{(3)} + \tilde{h}$ has the right boundary conditions $\a$ and $\b$ and is a hive.
 We see that $h^\# \in B_\infty (h^{(3)},  C \de)$ because $\tilde{h}$ has a suitably small $L^\infty$ norm.
 We have thus shown that \beq h^{\#} \in B_\infty (h',  C \de')\cap H(\la, \mu).\lab{eq:hsharp} \eeq if $\de$ is chosen to be less than  $\de'.$
It remains to be shown that for sufficiently small $\de$, $$I_1(h^\#) < I_1(h') + \de'.$$
Recall that for $h^\# \in H(\la, \mu; \nu^\#)$, we have \beq e^{-I_1(h^\#)} = \left(\frac{1}{V(\la)V(\mu)}\right)V(\nu^\#) \exp\left(- \int_T 2\sigma((-1)\hess h^\#(x))\Leb_2(dx)\right).\lab{eq:sharp2}\eeq 
We will study the contributions to $\int_T \sigma((-1)\hess h^\#(x))\Leb_2(dx)$ coming from the regions $F_1, F_2, F_3$ and $F_4$ separately, and also obtain bounds for $V(\nu^\#)$.
\ben
\item[(I)] %The fact that $h'$ is Lipschitz and $I_1(h') < \infty$ imply that $$\lim\limits_{\de \ra 0} \int_{T\setminus(T + (-\de, \de))} \sigma((-1)\hess h^\#(x))\Leb_2(dx) = 0.$$   
%For any $(s_1, s_2, s_3) \in \R^3_{> 0},$  by Lemma~\ref{lem:2.8},   $$\min_i \frac{s_i}{2} \leq \f(s) = \exp(-\sigma(s)) \leq e \min_i s_i.$$
%As a consequence,  
%$$\left| \,\,\int\limits_{F_1(\de) \setminus T} \sigma((-1)\hess h^\#(x))\Leb_2(dx)\right| \leq  O\left(\de\log \frac{1}{\de}\right).$$
The second inequality below uses the fact that 
when a function $h^{(2)}$ is convolved with an  approximate identity $\bar{\theta}_{\de_1}$, the Hessian of $h^{(3)}$ equals the convolution of the Hessian of $h^{(2)}$ with $\bar{\theta}_{\de_1}.$ By the convexity of $\sigma$, and the fact that it monotonically decreases along any ray emanating from the origin,  contained in $\R^3_{> 0},$ we have (taking into account boundary effects) the following.
\beqs \int\limits_{T\cap F_1(\de)} \sigma((-1)\hess h^\#(x))\Leb_2(dx) & \leq &  \int\limits_{T \cap F_1(\de)} \sigma((-1)\hess h^{(3)}(x))\Leb_2(dx)\\
 & \leq &  \int\limits_{F_1(\de)} \sigma((-1)\hess h^{(2)}(x))\Leb_2(dx) + O\left(\de \log\left(\frac{1}{\de}\right)\right)\\
& \leq &  \int\limits_T \sigma((-1)\hess h^{(1)}(x))\Leb_2(dx) + O\left(\de \log\left(\frac{1}{\de}\right)\right).\eeqs

Therefore,
$$\lim\sup\limits_{\de \ra 0} \int\limits_{F_1(\de) \cap T} \sigma((-1)\hess h^\#(x))\Leb_2(dx) \leq  \int_{T} \sigma((-1)\hess h'(x))\Leb_2(dx).$$

%Therefore, by the convexity of $\sigma$, and the fact that it monotonically decreases along any ray emanating from the origin,  contained in $\R^3_{> 0},$ 
\item[(II)] The fact that $\b$ is $C^2$ implies that 
$$\lim\limits_{\de \ra 0} \int\limits_{F_2(\de) \cap T} \sigma((-1)\hess h^\#(x))\Leb_2(dx) = 0.$$
\item[(III)] The fact that $h^\#$ is Lipschitz, together with the fact that the side length of the square region $F_3(\de)$ is $\de$, gives us that 
$$\lim\limits_{\de \ra 0} \int\limits_{F_3(\de)\cap T} \sigma((-1)\hess h^\#(x))\Leb_2(dx) = 0.$$
\item[(IV)] The fact that $\a$ is $C^2$ implies that 
$$\lim\limits_{\de \ra 0} \int\limits_{F_4(\de)\cap T} \sigma((-1)\hess h^\#(x))\Leb_2(dx) = 0.$$
\een

Lastly,  we need to analyze $V(\nu^\#).$ Let $v = (1, 1) \in \R^2.$ Then, $\nu^\#$ is a monotonically decreasing function on $[0, 1]$ and $\nu^\#(z)$  equals $\partial_v h^\#(z, z),$ for all $z \in (0, 1)$. The rhombus inequalities imply that $\partial_v h'$ is a monotonically decreasing function along line segments parallel to $v$, defined almost everywhere, and the following ``interlacing'' inequalities hold. Let $u = (1,  0)$. Then, for any two points $z_1 \in T$ and $z_2 \in T$ such that $z_2 - z_1$ is a positive multiple of $u$, then the  directional derivatives satisfy
\beq \lab{eq:int:1} \partial_v^+ h'(z_2) & \leq & \partial_v^+ h'(z_1)\\
\lab{eq:int:1.5}\partial_v^- h'(z_2) & \leq  & \partial_v^- h'(z_1).\eeq Similarly,
if $w = (0,  1)$,  then, for any two points $z_1 \in T$ and $z_2 \in T$ such that $z_2 - z_1$ is a positive multiple of $w$, then the directional derivatives satisfy
\beq \lab{eq:int:2} \partial_v^+ h'(z_2) & \leq & \partial_v^+ h'(z_1),\\
\lab{eq:int:2.5}\partial_v^- h'(z_2) & \leq & \partial_v^- h'(z_1).\eeq

Recall from Definition~\ref{def:32} that  $$V(\nu^\#) = \exp\left(2\int_{T\setminus\{(t, t)|t \in [0, 1]\}}\log\left(\frac{\nu^\#(x) - \nu^\#(y)}{x-y}\right)dxdy\right).$$

Because $$\int\limits_{T\cap\{(t, t')|0<|t-t'| < \de - \de_1\}}\log\left(\frac{\nu'(x) - \nu'(y)}{x-y}\right)dxdy < O(\de \log \frac{1}{\de}),$$ 
and
$$\int\limits_{T\cap\{(t, t')|0<|t-t'| < \de \}}\big|\log\left(\frac{\nu^\#(x) - \nu^\#(y)}{x-y}\right)\big|dxdy < O(\de \log \frac{1}{\de})$$ hold true, 
in order to show 
$$\int\limits_{T}\log\left(\frac{\nu^\#(x) - \nu^\#(y)}{x-y}\right)dxdy$$ is greater or equal to $$\int\limits_{T}\log\left(\frac{\nu'(x) - \nu'(y)}{x-y}\right)dxdy - C\de\log\left(\frac{1}{\de}\right),$$ where $C$ is a positive constant independent of $\de,$ it suffices to prove the following claim.

\begin{claim}\lab{claim:8}
$$\int\limits_{T\setminus\{(t, t')||t-t'| < \de\}}\log\left(\frac{\nu^\#(x) - \nu^\#(y)}{x-y}\right)dxdy$$ is greater or equal to $$\int\limits_{T\setminus\{(t, t')||t-t'| < \de - \de_1\}}\log\left(\frac{\nu'(x) - \nu'(y)}{x-y}\right)dxdy - C^\#\de\log\left(\frac{1}{\de}\right),$$ where $C^\#$ is a positive constant independent of $\de.$
\end{claim}
\begin{proof}
We observe that $$\int\limits_{T\setminus\{(t, t')||t-t'| < \de\}}\log\left(\frac{\nu^\#(x) - \nu^\#(y)}{x-y}\right)dxdy$$ is greater than 
$$\int\limits_{T\setminus\{(t, t')||t-t'| < \de\}}\log\left(\frac{\nu^{(3)}(x) - \nu^{(3)}(y)}{x-y}\right)dxdy$$ by pointwise domination.
We also claim that 
\begin{claim}
$$\int\limits_{T\setminus\{(t, t')||t-t'| < \de\}}\log\left(\frac{\nu^{(3)}(x) - \nu^{(3)}(y)}{x-y}\right)dxdy$$ is greater or equal to 
$$\int\limits_{T\setminus\{(t, t')||t-t'| < \de - \de_1\}}\log\left(\frac{\nu^{(2)}(x) - \nu^{(2)}(y)}{x-y}\right)dxdy - C^{(3)} \de\log\left(\frac{1}{\de}\right)$$ 
%by (\ref{eq:sharp}). 
%$$V(\nu^{(2)}) \leq V(\nu^{(3)}).$$
\end{claim}
\begin{proof}
%{\color{red}{Check carefully, both the statement of the claim and its proof.}}
Let $q(x, y) := x^2 + y^2 - xy.$ Note that $$h^{(3)} = (h^{(2)} - \frac{\de}{16} q) \ast \bar{\theta}_{\de_1}+ a',$$ for a suitable affine function $a'$. For a given $\eps_1$, let $(\de_x, \de_y) \in  [0, \de_1]\times [-\de_1, 0]$.
Define the hive $\check{h}:  T  \ra \R$ by  $\check{h}(x, y) := (h^{(2)} - \frac{\de}{16} q)(x + \de_x,  y + \de_y) + \check{a}(x, y),$ where $\check{a}$ is chosen to 
make the value of $\check{h}$ on the vertices of $T$ to be $0$, and let $\check{\la}, \check{\mu}$ and $\check{\nu}$ be such that $\check{h} \in H(\check{\la}, \check{\mu}; \check{\nu}).$ As a functional on the space of strongly decreasing functions from $[0, 1]$ to $\R$, $V$ is log concave. Therefore, to prove this claim, it suffices to show that when  $\de$ is sufficiently small, for all $(\de_x, \de_y) \in [0, \de_1] \times [-\de_1, 0]$, 
$$V(\nu^{(2)}) \leq V(\check{\nu})\exp(C\de \log \frac{1}{\de}).$$ To see this, 
note that for all $x < - 100\de_1 + y$, by the rhombus concavity of $h^{(2)}$, $$\check{\nu}(x)  - \check{\nu}(y) \geq \nu^{(2)}(x + 10 \de_1) - \nu^{(2)}(y - 10\de_1).$$
We assume that \beq \lab{eq:M'1} \nu^{(2)}(0) \leq \frac{M'}{2}\eeq and \beq \lab{eq:M'2} \nu^{(2)}(1) \geq - \frac{M'}{2}.\eeq

However, because of the presence of the term $ - \frac{\de}{16} q$ in the expression for $\check{h}$, $$\bigg|\int\limits_{\substack{(x, y) \in T\\ |x - y| < 100\de_1}} \log (\check{\nu}(x)  - \check{\nu}(y))dx dy\bigg| < C\de_1 \log \frac{1}{\de}.$$ 
Let $R := T \cap (\{x \leq 10 \de_1\} \cup \{y \geq 1 - 10 \de_1\} \cup\{|x - y| < 80 \de_1\}).$ By (\ref{eq:M'1}) and (\ref{eq:M'2}),
 $$\int\limits_R \log (\nu^{(2)}(x)  - \nu^{(2)}(y))dx dy  <  C(\log M') \de_1.$$

%By the rescaling,
% $$\bigg|\int\limits_{\substack{(x, y) \in T\\ |x - y| > 1 - 20\de_1}} \log (\nu^{(2)}(x)  - \nu^{(2)}(y))dx dy \bigg| <  C\de_1 \log M'.$$
Thus, when $\de$ is sufficiently small, 
for all $(\de_x, \de_y) \in [0, \de_1]\times[-\de_1, 0]$, 
$V(\nu^{(2)}) \leq V(\check{\nu})\exp(C\de \log \frac{1}{\de}),$ which implies the claim.
\end{proof}

Also, using  the fact that 
the bands near the boundary that are not contained in $F_1(\de)$ have thickness $O(\de),$ we see that
$$\int\limits_{T\setminus\{(t, t')||t-t'| < \de - \de_1\}}\log\left(\frac{\nu^{(2)}(x) - \nu^{(2)}(y)}{x-y}\right)dxdy - C^{(3)} \de\log\left(\frac{1}{\de}\right)$$
is greater or equal to
$$\int\limits_{T\setminus\{(t, t')||t-t'| < \de - \de_1\}}\log\left(\frac{\nu'(x) - \nu'(y)}{x-y}\right)dxdy - C'\de\log\left(\frac{1}{\de}\right).$$
\end{proof}
\begin{comment}
\begin{claim}

$$\int\limits_{T\setminus\{(t, t')||t-t'| \geq \de\}}\log\left(\frac{\nu^\#(x) - \nu^\#(y)}{x-y}\right)dxdy$$ is greater or equal to $$\int\limits_{T\setminus\{(t, t')||t-t'| \geq \de - \de_1\}}\log\left(\frac{\nu'(x) - \nu'(y)}{x-y}\right)dxdy - C^\#\de\log\left(\frac{1}{\de}\right),$$ where $C^\#$ is a positive constant independent of $\de.$
\end{claim}
\begin{proof}
{\color{red}{To do.}}
\end{proof}
\end{comment}
Claim~\ref{claim:8} and the discussion immediately preceding it imply that $$\lim\inf_{\de \ra 0} V(\nu^\#) \geq V(\nu').$$

Together with $(I), (II), (III)$ and $(IV)$ above, this implies that for sufficiently small $\de$, $$I_1(h^\#) < I_1(h') + \de',$$  and we are done.

\end{proof}

\begin{lemma}\lab{lem:upperII}
For all $h \in H(\la, \mu),$
\beq   \lim\limits_{\eps \ra 0} \inf\limits_{ h' \in B_\infty(h, \eps)\cap H(\la, \mu)}  I_1(h') = I_1(h).\lab{eq:tpt1}\eeq
\end{lemma}
\begin{proof}
 This follows immediately from Lemma~\ref{lem:38.5},  which states that $\J$ is upper semicontinuous on $H(\la, \mu).$
\end{proof}
\begin{thm}\lab{thm:5new}
Suppose  $\la, \mu$ are strongly decreasing $C^1$ functions.
Given a discrete hive $h_n \in H_n(\la_n, \mu_n; \nu_n)$, we define $h' = \iota_n(h_n)$ to be the  piecewise linear extension of $h_n$. This is a map $h'$ from $T$ to $\R$ such that for any $(x, y) \in T \cap \left(\Z/n \times \Z/n\right)$,  
$h'(x, y) = h_n(nx, ny)/n^2$, that is linear on right triangles with vertices of the form $\{(\frac{i}{n}, \frac{j}{n}), (\frac{i}{n}, \frac{j+1}{n}),  (\frac{i+1}{n}, \frac{j+1}{n})\}$ for $i \leq j$ and $i+1,  j+ 1 \in[n].$
For any Borel set $E\subset L^\infty(T)$, let $\PP_n(E) := \p_n[\iota_n(h_n) \in E].$
Let $a_n := \frac{n^2}{2}.$
For each Borel measurable set 
${\displaystyle E\subset L^\infty(T),}$ 
$${\displaystyle -\inf_{h\in E^{\circ }}I_1(h)\leq \liminf_{n \ra \infty}a_{n}^{-1}\log(\mathcal {P}_{n}(E))\leq \limsup_{n \ra \infty}a_{n}^{-1}\log(\mathcal {P}_{n}(E))\leq -\inf_{h\in {\overline {E}}}I_1(h).}$$
\end{thm}
\begin{proof}
For each Borel measurable set $E$, by covering $E^o$ using $L^\infty$ balls, it follows from Theorem~\ref{thm:8} that  $$-\inf_{h\in E^{\circ }}I_1(h)\leq \liminf_{n \ra \infty}a_{n}^{-1}\log(\mathcal {P}_{n}(E)).$$
Next, suppose for the sake of contradiction, that $$\limsup_{n \ra \infty}a_{n}^{-1}\log(\mathcal {P}_{n}(E)) = -\inf_{h\in {\overline {E}}}I_1(h) + \eps'$$ for some positive $\eps'$.

For all $h' \in \overline {E}$, such that $I_1(h') < + \infty$, let $\de_{h'}$ denote a positive real such that $- \inf_{h'' \in B_\infty(h', \de_{h'})} I_1(h'') < - I_1(h') + \eps'/2.$ The existence of such a positive real is guaranteed by Lemma~\ref{lem:upperII}. Since $\overline E \cap H(\la, \mu)$ (by Arzel\`{a}-Ascoli, in view of the Lipschitz norm being bounded and hence the family being equicontinuous) is compact, there is a finite subset $\La'$ of $\overline E \cap H(\la, \mu) $ such that the union of balls $ B_\infty(h', \de_{h'})$ over all $h' \in \La'$ covers $\overline E \cap H(\la, \mu)$. We see that \beqs \limsup_{n \ra \infty}a_{n}^{-1}\log(\mathcal {P}_{n}(E)) & \leq & \limsup_{n \ra \infty}a_{n}^{-1}\log\left(\mathcal {P}_{n}\left(\bigcup_{h \in \La'} B_\infty(h', \de_{h'})\right)\right)\\ & = & - \inf_{h' \in \La'}\left(\inf_{h \in B_\infty(h', \de_{h'})}I_1(h)\right)\\
& \leq & - \inf_{h' \in \overline{E}}I_1(h') + \eps'/2.\eeqs
This however contradicts $$-\inf_{h\in {\overline {E}}}I_1(h) + \eps' =  \limsup\limits_{n \ra \infty}a_{n}^{-1}\log(\mathcal {P}_{n}(E)).$$
Therefore we must have 
$$\limsup_{n \ra \infty}a_{n}^{-1}\log(\mathcal {P}_{n}(E)) \leq  -\inf_{h\in {\overline {E}}}I_1(h),$$ and we are done.
\end{proof}

\section{Large deviations for the spectrum of the sum of two random matrices}\lab{sec:8}

For a concave function $\g':[0, 1] \ra \R$ such that $\nu' = \partial^-\gamma'$, we define the rate function \beq\lab{eq:I} I(\g') :=  \log\left(\frac{V(\la)V(\mu)}{V(\nu')}\right) + \inf\limits_{h' \in H(\la, \mu; \nu')}  \int_T 2\sigma((-1)(\hess h')_{ac})\Leb_2(dx) .\eeq
\begin{thm} \lab{thm:4} 
Let $X_n, Y_n$ be independent random Hermitian matrices from unitarily invariant distributions with spectra $\la_n$,
$\mu_n$  respectively. 
Let $Z_n = X_n + Y_n$.
For any $\eps > 0$,
$$\lim\limits_{n \ra \infty}\left(\frac{-2}{n^2}\right)\log \p_n\left[\spec(Z_n) \in {B}_\I^n(\nu_n, \eps)\right] = \inf\limits_{\partial^-\gamma'\in {B}_\I(\nu, \eps)}I(\g').$$
\end{thm}
\begin{proof}
From Lemma~\ref{lem:44-new} and Lemma~\ref{lem:55-new}, we have the following.
$$\lim\limits_{n \ra \infty}\p_n\left[\spec(Z_n) \in {B}_\I^n(\nu_n, \eps)\right]^\frac{2}{n^2} =$$ $$\sup\limits_{\partial^-\gamma' \in {B}_\I(\nu, \eps)}\sup\limits_{h' \in H(\la, \mu; \nu')} \left(V(\la)V(\mu)\right)^{-1}V(\nu')\exp\left(-\int_T 2\sigma((-1)(\hess h')_{ac})\Leb_2(dx)\right).$$
The theorem now follows.
\end{proof}
We will next prove the following. 

\begin{lemma}\lab{lem:upperI}
For all $\nu = \partial^- \g$, where $\g$ is Lipschitz and concave, and $\g(0) = \g(1) = 0,$
\beq   \lim\limits_{\eps \ra 0} \inf\limits_{\partial^- \g' \in B_\I(\nu, \eps)}  I(\gamma') = I(\g).\lab{eq:tpt}\eeq
\end{lemma}
\begin{proof}

Since $\partial^-\gamma \in B_\I(\nu, \eps),$ it suffices to show that 
\beq  \lim\limits_{\eps \ra 0} \inf\limits_{\partial^- \g' \in B_\I(\nu, \eps)}  I(\gamma') \geq I(\g).\lab{eq:tpt2}\eeq 
 Suppose that  $$\lim\limits_{\eps \ra 0} \inf\limits_{\partial^- \g' \in B_\I(\nu, \eps)}  I(\gamma') < + \infty.$$ Then, there exists a sequence of functions $\g^{(1)}, \g^{(2)}, \dots $ converging in $L^\infty[0, 1]$ to $\g$, such that %$I(\g^{(i)}) \geq I(\g^{(i+1)})$ for each $i$ and such that 
 $$\lim\limits_{k \ra \infty} I(\g^{(k)})  = \lim\limits_{\eps \ra 0} \inf\limits_{\partial^- \g' \in B_\I(\nu, \eps)}  I(\gamma') < \infty.$$ Let $M:= \lim\limits_{k \ra \infty} I(\g^{(k)}) .$ 
Then, by Lemma~\ref{lem:38.5}, $\J$ is upper semicontinuous on $H(\la, \mu).$

Let $(h^{(k)})_{k \geq 1}$ be a sequence where $h^{(k)} \in H(\la, \mu ; \partial^-\gamma^{(k)})$ is such that for all $k$,  %\ben \item  $I_1(h^{(k)}) \geq I_1(h^{(k+1)})$ and \item 

$I_1(h^{(k)}) - I(\gamma^{(k)}) < \frac{1}{k}.$
%\een
We then see that $$\lim_{k \ra \infty} I_1(h^{(k)})  = M.$$
Since $H(\la, \mu)$ is compact in $L^\infty(T)$ by Arzel\`{a}-Ascoli,  $(h^{(i)})_{i \geq 1}$ has a convergent subsequence $(h^{(i_k)})_{k \geq 1},$ whose limit $h$  will belong to $H(\la, \mu; \nu).$ By the upper semicontinuity of $\J$,  $\limsup_k \J(h^{(i_k)}) \leq \J(h).$ This shows that $\liminf_k I_1(h^{(k)}) \geq I_1(h).$ However $I_1(h)$ is greater or equal to $I(\gamma)$ since $h \in H(\la, \mu; \nu).$
If $$\lim\limits_{\eps \ra 0} \inf\limits_{\partial^- \g' \in B_\I(\nu, \eps)}  I(\gamma') = + \infty,$$ there is nothing to prove and we are done.
\end{proof}
\begin{thm}\lab{thm:5}
Let $X_n, Y_n$ be independent random Hermitian matrices from unitarily invariant distributions with spectra $\la_n$,
$\mu_n$  respectively. 
Let $Z_n = X_n + Y_n$.
Given a monotonically decreasing sequence of $n$ numbers $v_n$ whose sum is $0$, we define $\iota_n(v_n)$ to be the unique piecewise linear (the pieces being intervals of the form $(\frac{i-1}{n}, \frac{i}{n})$) concave function $f$ from $[0, 1]$ to $\R$ such
$f(i/n) - f((i-1)/n) = v_n(i)/n^2$ for each $i \in [n]$ , and $f(0) = 0.$
For any Borel set $E\subset L^\infty[0, 1]$, let $\PP_n(E) := \p_n[\iota_n(\spec( Z_n)) \in E].$
Let $a_n := \frac{n^2}{2}.$
For each Borel measurable set 
${\displaystyle E\subset L^\infty[0, 1],}$ 
$${\displaystyle -\inf_{\g\in E^{\circ }}I(\g)\leq \liminf_{n \ra \infty}a_{n}^{-1}\log(\mathcal {P}_{n}(E))\leq \limsup_{n \ra \infty}a_{n}^{-1}\log(\mathcal {P}_{n}(E))\leq -\inf_{\g\in {\overline {E}}}I(\g).}$$
\end{thm}
\begin{proof}
For each Borel measurable set $E$, by covering $E^o$ using $L^\infty$ balls, it follows from Theorem~\ref{thm:4} that  $$-\inf_{\g\in E^{\circ }}I(\g)\leq \liminf_{n \ra \infty}a_{n}^{-1}\log(\mathcal {P}_{n}(E)).$$
Next, suppose for the sake of contradiction, that $$\limsup_{n \ra \infty}a_{n}^{-1}\log(\mathcal {P}_{n}(E)) = -\inf_{\g\in {\overline {E}}}I(\g) + \eps'$$ for some positive $\eps'$. For all $\gamma' \in \overline {E}$, such that $I(\gamma') < + \infty$, let $\de_{\gamma'}$ denote a positive real such that $- \inf_{\gamma'' \in B_\infty(\gamma', \de_{\gamma'})} I(\gamma'') < - I(\gamma') + \eps'/2.$ The existence of such a positive real is guaranteed by Lemma~\ref{lem:upperI}. Since $\overline E \cap H(\la, \mu)$ (by Arzel\`{a}-Ascoli, in view of the Lipschitz norm being bounded and hence the family being equicontinuous) is compact, there is a finite subset $\La'$ of $\overline E \cap H(\la, \mu)$ such that the union of balls $ B_\infty(\gamma', \de_{\gamma'})$ over all $\gamma' \in \La'$ covers $\overline E \cap H(\la, \mu)$. We see that \beqs \limsup_{n \ra \infty}a_{n}^{-1}\log(\mathcal {P}_{n}(E)) & \leq & \limsup_{n \ra \infty}a_{n}^{-1}\log\left(\mathcal {P}_{n}\left(\bigcup_{\g' \in \La'} B_\infty(\gamma', \de_{\gamma'})\right)\right)\\ & = & - \inf_{\g' \in \La'}\left(\inf_{\g \in B_\infty(\gamma', \de_{\gamma'})}I(\gamma)\right)\\
& \leq & - \inf_{\g' \in \overline{E}}I(\gamma') + \eps'/2.\eeqs
This however contradicts $$-\inf_{\g\in \overline{E}}I(\g) + \eps' =  \limsup\limits_{n \ra \infty}a_{n}^{-1}\log(\mathcal {P}_{n}(E)).$$
Therefore we must have 
$$\limsup_{n \ra \infty}a_{n}^{-1}\log(\mathcal {P}_{n}(E)) \leq  -\inf_{\g\in {\overline {E}}}I(\g),$$ and we are done.
\end{proof}

\section{Open problems}
\ben 
\item Is $\sigma$ strictly convex? An affirmative answer would imply the existence,  as $n \ra \infty,$ of a limit shape for a randomly chosen augmented hive with $\a, \b$ being strongly concave $C^2$ functions.
\item Is there an explicit analytic expression for $\sigma$? Based on the work of Wangru Sun (Definition 5.4 \cite{Wangru}) on the Bead model, and  the work of Shlyakhtenko and Tao, (Theorem 1.7, \cite{Shlyakhtenko}), and the work of Samuel Johnston \cite{Johnston} on random Gelfand-Tsetlin patterns, it is natural to expect that $$\lim_{s_2 \ra \infty} \sigma(s_0, s_1, s_2) = - \ln\left(\frac{e(s_0 + s_1)}{\pi}\sin\left(\frac{\pi s_0}{s_0 + s_1}\right)\right).$$
\item The open question $6.1$ in \cite{KT2} asks whether there exists a concrete measure preserving map from  the space $\MM_{\la_n \mu_n \nu_n}$ (of Hermitian triples $(X_n, Y_n, Z_n)$ satisfying $X_n + Y_n = Z_n$ and having spectrum $(\la_n, \mu_n, \nu_n)$) to  $H_n(\la_n, \mu_n; \nu_n)$. If there exists such a map (concrete or not) for each $n$ that is Lipschitz (the Lipschitz constant not depending on $n$) with respect to the $\|\cdot\|_\I$ norm on the spectra in the domain and sup norm in the range, it is then possible, to remove the restriction that $\a$ and $\b$ be $C^2$ in Theorem~\ref{thm:5new}, and merely require that they be Lipschitz and strongly concave. Do such  maps exist?
\item Are there similar large deviations principles for integer valued random hives? Integer valued augmented hives are significant in the representation theory of $GL_n(\mathbb{C})$ \cite{KT2}, and are connected to the irreducible representations appearing in the direct sum decomposition of the tensor product of two irreducible representations. Integer valued hives are also known to be in bijective correspondence with square triangle tilings of certain regions in the plane \cite{squaretri}.
\een

\section*{Acknowledgements}
We are very grateful to the anonymous reviewer for an exceptionally painstaking and careful review that pointed out several inaccuracies.
We thank Terence Tao for his valuable comments.
Hariharan Narayanan  is partially supported by a Ramanujan fellowship and a Swarna Jayanti fellowship,  instituted by the Government of India.  Scott Sheffield is partially supported by NSF awards DMS 1712862 and DMS 2153742.
\bibliographystyle{alpha}

%\bibliography{hives}

\end{document}